\documentclass[12pt]{article}
\usepackage{amsmath}
\usepackage{latexsym}
\usepackage{amssymb}
\newtheorem{thm}{Theorem}[section]
\newtheorem{la}[thm]{Lemma}
\newtheorem{Defn}[thm]{Definition}
\newtheorem{Remark}[thm]{Remark}
\newtheorem{Note}[thm]{Note}
\newtheorem{prop}[thm]{Proposition}

\newtheorem{cor}[thm]{Corollary}
\newtheorem{Example}[thm]{Example}
\newtheorem{Examples}[thm]{Examples}
\newtheorem{Problems}[thm]{Problems}

\newtheorem{Problem}[thm]{Problem}
\newtheorem{Convention}[thm]{Convention}
\newtheorem{Number}[thm]{\!\!}
\newenvironment{defn}{\begin{Defn}\rm}{\end{Defn}}

\newenvironment{example}{\begin{Example}\rm}{\end{Example}}

\newenvironment{rem}{\begin{Remark}\rm}{\end{Remark}}
\newenvironment{numba}{\begin{Number}\rm}{\end{Number}}
\newenvironment{proof}{{\noindent\bf Proof.}}%
                  {\nopagebreak\hspace*{\fill}$\Box$\medskip\medskip\par}   
\newcommand{\Punkt}{\nopagebreak\hspace*{\fill}$\Box$}
\newcommand{\wb}{\overline}
\newcommand{\ve}{\varepsilon}

\newcommand{\wt}{\widetilde}
\newcommand{\tensor}{\otimes}
\newcommand{\impl}{\Rightarrow}

\newcommand{\mto}{\mapsto}

\newcommand{\N}{{\mathbb N}}
\newcommand{\R}{{\mathbb R}}
\newcommand{\bO}{{\mathbb O}}

\newcommand{\K}{{\mathbb K}}

\newcommand{\Q}{{\mathbb Q}}

\newcommand{\Z}{{\mathbb Z}}

\newcommand{\C}{{\mathbb C}}

\newcommand{\cU}{{\cal U}}
\newcommand{\cO}{{\cal O}}
\newcommand{\cV}{{\cal V}}
\newcommand{\cW}{{\cal W}}

\newcommand{\pl}{{\displaystyle \lim_{\longleftarrow}}}

\newcommand{\one}{{\bf 1}}
\newcommand{\sub}{\subseteq}

\DeclareMathOperator{\im}{im}

\DeclareMathOperator{\pr}{pr}

\DeclareMathOperator{\id}{id}

\newcommand{\cF}{{\cal F}}

\newcommand{\cS}{{\cal S}}
\newcommand{\cK}{{\cal K}}
\newcommand{\cL}{{\cal L}}
\newcommand{\cT}{{\cal T}}
\newcommand{\cX}{{\cal X}}

\DeclareMathOperator{\ev}{ev}

\newcommand{\sbull}{{\scriptscriptstyle \bullet}}

\DeclareMathOperator{\LocPol}{LocPol}
\DeclareMathOperator{\Pol}{Pol}

\newcommand{\aeq}{\Leftrightarrow}

\begin{document}
\renewcommand{\thefootnote}{\fnsymbol{footnote}}
$\;$\\[-27mm]
\begin{center}
{\Large\bf Exponential laws for ultrametric partially\\[2.4mm]
differentiable functions and applications}\\[6mm]
{\bf Helge Gl\"{o}ckner}\vspace{2.5mm}
\end{center}
\renewcommand{\thefootnote}{\arabic{footnote}}
\setcounter{footnote}{0}
\begin{abstract}\vspace{1mm}
\hspace*{-8.3 mm}
We establish exponential laws for certain spaces of
differentiable\linebreak
functions
over a valued field $\K$.
For example, we show that
\[
C^{(\alpha,\beta)}(U\times V,E) \, \cong\,  C^\alpha(U,C^\beta(V,E))
\]
if $\alpha\in (\N_0\cup\{\infty\})^n$,
$\beta\in (\N_0\cup\{\infty\})^m$,
$U\sub \K^n$ and $V\sub \K^m$ are open (or suitable more general)
subsets, and $E$ is a topological vector space.
As a first application, we study the density of locally polynomial functions in spaces of partially differentiable functions over an ultrametric field (thus solving an open problem by Enno Nagel), and also global approximations by polynomial functions.
As a second application, we obtain a new proof for the characterization of $C^r$-functions on $(\Z_p)^n$
in terms of the decay of their Mahler expansions.
In both applications, the exponential laws enable simple inductive proofs via a reduction to the one-dimensional, vector-valued case.\vspace{2.6mm}
\end{abstract}
{\footnotesize {\em Classification}:
26E20;
26E30
(primary)
12J10;
12J25;
26B05;
%
30G06;
32P05;
%
54D50\\[2mm]
%
%
{\em Key words}:
Exponential law, function space, partial differential, finite order, product, ultrametric field, topological field,
non-archimedean analysis, Mahler expansion, differentiability, metrizability, valued field, compactly generated space, k-space, polynomial, locally polynomial function, density, approximation, Stone-Weierstrass Theorem}\\[11mm]
{\Large\bf Introduction and statement of results}\\[4mm]
Let $\alpha=(\alpha_1,\ldots,\alpha_n)$ with $\alpha_1,\ldots,\alpha_n\in \N_0\cup\{\infty\}$.
A function $f\colon U\to \R$ on an open set $U\sub \R^n$
is called $C^\alpha$ if it admits continuous partial derivatives
$\partial^\beta f\colon U\to \R$
for all multi-indices $\beta\in (\N_0)^n$
such that $\beta\leq\alpha$
(cf.\ \cite{Alz}).
In this article, we study
an analogous notion of $C^\alpha$-map
$f\colon U\to E$ defined with the help of continuous extensions
to certain partial difference quotient maps,
for $E$ a topological vector space over a
topological field~$\K$ and open subset $U\sub \K^n$.
In fact, beyond open domains we consider
$C^\alpha$-functions on subsets $U\sub \K^n$ which are \emph{locally cartesian}
in the sense that each point in $U$
has a relatively open neighbourhood $V\sub U$
of the form $V=V_1\times\cdots\times V_n$
for subsets $V_1,\ldots, V_n\sub \K$ without isolated points (see Definition~\ref{defSDS}).
For $\K$ a complete ultrametric field and $E$ an ultrametric Banach space,
such functions have recently been introduced by E. Nagel \cite{Nag}.
Our (equivalent) definition differs from his in detail
and avoids the use of spaces of linear operators
(mimicking, instead, an approach to multi-variable $C^r$-maps
pursued by De Smedt \cite{DSm}, Schikhof \cite[\S84]{Sch}
and the author~\cite[Definition 1.13]{CMP}).\\[2.4mm]
We endow the space $C^\alpha(U,E)$
of all $E$-valued $C^\alpha$-maps on~$U$
with its natural vector topology (see Definition~\ref{coalpha}).
The main result of this article is the following exponential law.\\[3.5mm]
{\bf Theorem A.}
\emph{Let $\K$ be a topological field,
$n,m\in \N$,
$U\sub \K^n$ and $V\sub \K^m$ be locally cartesian subsets,
and $\alpha\in (\N_0\cup\{\infty\})^n$,
$\beta\in (\N_0\cup\{\infty\})^m$.
Then $f(x,\sbull)\in C^\beta(V,E)$
for each $f \in C^{(\alpha,\beta)}(U\times V,E)$,
the map}
\[
f^\vee\colon U \to C^\beta(V,E),\quad
x\mto f(x,\sbull)
\]
\emph{is $C^\alpha$, and the mapping}
\[
\Phi\colon
C^{(\alpha,\beta)}(U\times V,E)\to C^\alpha(U,C^\beta(V,E)),\quad
f \mto f^\vee
\]
\emph{is $\K$-linear and a topological embedding.
If $\K$ is metrizable} (\emph{e.g., if $\K$
is a valued field}), \emph{then $\Phi$ is an isomorphism
of topological $\K$-vector spaces.}\\[3.5mm]
Let $n\in \N$, $\alpha\in (\N_0\cup\{\infty\})^n$, $\K$ be a field
and $E$ be a $\K$-vector space.
A function $p \colon U \to E$ on a subset $U\sub \K^n$
is called a \emph{polynomial function of multidegree $\leq \alpha$}
if there exist $a_\beta\in E$ for multi-indices $\beta\in (\N_0)^n$
with $\beta\leq\alpha$ such that $a_\beta=0$
for all but finitely many $\beta$ and
\[
p(x)=\sum_{\beta\leq\alpha}x^\beta a_\beta\quad\mbox{for all $\,x=(x_1,\ldots, x_n)\in U$,}
\]
with $x^\beta:=x_1^{\beta_1}\cdot\ldots\cdot x_n^{\beta_n}$.
We write $\Pol(U,E)$ for the space of all $E$-valued polynomial functions
on~$U$.
If $(\K,|.|)$ is a valued field,
$U\sub \K^n$ a subset and
$E$ a topological $\K$-vector space, we say that a function
$f \colon U\to E$
is \emph{locally polynomial} of multidegree $\leq \alpha$
if each $x\in U$ has an open neighbourhood $V$ in~$U$
such that $f|_V=p$
for some polynomial function $p\colon \K^n\supseteq V \to E$
of multidegree $\leq \alpha$.
We write $\text{LocPol}_{\leq\alpha}(U,E)$
for the space of all locally polynomial $E$-valued
functions of multidegree $\leq\alpha$ on~$U$.
Using Theorem~A to perform a reduction
to the one-dimensional case
(settled for scalar-valued functions
as a special case of \cite[Proposition~II.40]{Nag}),
we obtain:\\[3.5mm]
{\bf Theorem B.}
\emph{For every  complete ultrametric field $\K$,
compact cartesian subset $U\sub \K^n$,
locally convex topological $\K$-vector space~$E$
and $\alpha\in (\N_0\cup\{\infty\})^n$,
the space $\LocPol_{\leq \alpha}(U,E)$
of $E$-valued locally polynomial functions of
multidegree $\leq\alpha$ is dense
in $C^\alpha(U,E)$.}\\[3.5mm]
For $E=\K$ and $n\geq 2$, this answers a question raised
by E. Nagel in the case of integer exponents
(he also considered fractional differentiability; the case of $\alpha\in [0,\infty[^n\setminus \N^n_0$
remains open).\footnote{Personal communication at the 12th International Conference on $p$-Adic Functional Analysis, Winnipeg, July 2012.} \\[2.4mm]
Likewise, the exponential law
can be used to reduce questions concerning
the Mahler expansions of functions of several variables
to the familiar case of single-variable functions.
%
%
We obtain the following result (the scalar-valued case of which was published earlier
in the independent work \cite{Nag},
based on a different strategy of proof which exploits topological tensor products):\footnote{Compare also
\cite[theorem on p.\,140]{DSm} for the case of scalar-valued $C^1$-functions on $\Z_p\times\Z_p$, established by direct calculations.}\\[3.5mm]
{\bf Theorem C.}
\emph{Let $E$ be a sequentially complete locally convex space over
$\Q_p$ and $f\colon (\Z_p)^n\to E$ be a continuous function, where
$n\in \N$. Let\vspace{-1mm}
\[
f(x)=\sum_{\nu\in \N_0^n}\left(
\begin{array}{c}
x\\
\nu
\end{array}
\right)a_\nu\vspace{-1mm}
\]
be the Mahler expansion of $f$ with the Mahler coefficients
$a_\nu\in E$.
Let $\alpha\in \N_0^n$.
Then $f$ is $C^\alpha$ if and only if}
\begin{equation}\label{decaycond}
q(a_\nu) \nu^\alpha\to 0\quad\mbox{as $\,|\nu|\to\infty$}
\end{equation}
\emph{for each continuous ultrametric seminorm $q$ on~$E$,
where $|\nu|:=\nu_1+\cdots+\nu_n$.
Given $r\in \N_0$, the map $f$ is $C^r$ if and only if}
\[
q(a_\nu) |\nu|^r \to 0\quad\mbox{as $\,|\nu|\to\infty$,}
\]
\emph{for each $q$ as before.}\vspace{3mm}\pagebreak

\noindent
Passage to the Mahler coefficients allows $C^\alpha(\Z_p^n,E)$
to be identified with a suitable weighted $E$-valued $c_0$-space,
whose weight is modelling the decay condition~(\ref{decaycond})
(see Proposition~\ref{isweightedf} for details).\\[2.4mm]
Here, we use that a function $f \colon \K^n\supseteq U\to E$
is $C^r$ if and only if it is $C^\alpha$ for all $\alpha\in \N_0^n$
such that $|\alpha|\leq r$.
To generalize this idea,
let us return to a topological field $\K$
and topological $\K$-vector space $E$.
Let $n_1,\ldots, n_\ell\in \N$,
$n:=n_1+\cdots+n_\ell$ and $\alpha\in (\N_0\cup\{\infty\})^\ell$.
Agreeing to consider $\K^n$ as the direct product
$\K^{n_1}\times\cdots\times \K^{n_\ell}$,
we say that a function $f\colon U\to E$
on a locally cartesian subset $U\sub \K^n$
is $C^\alpha$ if it is
$C^\beta$ for all $\beta\in \N_0^n$
such that $|\beta_j|\leq\alpha_j$ for all $j\in \{1,\ldots, \ell\}$,
where we wrote $\beta=(\beta_1,\ldots,\beta_\ell)$
according to the decomposition $\N_0^n=\N_0^{n_1}\times\cdots\times\N_0^{n_\ell}$.
Then again an exponential law holds.\\[4mm]
{\bf Theorem D.}
\emph{Let $\K$ be a Hausdorff topological field,
$n,m\in \N$
and $U\sub \K^n$ as well as $V\sub \K^m$ be locally cartesian subsets.
Fix decompositions $n=n_1+\cdots+ n_\ell$ and $m=m_1+\cdots+m_k$
with $n_1,\ldots, n_\ell,m_1,\ldots, m_k\in \N$.
Let $\alpha\in (\N_0\cup\{\infty\})^\ell$
and $\beta\in (\N_0\cup\{\infty\})^k$.
Then $f(x,\sbull)\in C^\beta(V,E)$
for each $f\in C^{(\alpha,\beta)}(U\times V,E)$,
the map}
\[
f^\vee\colon U \to C^\beta(V,E),\quad
x\mto f(x,\sbull)
\]
\emph{is $C^\alpha$, and the mapping}
\[
\Phi\colon
C^{(\alpha,\beta)}(U\times V,E)\to C^\alpha(U,C^\beta(V,E)),\quad
f \mto f^\vee
\]
\emph{is $\K$-linear and a topological embedding.
If $\K$ is metrizable} (\emph{e.g., if $\K$
is a valued field}), \emph{then $\Phi$ is an isomorphism
of topological $\K$-vector spaces.}\\[4mm]
As a consequence, reducing to the case of $C^r$-maps on compact cartesian sets treated
(in the scalar-valued case) in \cite[Proposition~II.40]{Nag},
we obtain the density of suitable
locally polynomial functions in $C^\alpha(U,E)$
for all locally closed, locally cartesian sets $U\sub \K^{n_1}\times\cdots\times \K^{n_\ell}$
and $\alpha\in \N_0^\ell$ (Proposition~\ref{betterdensloc}).
Generalizing the case of scalar-valued functions
on compact cartesian sets treated in \cite[Corollary II.42]{Nag},
we prove the density of polynomial functions
in many cases. Taking $\ell=1$ and $\ell=n$, respectively,
our Proposition~\ref{denseglobpo} subsumes:\\[4mm]
{\bf Theorem E.} \emph{Let $\K$ be a complete
ultrametric field,
$E$ be a locally convex topological $\K$-vector space,
$n\in \N$ and $U\sub \K^n$ be
a locally closed, locally cartesian subset.
Then the space $\Pol(U,E)$ of all $E$-valued polynomial functions on~$U$
is dense in $C^r(U,E)$, for each $r\in \N_0\cup\{\infty\}$.
Moreover,
$\Pol(U,E)$ is dense in $C^\alpha(U,E)$, for each $\alpha\in (\N_0\cup\{\infty\})^n$.}\\[4mm]
Every open subset $U\sub \K^n$
satisfies the hypotheses of Theorem~E,
and also every locally cartesian subset which is closed or locally compact.\\[2.4mm]
Polynomial approximations to functions of a single variable
were already studied in~\cite{ArS}.\\[2.4mm]
As a special case, taking $\ell=2$ the preceding approach subsumes
a notion of $C^{(r,s)}$-maps on locally cartesian subsets
$U\sub \K^n\times \K^m$.
For a corresponding notion of $C^{(r,s)}$-maps
on open subsets $U\sub \R^n\times\R^m$
(or, more generally, on open subset $U\sub E_1\times E_2$ with real locally convex spaces
$E_1$ and $E_2$),
the reader is referred to the earlier works \cite{AaS} and \cite{Alz},
where also exponential laws for the corresponding function spaces
(similar to Theorem~D) can be found.
Special cases (and variants) of such $C^{(r,s)}$-maps have been
used in many parts of analysis (see, e.g.,
\cite{Ama} for analogues of $C^{(0,r)}$-maps
on real Banach spaces based on continuous Fr\'{e}chet differentiability;
\cite[1.4]{GCX} for $C^{(0,r)}$-maps between real locally convex spaces;
\cite{DIF} for $C^{(r,s)}$-maps on finite-dimensional real domains; and
\cite[p.\,135]{FaK} for certain $\mbox{Lip}^{(r,s)}$-maps
in the convenient setting of ana\-lysis
(a backbone of which are exponential laws for spaces of smooth functions,
see also \cite{KaM}).
We mention that exponential laws for suitable spaces of \emph{smooth} functions
over locally compact topological fields (also in infinite dimensions)
were already established in \cite[Propositions 12.2 and 12.6]{ZOO}.
The possibility of exponential laws for $C^{r,s}$-maps
was conjectured there~\cite[p.\,10]{ZOO}.\\[3mm]
\emph{Acknowledgement and remark.}
The questions concerning exponential laws
for $C^{r,s}$-functions
and their application to Mahler expansions
were part of the DFG-project GL~357/8-1
(Project 3 in the proposal from October 2008).
%
%
%
%
%
%
%
%
%
%
%
%
%
%
\section{Preliminaries and notation}\label{secnot}
%
%
%
%
All topological fields occurring in this article
are assumed Hausdorff and non-discrete;
all topological vector spaces are assumed Hausdorff.
A~\emph{valued field} is a field~$\K$,
equipped with an absolute value
$|.|\colon \K\to [0,\infty[$
which defines a non-discrete topology on~$\K$.
If $|.|$ satisfies the ultrametric
inequality, we call $(\K,|.|)$
an \emph{ultrametric field}.
In this case, we call $\bO:=\{z\in \K\colon |z|\leq 1\}$
the \emph{valuation ring} of $\K$.
A topological vector space $E$ over
an ultrametric field $(\K,|.|)$
is called \emph{locally convex}
if the set of all open $\bO$-submodules of~$E$
is a basis for the filter of $0$-neighbourhoods in~$E$
(we then simply call $E$ a \emph{locally convex space}).
Equivalently, $E$ is locally convex
if its topology is defined
by a set of seminorms $q\colon E\to [0,\infty[$
which are ultrametric in the sense that
$q(x+y)\leq\max\{q(x),q(y)\}$ for all $x,y\in E$
(see \cite{Mon}, \cite{PSn} and \cite{vRo}
for further information on such spaces).
A valued field $(\K,|.|)$ is called \emph{complete} if
the metric defined via $d(x,y):=|x-y|$
makes $\K$ a complete metric space.
If $X$ is a topological space, $Y\sub X$ a subset and $U\sub Y$
a relatively open subset, we usually simply say that
``$U$ is open in $Y$'' or
``$U\sub Y$ is open''.
In the rare cases when
$U$ is intended to by open in~$X$,
we shall say so explicitly.
As usual, by a ``clopen'' set we mean a set which is both closed and open.
If $\K$ is a topological field, we always
endow $\K^n$ with the product topology.
A subset $Y$ of a topological space~$X$ is called \emph{locally closed}
if, for each $x\in Y$ and neighbourhood $U\sub Y$ (with respect to the induced topology)
there exists a neighbourhood $V\sub U$ of~$x$
which is closed in~$X$.
For example, every closed subset of a regular topological space (like $\K^n$)
is locally closed, and also every open subset.
Locally compact subsets of a topological space are locally closed
as well.
If $(X,d)$ is a metric space, $x\in X$ and $r>0$,
we write $B^d_r(x):=\{y\in X\colon d(x,y)<r\}$
and $\wb{B}^d_r(x):=\{y\in X\colon d(x,y)\leq r\}$
for the open and closed balls, respectively.
Given a non-empty subset $Y\sub X$,
we set $B_r^d(Y):=\{x\in X\colon (\exists y\in Y)\, d(x,y)<r\}
=\bigcup_{y\in Y}B^d_r(y)$.
If $(E,\|.\|)$ is a normed space over
a valued field $(\K,|.|)$,
$x\in\K$ and $r>0$,
we use the notation $B^E_r(x):=\{y\in E \colon |x-y|)<r\}$.
%
We write $\N=\{1,2,\ldots\}$
and $\N_0:=\N\cup\{0\}$.\\[2.4mm]
If $X$ is a set and
$(f_j)_{j\in J}$ a family of mappings $f_j\colon X\to X_j$ to topological spaces~$X_j$,
then the set of all pre-images $f_j^{-1}(U)$ (with $j\in J$ and $U$ open in~$X_j$)
is a subbasis for a topology~$\cO$ on~$X$, called the \emph{initial topology}
with respect to the family $(f_j)_{j\in J}$.
For any topological space~$Y$,
a map
\[
f \colon Y\to (X,\cO)
\]
is continuous
if and only if $f_j\circ f$ is continuous for each $j\in J$.
The topology $\cO$ is unchanged if we take $U$ only in a basis
of open subsets of~$X_j$, or in a subbasis. This readily implies
the following well-known fact
(the ``transitivity of initial topologies''),
which will be used frequently:
\begin{la}
Let $X$ be a topological space whose topology $\cO$
is initial with respect to a family $(f_j)_{j\in J}$ of mappings
$f_j\colon X\to X_j$ to topological spaces~$X_j$.
Assume that the topology on $X_j$ is initial with respect to a family
$(f_{i,j})_{i\in I_j}$ of mappings $f_{i,j}\colon X_j\to X_{i,j}$
to topological spaces~$X_{i,j}$.
Then $\cO$ is also the initial topology with respect to the
family $(f_{i,j}\circ f_j)_{j\in J,i\in I_j}$ of compositions
$f_{i,j}\circ f_j\colon X\to X_{i,j}$.\,\Punkt
\end{la}
A mapping $f\colon X\to Y$ between topological
spaces is called a \emph{topological embedding} (or simply: an embedding)
if it is a homeomorphism onto its image.
This holds if and only if $f$ is injective and the topology on~$X$
is initial with respect to~$f$.
\section{{\boldmath $C^\alpha$}-maps on subsets of {\boldmath $\K^n$}}
We define $C^\alpha$-maps and record
elementary properties of such maps and their domains
of definition.
The discussion follows the treatment of $C^k$-maps
in \cite[2.5--2.16]{CMP} so closely,
that some proofs can be replaced by references to~\cite{CMP}.
%
%
%
\begin{numba}
Throughout this section,
let $\K$ be a topological field,
$E$ be a topological $\K$-vector space,
$n\in \N$ and $U\sub \K^n$
be a subset
(where $\K^n$ is equipped with the product
topology).
\end{numba}
\begin{numba}
We say that $U$ \emph{cartesian}
if it is of the form $U=U_1\times\cdots\times U_n$,
for some subsets $U_1,\ldots, U_n$
of $\K$ without isolated points.\footnote{Compare \cite[Remark~2.16]{CMP}.}
If every point in $U$ has a relatively open neighbourhood
in $U$ which is cartesian,
then $U$ is called \emph{locally cartesian}.\footnote{In the case of a complete ultrametric
field, such a set is called a  ``locally cartesian set whose factors contain no isolated points''
in \cite{Nag}.}
\end{numba}
\begin{numba}\label{nicernbhd}
Observe that if $U\sub \K^n$ is a locally cartesian subset, then each
$x\in U$ has an open cartesian neighbourhood $W\sub U$
of the form $W=U\cap Q$, for some open \emph{cartesian} subset $Q\sub \K^n$
(which is sometimes useful).\\[2.4mm]
[Indeed, $x=(x_1,\ldots, x_n)\in U$
has a relatively open neighbourhood $V\sub U$ of the form
$V=V_1\times\cdots\times V_n$,
for some subsets $V_1,\ldots, V_n\sub \K$
without isolated points.
Thus, there is an open subset $Y\sub \K^n$
such that $U\cap Y=V_1\times\cdots \times V_n$.
If $Q_1,Q_2,\ldots, Q_n$ are open neighbourhoods of
$x_1,\ldots, x_n$ in $\K$,
respectively, such that $Q:=Q_1\times\cdots\times Q_n\sub Y$,
then
\begin{equation}\label{hlpone}
U\cap Q=U\cap Y\cap Q=V\cap Q=(Q_1\times V_1)\times\cdots\times (Q_n\cap V_n),
\end{equation}
where $V_1\cap Q_1$, $\ldots$, $V_n\cap Q_n$
do not have isolated points.]\\[2.4mm]
Note that,
if $P\sub U$ is a given neighbourhood of~$x$,
we can choose $Q$ so small that $U\cap Q\sub P$.
\end{numba}
\begin{numba}
As usual, for $i\in \{1,\ldots, n\}$
we set $e_i:=(0,\ldots,0,1,0,\ldots, 0)\in \K^n$,
with $i$-th entry~$1$.
Given a ``multi-index'' $\alpha=(\alpha_1,\ldots,\alpha_n)
\in \N_0^n$,
we write $|\alpha|:=\sum_{i=1}^n \alpha_i$
and $x^\alpha:=x_1^{\alpha_1}\cdots x_n^{\alpha_n}$ for $x=(x_1,\ldots, x_n)\in \K^n$.
If $\alpha,\beta\in (\N_0\cup\{\infty\})^n$,
we say that $\alpha\leq \beta$ if $\alpha_j\leq\beta_j$ for all $j\in\{1,\ldots, n\}$.
We define $\alpha+\beta$ component-wise, with $r+\infty=\infty+ r:=\infty$
for all $r\in \N_0\cup\{\infty\}$.
\end{numba}
%
%
\begin{numba}\label{notconv}
For $\alpha\in (\N_0\cup\{\infty\})^n$, our definition of a $C^\alpha$-map
$f\colon U\to E$ on a cartesian set~$U$
will involve a certain continuous extension
$f^{<\beta>}$ to an iterated partial difference
quotient map $f^{>\beta<}$ corresponding to
each multi-index $\beta\in \N_0^n$
such that $\beta \leq \alpha$.
It is convenient to define
the domains $U^{<\beta>}$ and $U^{>\beta<}$
of these mappings first.
They will be subsets
of $\K^{n+|\beta|}$.
It is useful to write elements
$x\in \K^{n+|\beta|}$ in the form
$x=(x^{(1)}, x^{(2)},\ldots, x^{(n)})$,
where $x^{(i)}\in \K^{1+\beta_i}$
for $i\in \{1,\ldots, n\}$.
We write $x^{(i)}=
(x^{(i)}_0,x^{(i)}_1,\ldots, x^{(i)}_{\beta_i})$
with $x^{(i)}_j\in \K$
for $j\in \{0,\ldots,\beta_i\}$.
\end{numba}
\begin{numba}
Let $U\sub \K^n$ be locally cartesian.
Given $\beta\in \N_0^n$, we define
$U^{<\beta>}$ as the set of all
$x\in \K^{n+|\beta|}$
such that, for all
$i_1\in \{0,1,\ldots,\beta_1\},\ldots,
i_n \in \{0, 1,\ldots,\beta_n\}$,
we have
\[
(x^{(1)}_{i_1},\ldots,
x^{(n)}_{i_n})\in U\, .
\]
We let $U^{>\beta<}$ be the set of all
$x\in U^{<\beta>}$ such that,
for all $i\in \{1,\ldots, n\}$
and $0\leq j<k\leq \beta_i$,
we have $x^{(i)}_j\not=x^{(i)}_k$.
\end{numba}
%
%
%
\begin{example}\label{mimpex}
If $U\sub \K^n$ is cartesian,
say $U=U_1\times\cdots\times U_n$
with subsets $U_i\sub \K$,
then simply
%
\begin{equation}\label{spcas}
U^{<\beta>}\; =\; U_1^{1+\alpha_1}\times
U_2^{1+\beta_2}\times\cdots
\times U_n^{1+\beta_n}\,.
\end{equation}
\end{example}
\begin{numba}
It is easy to see that $U^{>\beta<}$
is an open subset of
$U^{<\beta>}$.
If $U$ is cartesian, it is also easy to check
that $U^{>\beta<}$ is dense in $U^{<\beta>}$.
\end{numba}
\begin{rem}
Unfortunately,
$U^{>\beta<}$ need \emph{not} be dense in $U^{<\beta>}$
if $U\sub \K^n$ is merely locally cartesian.
For this reason,
we shall define $C^\alpha$-maps
by a reduction to the case
of cartesian subsets.
Also the definition of the topology on function spaces
is made slightly more complicated by the fact
that continuous extensions $f^{<\beta>}$
cannot be defined globally in general (only after restriction of $f$
to cartesian subsets).
\end{rem}
\begin{example}\label{exa1}
Let $p\geq 3$ be a prime and
$V_2$ and $W_2$ be compact, open, non-empty, disjoint subsets of $\Q_p$.
Pick $v\in V_2$, $w\in W_2$.
Let
\[
V_1:=\left\{\sum_{k=0}^\infty a_k p^k\colon (a_k)_{k\in \N_0}\in \{0,1\}^{\N_0}\right\}
\]
and
$W_1:=\{\sum_{k=0}^\infty a_k p^k\colon (a_k)_{k\in \N_0}\in \{0,2\}^{\N_0}\}$.
Then $V_1$ and $W_1$ are compact subsets of $\Q_p$
without isolated points, and thus
\[
U:=(V_1\times V_2)\cup (W_1\times W_2)
\]
is a compact, locally cartesian subset of $\Q_p \times \Q_p$.
Note that $(0,0,v,w)
\in U^{<(1,1)>}$ (because $(0,v),(0,w)\in U$),
and that
\[
Q:=(\Q_p\times \Q_p\times V_2 \times W_2)\cap U^{<1,1>}
\]
is an open neighbourhood of $(0,0,v,w)$ in $U^{<1,1>}$.
One can show that $Q\cap U^{>(1,1)<}=\emptyset$.
Hence $(0,0,v,w)$ is not in the closure of $U^{>(1,1)<}$.\\[2.4mm]
[For the proof, we use that
$(V_1\times V_2)\cap (W_1\times W_2)=\emptyset$
and $V_1\cap W_1=\{0\}$.
Suppose there exists $(x,y,v_2,w_2)\in
Q\cap U^{>(1,1)<}$.
Then $x\not= y$, $v_2 \not=w_2$
and $v_2\in V_2$, $w_2\in W_2$.
Note that $(x,v_2),(x,w_2)\in U$
forces $x\in V_1\cap W_1=\{0\}$.
Likewise, $(y,v_2),(y,w_2)\in U$
entails that $y=0$. Hence $x=y$,
contradiction.
Hence $(x,y,v_2,w_2)$ cannot exist.]
\end{example}
\begin{rem}
A simple induction on $|\beta|$ shows that
the sets $U^{<\beta>}$ can be defined
alternatively by recursion on $|\beta|$,
as follows: Set $U^{<0>}:=U$.
Given $\beta\in \N_0^n$ such that
$|\beta|\geq 1$, pick $\gamma\in \N_0^n$
such that $\beta=\gamma+e_i$
for some $i\in \{1,\ldots, n\}$.
Then
$U^{<\beta>}$
is the set of all elements $x\in \K^{n+|\beta|}$
such that\linebreak
$(x^{(1)},\ldots, x^{(i-1)},
x^{(i)}_0,x^{(i)}_1,\ldots,x^{(i)}_{\beta_i-1},
x^{(i+1)},\ldots, x^{(n)})\in U^{<\gamma>}$
holds as well as\linebreak
$(x^{(1)},\ldots, x^{(i-1)},
x^{(i)}_{\beta_i},x^{(i)}_1,\ldots,x^{(i)}_{\beta_i-1},
x^{(i+1)},\ldots, x^{(n)})\in U^{<\gamma>}$.
\end{rem}
We now define certain mappings
$f^{>\beta<}\colon U^{>\beta<}\to E$ and show afterwards
that they can be interpreted
as partial difference
quotient maps.
\begin{defn}
We set $f^{>0<}:=f$.
Given a multi-index $\beta\in \N_0^n$
such that $|\beta|\geq 1$,
we define $f^{>\beta<}(x)$ as the sum
%
\begin{equation}\label{badbd}
\sum_{j_1=0}^{\beta_1}\cdots
\sum_{j_n=0}^{\beta_n}
\left(\prod_{k_1\not=j_1} \frac{1}{x^{(1)}_{j_1}-x^{(1)}_{k_1}}
\cdot\ldots\cdot
\prod_{k_n\not=j_n}\frac{1}{x^{(n)}_{j_n}-x^{(n)}_{k_n}}\right)
f(x^{(1)}_{j_1},\ldots, x^{(n)}_{j_n})
\end{equation}
for $x\in U^{>\beta<}$,
using the notational conventions
from {\bf\ref{notconv}}.
The products are taken over all
$k_\ell\in \{0,\ldots, \beta_\ell\}$
such that $k_\ell\not=j_\ell$,
for $\ell\in \{1,\ldots, n\}$
(and empty products are defined as the element $1\in\K$).
\end{defn}
The map $f^{>\beta<}$ has
important symmetry properties.
%
%
\begin{la}\label{presymm}
Assume that $\beta\in \N_0^n$,
$i\in \{1,\ldots, n\}$
and $\pi$ is a permutation of\linebreak
$\{0,1,\ldots, \beta_i\}$.
Then
$(x^{(1)},\ldots, x^{(i-1)},
x^{(i)}_{\pi(0)},\ldots, x^{(i)}_{\pi(\beta_i)},
x^{(i+1)},\ldots, x^{(n)})\in U^{>\beta<}$
for each $x\in U^{>\beta<}$,
and
%
\begin{equation}\label{preeqsmm}
f^{>\beta<}(x^{(1)},\ldots, x^{(i-1)},
x^{(i)}_{\pi(0)},\ldots, x^{(i)}_{\pi(\beta_i)},
x^{(i+1)},\ldots, x^{(n)})\;=\; f^{>\beta<}(x)\, .
\end{equation}
\end{la}
\begin{proof}
The proof of
\cite[Lemma~2.11]{CMP}
can be repeated verbatim.
\end{proof}
The following lemma
shows that $f^{>\beta<}$ can indeed be
interpreted as a partial difference quotient map.
%
%
\begin{la}\label{interpr}
For each $i\in \{1,\ldots, d\}$ and $x\in U^{>e_i<}$,
the element $f^{>e_i<}(x)$ is given by
\[
\frac{f(x^{(1)},\ldots, x^{(i-1)},x^{(i)}_0, x^{(i+1)},\ldots, x^{(n)})
-
f(x^{(1)},\ldots, x^{(i-1)},x^{(i)}_1,
x^{(i+1)},\ldots, x^{(n)})}{x^{(i)}_0-x^{(i)}_1}.
\]
If $\beta\in \N_0^n$ such that $|\beta|\geq 2$,
let
$\gamma\in \N_0^n$ be a multi-index such that
$\beta=\gamma+e_i$
for some
$i\in \{1,\ldots,n\}$. Then $f^{>\beta<}(x)$ is given by
%
\begin{eqnarray}
\hspace*{-5mm}\lefteqn{\frac{1}{x^{(i)}_0-x^{(i)}_{\beta_i}}\cdot
\Big(f^{>\gamma<}(x^{(1)},\ldots, x^{(i-1)},
x^{(i)}_0, x^{(i)}_1, \ldots, x^{(i)}_{\beta_i-1},
x^{(i+1)},\ldots, x^{(n)})}\quad\quad\quad \notag\\
& & \;\;\; -
f^{>\gamma<}(x^{(1)},\ldots, x^{(i-1)},
x^{(i)}_{\gamma_i}, x^{(i)}_1, \ldots, x^{(i)}_{\beta_i-1},
x^{(i+1)},\ldots, x^{(n)})\Big)
\label{schlmm}
\end{eqnarray}
for all $x\in U^{>\beta<}$.
\end{la}
\begin{proof}
The proof of \cite[Lemma~2.12]{CMP}
carries over verbatim.
\end{proof}
%
%
\begin{defn}\label{defSDS}
Let $\alpha\in (\N_0\cup\{\infty\})^n$,
$E$ be a topological $\K$-vector space and
$U\sub \K^n$ be a subset.
\begin{itemize}
\item[(a)]
If $U$ is cartesian,
we say that a function $f\colon U\to E$ is $C^\alpha$
if the mapping $f^{>\beta<}\colon U^{>\beta<}\to E$
admits a continuous extension
$f^{<\beta>}\colon U^{<\beta>}\to E$,
for each $\beta\in \N_0^n$
such that $\beta\leq \alpha$.
\item[(b)]
If $U$ is locally cartesian,
we say that $f$ is $C^\alpha$
if $f|_V$ is $C^\alpha$, for each cartesian relatively open subset
$V\sub U$.
\end{itemize}
\end{defn}
\begin{rem}
Since $U^{>\beta<}$ is dense in
$U^{<\beta>}$, the continuous extension
$f^{<\beta>}$ of~$f^{>\beta<}$ in
Definition~\ref{defSDS}\,(a)
is unique whenever it exists.
\end{rem}
\begin{rem}
Assume that $U$ is cartesian and $f\colon U\to E$
is $C^\alpha$ in the sense
of Definition~\ref{defSDS}\,(a).
If $V\sub U$ is a cartesian open subset,
then $f^{<\beta>}|_{V^{<\beta>}}$ provides a continuous
extension for $(f|_V)^{>\beta<}$.
Hence $f$ is $C^\alpha$ also in the sense
of Definition~\ref{defSDS}\,(b),
with $(f|_V)^{<\beta>}=f^{<\beta>}|_{V^{<\beta>}}$.
\end{rem}
\begin{rem}\label{alphacts}
If $f\colon U\to E$ is $C^\alpha$, then $f$ is continuous,\\[2.4mm]
[In fact, $f|_V=(f|_V)^{<0>}$ is continuous for each cartesian
open subset $V\sub U$. As the latter form an open
cover of~$U$, the assertion follows.]
\end{rem}
\begin{rem}
If $f\colon U\to E$ is $C^\alpha$,
then $f$ is also $C^\beta$ for
each $\beta\in (\N_0\cup\{\infty\})^n$ such that $\beta\leq\alpha$.
This is immediate from the definition.
\end{rem}
We readily deduce from Lemma~\ref{presymm}:
%
%
\begin{la}\label{partsymm}
Let $f\colon U\to E$ be a $C^\alpha$-map
on a cartesian subset $U\sub\K^n$,
$\beta\in \N_0^n$ with $\beta\leq\alpha$,
$i\in \{1,\ldots, n\}$
and $\pi$ be a permutation of
$\{0,1,\ldots, \beta_i\}$.
Then
%
\begin{equation}\label{nweq}
(x^{(1)},\ldots, x^{(i-1)},
x^{(i)}_{\pi(0)},\ldots, x^{(i)}_{\pi(\beta_i)},
x^{(i+1)},\ldots, x^{(n)})\; \in \; U^{<\beta>}
\end{equation}
for each $x\in U^{<\beta>}$, and
%
\begin{equation}\label{eqsmm}
f^{<\beta>}(x^{(1)},\ldots, x^{(i-1)},
x^{(i)}_{\pi(0)},\ldots, x^{(i)}_{\pi(\beta_i)},
x^{(i+1)},\ldots, x^{(n)})\;=\; f^{<\beta>}(x)\, .
\end{equation}
\end{la}
\begin{proof}
The proof of \cite[Lemma~2.14]{CMP}
can be repeated verbatim.
\end{proof}
The following variant of our Lemma~\ref{interpr}
is available for~$f^{<\beta>}$:
%
%
\begin{la}\label{extmore}
Let $f\colon U\to E$ be a $C^\alpha$-map
on a cartesian subset $U\sub \K^n$
and $\beta,\gamma \in \N_0^n$ such that
$\beta=\gamma+e_i$ for some
$i\in \{1,\ldots, n\}$, and $\beta\leq\alpha$.
Then $f^{<\beta>}(x)$ is given by
%
%
\begin{eqnarray}
\hspace*{-5mm}
\lefteqn{\frac{1}{x^{(i)}_0-x^{(i)}_{\beta_i}}
\cdot \Big(f^{<\gamma>}(x^{(1)},\ldots, x^{(i-1)},
x^{(i)}_0, x^{(i)}_1, \ldots, x^{(i)}_{\beta_i-1},
x^{(i+1)},\ldots, x^{(n)})}\qquad\quad \notag\\
& & \;\;\; -
f^{<\gamma>}(x^{(1)},\ldots, x^{(i-1)},
x^{(i)}_{\beta_i}, x^{(i)}_1, \ldots, x^{(i)}_{\beta_i-1},
x^{(i+1)},\ldots, x^{(n)})\Big)\label{schlmm2}
\end{eqnarray}
for all $x\in U^{<\beta>}$
such that $x^{(i)}_0\not=x^{(i)}_{\beta_i}$.
\end{la}
\begin{proof}
The proof of \cite[Lemma~2.15]{CMP}
can be repeated verbatim.
\end{proof}
\begin{rem}
If $U\sub \K^n$ is locally cartesian
and $P\sub U$ an open subset,
then also $P$ is locally cartesian.\\[2.4mm]
[Proof.
Let $x\in U$. We have seen at the end of \ref{nicernbhd}
that there exists an open cartesian neighbourhood $U\cap Q$
of~$x$ in~$U$ that is contained in~$P$.]
\end{rem}
\begin{rem}\label{passtosub}
Let $U\sub \K^n$ be locally cartesian
and $W\sub U$ be an open subset.
Since every cartesian open subset of~$W$ is also
a cartesian open subset of~$U$,
it is clear that $f|_W$ is $C^\alpha$
for each $C^\alpha$-map $f\colon U\to E$.
\end{rem}
Frequently, differentiability
properties that are defined via global continuous extensions
of difference quotient maps are, nonetheless,
local properties.\footnote{See \cite[Lemma~4.9]{BGN}
for the paradigmatic case of $C^r$-functions on open sets,
\cite[Remark~II.4.\,(i) and Prop.~II.5]{Nag}
for $C^r$-functions on locally cartesian sets in Nagel's setting.}
Also being $C^\alpha$ is a local property.
\begin{la}\label{islocal}
Let $f\colon U\to E$ be a mapping on a locally cartesian subset
$U\sub \K^n$.
Assume that each $x\in U$ has an open cartesian neighbourhood $W$
in $U$ such that $f|_W$ is $C^\alpha$.
Then $f$ is $C^\alpha$.
\end{la}
\begin{proof}
The proof is presented in a form which can be re-used later.
We show that $f|_V$
is $C^\alpha$ for each open cartesian subset $V\sub U$.

\emph{Construction of $(f|_V)^{<\beta>}$.}
We show that $(f|_V)^{>\beta<}$ has a continuous
extension $(f|_V)^{<\beta>}$, for each $\beta\in \N_0^n$ such that $\beta\leq\alpha$.
The proof is by induction on~$|\beta|$.
If $\beta=0$ and each point has a neighbourhood on which $f$
is continuous, then $(f|_V)^{<0>}=f|_V$ is continuous.
Now assume that $|\beta|\geq 1$.
Set $I_\beta:=\{i\in \{1,\ldots, n\}\colon \beta_i>0\}$.
For $i\in I_\beta$
and $j<k$ in $\{0,1,\ldots, \beta_i\}$,
define
\[
D_{\beta,i,j.k}:=\{x\in
V^{<\beta>}
\colon x^{(i)}_j\not=x^{(i)}_k\}\,.
\]
Then $V^{>\beta<}\sub D_{\beta,i,j,k}$, and
$D_{\beta,i,j.k}$ is an open and
dense set in $V^{<\beta>}$.
Let $\beta_i':=\beta-e_i$.
By induction, continuous maps
$(f|_V)^{<\gamma>}\colon
V^{<\gamma>}\to E$ extending $(f|_V)^{>\gamma<}$
exist for
all $\gamma\in \N_0^n$ such that $\gamma\leq\alpha$
and $|\gamma|<|\beta|$.
In particular, $(f|_V)^{<\beta_i'>}$ exists,
enabling us to define
continuous maps $f_{\beta,i,j,k}\colon D_{\beta,i,j,k}\to E$
sending $x\in D_{\beta,i,j,k}$ to {\small
\begin{eqnarray}
\hspace*{-5mm}\lefteqn{\frac{1}{x^{(i)}_j-x^{(i)}_k}
\big((f|_V)^{<\gamma>}(x^{(1)},\ldots, x^{(i-1)},
x^{(i)}_0, \ldots, x^{(i)}_{k-1},x^{(i)}_{k+1}, \ldots, x^{(i)}_{\gamma_i},
x^{(i+1)}\!,\ldots, x^{(n)})}\qquad\quad \notag\\
& & \hspace*{-15mm}-
(f|_V)^{<\gamma>}(x^{(1)},\ldots, x^{(i-1)},
x^{(i)}_0, \ldots, x^{(i)}_{j-1},x^{(i)}_{j+1}, \ldots, x^{(i)}_{\gamma_i},
x^{(i+1)},\ldots, x^{(n)})\big),\label{varschlmm}
\end{eqnarray}}
\hspace*{-1mm}where we abbreviated $\gamma:=(\gamma_1,\ldots,\gamma_n):=\beta'_i$.
Let $D_\beta:=\bigcup_{i,j,k} D_{\beta,i,j,k}$.
Then $V^{<\beta>}\setminus D_\beta$
is the set of all $x(t)\in \K^{n+|\beta|}$
of the form $x(t)=(x^{(1)},\ldots, x^{(n)})$
with $x^{(i)}=(t_i,\ldots, t_i)$ for $i\in \{1,\ldots, n\}$,
with $t=(t_1,\ldots, t_n)\in V$.
Let $W_t\sub V$ be an open cartesian neighbourhood
of $t$ such that $f|_{W_t}$ is $C^\alpha$
(cf.\ \ref{nicernbhd} and Remark~\ref{passtosub}).
Since $W_t$ is open in~$V$, its $i$-th component $W_{t,i}\sub\K$ is open in the $i$-th component
$V_i$ of the cartesian set~$V$,
for $i\in \{1,\ldots, n\}$.~Now
\[
V^{<\beta>}=D_\beta \cup \bigcup_{t\in V} (W_t)^{<\beta>},
\]
where $D_\beta$ and each $W_t^{<\beta>}=W_{t,1}^{1+\beta_1}\times\cdots\times
W_{t,n}^{1+\beta_n}$
is open in $V^{<\beta>}=V_1^{1+\beta_1}\times\cdots\times
V_n^{1+\beta_n}$.
We define
$(f|_V)^{<\beta>}\colon V^{<\beta>}\to E$ via
\begin{equation}\label{nfortop}
(f|_V)^{<\beta>}(x):=
\left\{
\begin{array}{cc}
f_{\beta,i,j,k}(x) & \mbox{if $x\in D_{\beta,i,j,k}$\,;}\\
(f|_{W_t})^{<\beta>}(x) & \mbox{if $x\in W_t^{<\beta>}$.}
\end{array}
\right.
\end{equation}
To see that this mapping is well defined, let $g$ and $h$ be two
of the maps used in the piecewise
definition.
Let $G$ and $H$ be the domain of $h$ and $g$, respectively.
Then $G$ and $H$ are open in $V^{<\alpha>}$,
hence also $G\cap H$.
Thus $V^{>\beta<}\cap (G\cap H)$ is dense
in $G\cap H$. Since $g$ and $h$ are continuous and coincide
on $V^{>\beta<}\cap (G\cap H)$
(where they coincide with $(f|_V)^{>\beta<}$,
using Lemmas~\ref{presymm} and \ref{interpr}),
it follows that $g|_{G\cap H}= h|_{G\cap H}$.\\[2.4mm]
Next, because each $g\colon G\to E$ coincides
on $G\cap V^{>\beta<}$ with $(f|_V)^{>\beta<}$,
we see that $(f|_V)^{<\beta>}|_{V^{>\beta<}}=
(f|_V)^{>\beta<}$.
This completes the proof.
\end{proof}
%
%
%
\begin{rem}
Assume that $U\sub \K^n$ is open.
Then $U$ is locally cartesian
and $U^{>\beta<}$ is dense in
$U^{<\beta>}$
for each $\beta\in \N_0^n$
(as is easy to see).
Moreover, if $f\colon U\to E$
is $C^\alpha$, then
$f^{>\beta<}$ admits a continuous extension
$f^{<\beta>}\colon
U^{<\beta>}\to E$ to all of $U^{<\beta>}$,
for each $\beta\in \N_0^n$
such that $\beta\leq\alpha$
(let $V:=U$ and construct continuous maps
$f^{<\beta>}=(f|_V)^{<\beta>}$
verbatim as in the proof of Lemma~\ref{islocal}).
\end{rem}
We record a simple version of the Chain Rule.
\begin{la}\label{linchain}
Let $f\colon \K^n\supseteq U\to E$ be a $C^\alpha$-map 
and $\lambda\colon E\to F$ be a continuous linear map between topological $\K$-vector spaces.
Then $\lambda\circ f$ is $C^\alpha$, and
\begin{equation}\label{eqlincha}
(\lambda\circ f|_V)^{<\beta>}=
\lambda\circ (f|_V)^{<\beta>}
\end{equation}
holds for
each $\beta\in \N_0^n$ such that $\beta\leq\alpha$ and
each open cartesian subset $V\sub U$.
\end{la}
\begin{proof}
The linearity of $\lambda$ implies that
$(\lambda\circ f|_V)^{>\beta<}=
\lambda\circ (f|_V)^{>\beta<}$.\linebreak
Because the right-hand side admits the continuous extension
$\lambda\circ (f|_V)^{<\beta>}$,
the assertions follow.
\end{proof}
The next three lemmas are analogues of
\cite[Lemmas 10.1--10.3]{BGN}.
\begin{la}\label{inprodu}
Let $(E_j)_{j\in J}$ be a family of topological $\K$-vector spaces
and $E:=\prod_{j\in J}E_j$ be the direct product.
Let $\pr_j\colon E\to E_j$ be the canonical projections, 
$\alpha\in (\N_0\cup\{\infty\})^n$
and $U\sub \K^n$ be a locally cartesian subset.
Then a map $f=(f_j)_{j\in J}\colon U\to E$ is $C^\alpha$
if and only if all of its components~$f_j$ are~$C^\alpha$.
In this case,
\begin{equation}\label{difqprod}
(f|_V)^{<\beta>}=((f_j|_V)^{<\beta>})_{j\in J}
\end{equation}
for all open cartesian subsets $V\sub U$ and
all $\beta\in \N_0^n$ such that $\beta\leq\alpha$.
\end{la}
\begin{proof}
If $f$ is $C^\alpha$, then also $f_j=\pr_j\circ f$ is
$C^\alpha$, by Lemma~\ref{linchain}.
Conversely, assume that each $f_j$ is $C^\alpha$.
For $V$ and $\beta$ as above,
we have $(f|_V)^{>\beta<}=((f_j|_V)^{>\beta<})_{j\in J}$.
Because $((f_j|_V)^{<\beta>})_{j\in J}$
is a continuous extension for the right-hand side, we deduce
that $f$ is $C^\alpha$ and (\ref{difqprod}) holds.
\end{proof}
\begin{la}\label{difinsub}
Let $E$ be a topological $\K$-vector space,
$E_0\sub E$ be a closed vector subspace,
$U\sub \K^n$ be a locally cartesian subset
and $f\colon U\to E_0$ be a map.
Then $f$ is $C^\alpha$ as a map to~$E_0$
if and only if $f$ is $C^\alpha$ as a map to~$E$.
\end{la}
\begin{proof}
The inclusion map $\iota\colon E_0\to E$ is continuous and linear.
Hence, if $f\colon U\to E_0$ is $C^\alpha$, then also
$\iota\circ f\colon U\to E$ is $C^\alpha$, by Lemma~\ref{linchain}.
Conversely, assume that $\iota\circ f$ is $C^\alpha$.
Let $V\sub U$ be an open cartesian set and $\beta\in \N_0^n$ with $\beta\leq\alpha$.
Let $x\in V^{<\beta>}$.
Since $V^{>\beta<}$
is dense in $V^{<\beta>}$,
there is a net $(x_a)_{a\in A}$ in
$V^{>\beta<}$ such that $x_a\to x$.
Then $(\iota\circ f|_V)^{<\beta>}(x)=\lim
(\iota\circ f|_V)^{<\beta>}(x_a)$,
where
$(\iota\circ f|_V)^{<\beta>}(x_a)=
(\iota\circ f|_V)^{>\beta<}(x_a)=
(f|_V)^{>\beta<}(x_a)\in E_0$
for each $a\in A$. Since $E_0$ is closed,
we deduce that $(\iota\circ f|_V)^{<\beta>}(x)\in E_0$.
Thus $(\iota\circ f|_V)^{<\beta>}\colon V^{<\beta>}\to E_0$
is a continuous extension to $(f|_V)^{>\beta<}$
and we deduce that~$f$ is~$C^\alpha$.
\end{proof}
If $\K$ is metrizable,
then it suffices to assume that~$E_0$ is sequentially closed
in Lemma~\ref{difinsub}
(because the net $(x_a)_{a\in A}$
can be replaced with a sequence then).
\begin{la}\label{plfirst}
Let $(J,\leq )$ be a directed set,
$\cS:=((E_j)_{j\in J},(\phi_{i,j})_{i\leq j})$ be a projective
system of topological $\K$-vector spaces~$E_j$,
with continuous linear maps $\phi_{i,j}\colon E_j\to E_i$ satisfying $\phi_{i,j}\circ\phi_{j,k}=\phi_{i,k}$
for $i\leq j\leq k$ in~$J$, and $\phi_{i,i}=\id_{E_i}$.
Let $(E,(\phi_i)_{i\in I})$ be a projective limit of
$\cS$, with the continuous linear maps $\phi_j\colon E\to E_j$
such that $\phi_{j,k}\circ \phi_k=\phi_j$.
Then a map $f\colon \K^n\supseteq U\to E$ is $C^\alpha$
if and only if $\phi_j\circ f$ is $C^\alpha$ for each $j\in J$.
\end{la}
\begin{proof}
After applying an isomorphism of topological vector spaces,
we may assume that $E$ is realized in the usual form
as a closed vector subspace of the direct product
$\prod_{j\in J}E_j$, and $\phi_j=\pr_j|_E$.
The assertion now follows from Lemmas~\ref{inprodu}
and \ref{difinsub}.
\end{proof}
\section{The compact-open {\boldmath$C^\alpha$}-topology}\label{secthetop}
Throughout this section,
$\K$ is a topological field, $E$ a topological $\K$-vector space,
$n\in \N$, $U\sub \K^n$ a locally cartesian subset and
$\alpha\in (\N_0\cup\{\infty\})^n$.
\begin{defn}\label{coalpha}
We endow that space $C^\alpha(U,E)$ of all $E$-valued $C^\alpha$-maps on~$U$
with the initial topology $\cO$ with respect to the mappings
\begin{equation}\label{pamps}
\Delta_{\beta,V} \colon C^\alpha(U,E)\to C(V^{<\beta>},E),\quad f\mto (f|_V)^{<\beta>},
\end{equation}
for all $\beta\in \N_0^n$ with $\beta\leq\alpha$
and open cartesian subsets $V\sub U$;
the spaces $C(V^{<\beta>},E)$ on the right-hand side
are endowed with the compact-open topology.
We call $\cO$ the \emph{compact-open $C^\alpha$-topology}.
\end{defn}
\begin{rem}\label{compaco}
\begin{itemize}
\item[(a)]
Since each $C(V^{<\beta>},E)$ is a topological vector space and the mappings in
(\ref{pamps}) are linear (and separate points), also $C^\alpha(U,E)$ is a topological
vector space.
\item[(b)]
If $\alpha=0$, our definition yields
the compact-open topology on~$C^0(U,E)$,
by Lemma~\ref{coveremb}.
\item[(c)]
If $\alpha'\in (\N_0\cup\{\infty\})^n$ with $\alpha'\leq\alpha$,
then the inclusion mapping\linebreak
$\iota\colon C^\alpha(U,E)\to C^{\alpha'}(U,E)$
is continuous.\\[2.4mm]
[For $\beta\in \N_0^n$ such that $\beta\leq \alpha'$, let
the mapping $\Delta_{\beta,V}$ be as in (\ref{pamps})
and $\Delta_{\beta,V}'\colon C^{\alpha'}(U,E)\to C(V^{<\beta>},E)$
be the analogous map. Then $\Delta_{\beta,V}'\circ\iota=\Delta_{\beta,V}$ is
continuous, from which the assertion follows.]
\item[(d)]
In particular, the compact-open $C^\alpha$-topology on $C^\alpha(U,E)$
is always finer than the compact-open topology (induced from $C(U,E)$).
\end{itemize}
\end{rem}
\begin{la}\label{covcart}
Let $(V_a)_{a\in A}$ be a cover of $U$
by open cartesian subsets $V_a\sub U$.
Then the compact-open $C^\alpha$-topology on $C^\alpha(U,E)$
is the initial topology $\cT$
with respect to the maps $\Delta_{\beta,V_a}$,
for $\beta\in \N_0^n$ such that $\beta\leq\alpha$ and $a\in A$.
\end{la}
\begin{proof}
By definition, the compact-open $C^\alpha$-topology $\cO$ makes each of the maps
$\Delta_{\beta,V_a}$ continuous. Hence $\cT\sub \cO$.
To see that $\cO\sub \cT$, we have to show
that $\Delta_{\beta,V}$ is continuous as a map on $(C^\alpha(U,E),\cT)$,
for each $\beta\in \N_0^n$ such that $\beta\leq\alpha$ and
cartesian open subset $V\sub U$.\\[2.4mm]
Given $\beta$ and $V$,
pick open cartesian subsets $W_t\sub V$
for $t\in V$ such that $W_t\sub V_{a_t}$
for some $a_t\in A$.
Define $I_\beta$ and open sets $(W_t)^{<\beta>}$
and $D_{\beta,i,j,k}$
as in the proof of Lemma~\ref{islocal},
for $i\in I_\beta$ and $j<k$ in $\{0,1,\ldots,\beta_i\}$.
Let
\[
\rho_t\colon
C(V^{<\beta>},E)\to C((W_t)^{<\beta>},E),\quad
\sigma_t\colon C((V_{a_t})^{<\beta>},E)\to C((W_t)^{<\beta>},E)
\]
and $\rho_{i,j,k}\colon C(V^{<\beta>},E)\to C(D_{\beta,i,j,k},E)$
be the respective restriction maps.
Since $V^{<\beta>}$ is the union
of the open sets $(W_t)^{<\beta>}$ and $D_{\beta,i,j,k}$,
we deduce with Lemma~\ref{coveremb} that~$\cT$ will make $\Delta_{\beta,V}$
continuous if it makes all of the maps
$\rho_t\circ \Delta_{\beta,V}$ and $\rho_{i,j,k}\circ \Delta_{\beta,V}$
continuous.
We first note that $\cT$ makes
$\rho_t\circ \Delta_{\beta,V}=\sigma_t\circ \Delta_{\beta,V_{a_t}}$
continuous. If $\beta=0$, this implies the continuity of the restriction map
$\Delta_{0,V}$ (since $V=V^{<0>}$ is covered by the sets $W_t=(W_t)^{<0>}$).
Now assume that $|\beta|\geq 1$ and assume,
by induction, that $\cT$ makes $\Delta_{\gamma,V}$
continuous for all $\gamma\in \N_0^n$ such that $\gamma\leq\alpha$
and $|\gamma|<|\beta|$.
Fix $i\in I_\beta$ and $j<k$ in $\{0,1,\ldots,\beta_i\}$.
Abbreviate $\gamma:=\beta- e_i$.
The map
\[
h\colon D_{\beta,i,j,k}\to \K,\quad
x=(x^{(1)},\ldots, x^{(n)})\mto \frac{1}{x^{(i)}_j-x^{(i)}_k}
\]
is continuous (where $D_{\beta,i,j,k}$
is as in the proof of Lemma~\ref{islocal}).
Hence also the multiplication operator
\[
m_h\colon C(D_{\beta,i,j,k},E)\to C(D_{\beta,i,j,k},E),\quad g\mto h\cdot g
\]
is continuous, by Lemma~\ref{sammelsu}\,(f).
Next, let
$g_1,g_2\colon D_{\beta,i,j,k}\to V^{<\gamma>}$
be the mappings
taking $x=(x^{(1)},\ldots, x^{(n)})$
to
\[
x^{(1)},\ldots, x^{(i-1)},
x^{(i)}_0, \ldots, x^{(i)}_{k-1},x^{(i)}_{k+1}, \ldots, x^{(i)}_{\gamma_i},
x^{(i+1)}\!,\ldots, x^{(n)})
\]
and
$x^{(1)},\ldots, x^{(i-1)},
x^{(i)}_0, \ldots, x^{(i)}_{j-1},x^{(i)}_{j+1}, \ldots, x^{(i)}_{\gamma_i},
x^{(i+1)},\ldots, x^{(n)})$,
respectively.
Then $g_1$ and $g_2$ are continuous,
entailing that also the pullbacks $C(g_1,E)$ and $C(g_2,E)$ are continuous
as mappings $C(V^{<\gamma>},E)\to C(D_{\beta,i,j,k},E)$.
By (\ref{varschlmm}) and (\ref{nfortop}), we have
\[
\rho_{i,j,k}(\Delta_{\beta,V}(f))=(f|_V)^{<\beta>}|_{D_{\beta,i,j,k}}=f_{\beta,i,j,k}
= h\cdot ((f|_V)^{<\gamma>}\circ g_1 -
(f|_V)^{<\gamma>}\circ g_2).
\]
Hence
$\rho_{i,j,k}\circ \Delta_{\beta,V}=
m_h\circ (C(g_1,E)\circ \Delta_{\gamma,V}-C(g_2,E)\circ \Delta_{\gamma,V})$.
The right-hand side is composed of continuous maps
and hence continuous.
Thus $\rho_{i,j,k}\circ \Delta_{\beta,V}$ is continuous on $(C^\alpha(U,E),\cT)$.
\end{proof}
\begin{la}
If $U\sub \K^n$ is open or cartesian, then the
compact-open $C^\alpha$-topology on $C^\alpha(U,E)$
is the initial topology
with respect to the maps
\[
\Delta_\beta
\colon C^\alpha(U,E)\to C(U^{<\beta>},E),\quad f\mto f^{<\beta>},
\]
for $\beta\in \N_0^n$ such that $\beta\leq\alpha$.
\end{la}
Again, we use the compact-open topology on
$C(U^{<\beta>},E)$ here.\\[2mm]
\begin{proof}
The preceding proof can be repeated verbatim with $V:=U$.
\end{proof}
\begin{la}\label{basicres}
Let $S\sub \K^n$ be a locally cartesian
subset such that $S \sub U$.
Then $f|_S$ is $C^\alpha$ for each $C^\alpha$-map
$f\colon U\to E$. Moreover, the
``restriction map''
\[
\rho\colon C^\alpha(U,E)\to C^\alpha(S,E),\quad f\mto f|_S
\]
is continuous and linear.
\end{la}
\begin{proof}
For $x\in S$, let $V_x\sub U$ and $W_x\sub S$ be open cartesian
subsets containing~$x$.
By \ref{nicernbhd}, we may assume that $V_x=U\cap Q$ and $W_x=S\cap P$
with open cartesian subsets $Q,P$ of $\K^n$.
After replacing $P$ and $Q$ with their open cartesian subset $P\cap Q$
(which is possible by the proof of \ref{nicernbhd}), we may assume that $V_x=U\cap Q$ and $W_x=S\cap Q$,
whence $W_x=S\cap V_x\sub V_x$.
Because the sets $W_x$ form a cover of $S$
for $x\in S$,
and $(f|_{V_x})^{<\beta>}|_{(W_x)^{<\beta>}}$
is a continuous extension of $(f|_{W_x})^{>\beta<}$
for each $\beta\in \N_0^n$ such that $\beta\leq\alpha$,
we deduce with Lemma~\ref{islocal} that $f|_S$ is $C^\alpha$,
with
\begin{equation}\label{preok}
(f|_{W_x})^{<\beta>}=(f|_{V_x})^{<\beta>}|_{(W_x)^{<\beta>}}.
\end{equation}
Clearly $\rho$ is linear.
Define
$\Delta_{\beta,V_x} \colon C^\alpha(U,E)\to C((V_x)^{<\beta>},E)$, $f\mto (f|_{V_x})^{<\beta>}$ and
$\Delta_{\beta,W_x}^S \colon C^\alpha(S,E)\to C((W_x)^{<\beta>},E)$, $f\mto (f|_{W_x})^{<\beta>}$.
Let\linebreak
$\rho_{\beta,x}\colon C((V_x)^{<\beta>},E)\to C((W_x)^{<\beta>},E)$
be the restriction map.
Then (\ref{preok}) can be rewritten as
\begin{equation}\label{nowok}
\Delta_{\beta,W_x}^S\circ \rho =
\rho_{\beta,x}\circ \Delta_{\beta,V_x}\,.
\end{equation}
The right-hand side of (\ref{nowok})
is continuous by Remark~\ref{resop} and the continuity of the maps
 $\Delta_{\beta,V_x}$.
 Hence also the left-hand side is continuous and therefore
 $\rho$ is continuous, using that
 the topology on $C^\alpha(S,E)$
is initial with respect to the maps $\Delta_{\beta,W_x}^S$,
by Lemma~\ref{covcart}.
\end{proof}
\begin{la}\label{covarbop}
Let $(U_a)_{a\in A}$ be a cover of $U$
by open subsets $U_a\sub U$.
Then the compact-open $C^\alpha$-topology on $C^\alpha(U,E)$
is the initial topology $\cT$
with respect to the restriction maps
\[
\rho_a\colon C^\alpha(U,E)\to C^\alpha(U_a,E),\quad
f\mto f|_{U_a}\qquad\mbox{for $\,a\in A$.}
\]
\end{la}
\begin{proof}
For $a\in A$, let $\cV_a$ be the set of all cartesian open subsets
$V\sub U_a$. For $V\in \cV_a$ and $\beta\in \N_0^n$ such that $\beta\leq\alpha$,
let $\Delta_{\beta,V}\colon C^\alpha(U,E)\to C(V^{<\beta>},E)$
and $\Delta_{\beta,V}^a\colon C^\alpha(U_a,E)\to C(V^{<\beta>},E)$
be defined via $f\mto (f|_V)^{<\beta>}$.
Because the topology on $C^\alpha(U_a,E)$ is initial
with respect to the maps $\Delta^a_{\beta,V}$ with $\beta$ and $V$ as before,
the
``transitivity of initial topologies'' implies that $\cT$
is the initial topology with respect to the mappings
$\Delta_{\beta,V}^a\circ \rho_a=\Delta_{\beta,V}$.
As $\bigcup_{a\in A}\cV_a$ is a cover of~$U$
by cartesian open subsets of~$U$,
the latter topology coincides with the compact-open $C^\alpha$-topology on $C^\alpha(U,E)$,
by Lemma~\ref{covcart}.
\end{proof}
\begin{la}\label{cloima}
For each cover $\cV$ of $U$ by open cartesian sets $V\sub U$,
the map
\begin{equation}\label{nowprd}
\Delta:=(\Delta_{\beta, V})\colon C^\alpha(U,E)\to\prod_{\beta, V}C(V^{<\beta>},E),\quad
f\mto ((f|_V)^{<\beta>})_{\beta,V}
\end{equation}
is linear and a topological embedding with closed image $($where the product
is taken over all $V\in \cV$ and all $\beta\in \N_0^n$ such that $\beta\leq \alpha)$.
\end{la}
\begin{proof}
The linearity is clear.
As $\Delta_{0,V}$ is the restriction map $f\mto f|_V$ and $\cV$
is a cover of~$U$, the map $\Delta$ is injective.
Combining this with Lemma~\ref{covcart}, we find that $\Delta$
is a topological embedding.
To see that the image is closed, let $(f_a)_{a\in A}$
be a net in $C^\alpha(U,E)$ such that $\Delta(f_a)\to (g_{\beta,V})_{\beta,V}$
with functions $g_{\beta,V}\in C(V^\beta,E)$.
If $x\in U$, there is $V\in \cV$ such that $x\in V$.
Then $f(x):=\lim f_a(x)$ exists,
because $f_a(x)=\Delta_{0,V}(f_a)(x)\to g_{0,V}(x)$.
By the preceding, $f|_V=g_{0,V}$,
whence $f$ is continuous.
It is clear from (\ref{badbd}) that $((f_a)|_V)^{>\beta<}$
converges pointwise to $(f|_V)^{>\beta<}$.
On the other hand, $((f_a)|_V)^{>\beta<}=\Delta_{\beta,V}(f_a)|_{V^{>\beta<}}$
converges pointwise to $g_{\beta,V}|_{V^{>\beta<}}$.
Hence $(f|_V)^{>\beta<}=g_{\beta,V}|_{V^{>\beta<}}$.
Since $g_{\beta,V}$ provides a continuous extension for the latter function, we see with Lemma~\ref{islocal}
that $f$ is $C^\alpha$,
with $(f|_V)^{<\beta>}=g_{\beta,V}$ for all $V\in \cV$ and $\beta\in \N_0^n$ such that $\beta\leq\alpha$.
\end{proof}
\begin{la}\label{pllower}
Let $B_\alpha$ be the set of all $\beta\in\N_0^n$ such that $\beta\leq \alpha$.
The spaces $C^\beta(U,E)$ form a projective system
of topological $\K$-vector spaces for $\beta\in B_\alpha$, together with the inclusion
maps $\phi_{\gamma,\beta}\colon C^\beta(U,E)\to C^\gamma(U,E)$
for $\gamma\leq\beta$ in $B_\alpha$.
Moreover, $C^\alpha(U,E)$ is the projective limit of this system,
together with the inclusion maps $\phi_\beta\colon
C^\alpha(U,E)\to C^\beta(U,E)$ for $\beta\in B_\alpha$.
\end{la}
\begin{proof}
If $(f_\beta)_{\beta\in B_\alpha}$ is an element of
\[
\pl \,C^\beta(U,E)=\{ (f_\beta)_{\beta\in B_\alpha}\in \prod_{\beta\in B_\alpha}
C^\beta(U,E)\colon (\beta\geq \gamma)\impl
f_\gamma=\phi_{\gamma,\beta}(f_\beta)\},
\]
then
$\phi_{\gamma,\beta}(f_\beta)=f_\gamma$ and thus $f_\beta=f_\gamma$,
whenever $\beta\geq \gamma$.
Hence $f_0=f_\beta$ for all $\beta\in B_\alpha$, and thus $f_0\in C^\alpha(U,E)$.
As a consequence, the map
\[
\Psi\colon C^\alpha(U,E)=\bigcap_{\beta\in B_\alpha}  C^\beta(U,E)\to \, \pl \, C^\beta(U,E),\quad
f\mto (f)_{\beta\in B_\alpha}
\]
is an isomorphism of vector spaces.
Let $\Delta_{\beta, V}\colon C^\alpha(U,E)\to C(V^{<\beta>},E)$
be as in (\ref{pamps}),
$\Delta_{\gamma, V}^\beta\colon C^\beta(U,E)\to C(V^{<\gamma>},E)$ for $\gamma\leq \beta$
be the analogous map and $\pi_\beta\colon \pl\,C^\beta(U,E)\to C^\beta(U,E)$
be the projection onto the $\beta$-component.
Using the transitivity of initial topologies twice, we see:
The initial topology on $C^\alpha(U,E)$ with respect to $\Psi$
coincides with the initial topology with respect to the maps $\pi_\beta\circ\Psi$;
and this topology is initial with respect to the maps $\Delta^\beta_{\gamma,V}\circ \pi_\beta\circ\Psi
=\Delta_{\gamma,V}$. It therefore coincides with the compact-open $C^\alpha$-topology on $C^\alpha(U,E)$,
and hence $\Psi$ is a homeomorphism.
\end{proof}
Recall that a topological space $X$ is called \emph{hemicompact}
if there exists a sequence $K_1\sub K_2\sub\cdots$
of compact subsets of~$X$, with $\bigcup_{n\in \N}K_n=X$, such that
each compact subset of~$X$ is contained in some~$K_n$.
For example, every $\sigma$-compact locally compact space is hemicompact.\\[2.4mm]
Also, we recall: For each point $x$ in a locally compact topological space~$X$ and open neighbourhood
$P$ of~$x$, there exists a $\sigma$-compact open neighbourhood $Q$ of $x$ in $X$
such that $Q\sub P$.\footnote{If $L\sub P$ is compact, then each $y\in L$ has a compact neighbourhood
$L_y$ in $P$. By compactness of $L$, we have $L\sub \bigcup_{y\in \Phi}L_y=:L'$
for some finite subset $\Phi\sub L$. Then $L$ is in the interior of $L'$ and $L'\sub P$.
Starting with $L_1:=\{x\}$, this construction yields compact subsets $L_{n+1}:=(L_n)'$ of $P$
for $n\in \N$ such that $L_n$ is in the interior $L_{n+1}^0$.
Thus $Q=\bigcup_{n\in \N}L_n=\bigcup_{n\in \N}L_n^0$ has the desired properties.}
\begin{la}\label{nula1}
If $U\sub \K^n$ is a locally cartesian subset which is locally compact and $\sigma$-compact,
then $U$ admits a countable cover by cartesian open subsets $V\sub U$ which are
$\sigma$-compact.
\end{la}
\begin{proof}
It suffices to show that each $x\in U$ has an open cartesian neighbourhood $V_x\sub U$ which is $\sigma$-compact.
In fact, writing $U=\bigcup_{n\in \N}C_n$ with compact sets $C_n$,
each $C_n$ will be covered by $V_x$ with $x$ in a finite subset $\Phi_n\sub C_n$.
Then the $V_x$ for $x$ in the countable set
$\bigcup_{n\in \N}\Phi_n$ provide the desired countable cover for~$U$.\\[2.4mm]
Fix $x=(x_1,\ldots, x_n)\in U$. To construct $V_x$, start with a cartesian open neighbourhood
$W=W_1\times\cdots\times W_n\sub U$ of~$x$. Since $U$ is locally compact, there exists
a compact neighbourhood $K\sub W$ of~$x$.
Let $\pr_j\colon \K^n\to \K$ be the projection onto the $j$-th component, for $j\in \{1,\ldots, n\}$.
Then $K_j:=\pr_j(K)$ is a compact subset of $W_j$
and thus $K_1\times\cdots \times K_n\sub W$.
After replacing $K$ with a larger set, we may therefore assume that
$K=K_1\times\cdots\times K_n$.
Let $K^0$ be the interior of $K$ relative~$W$.
Then $K^0=K_1^0\times\cdots\times K_n^0$,
where $K_j^0$ is the interior of~$K_j$ relative~$W_j$.
Because $K_j$ is compact and $K_j^0$ an open neighbourhood of~$x_j$ in~$K_j$,
there exists a $\sigma$-compact, open neighbourhood~$Q_j$ of~$x_j$ in~$K_j$
such that $Q_j\sub K_j^0$. Then $Q_j$ is open in $K_j^0$ and hence also
open in~$W_j$, whence $Q_j$ does not have isolated points and
$V_x:=Q_1\times\cdots\times Q_n$
is open in~$W$ and hence in~$U$. Since each $Q_j$ is $\sigma$-compact,
so is the finite direct product~$V_x$.
\end{proof}
\begin{la}\label{nula2}
If $V\sub \K^n$ is a cartesian subset which is locally compact and $\sigma$-compact,
then $V^{<\beta>}$ is locally compact and $\sigma$-compact,
for each $\beta\in \N_0^n$.
\end{la}
\begin{proof}
Let $\pr_j\colon \K^n\to \K$ be the projection onto the $j$-th component, for $j\in \{1,\ldots, n\}$.
We may assume $V\not=\emptyset$.
Write $V=V_1\times\cdots\times V_n$ with $V_j\sub \K$.
As $V$ is locally compact, each $V_j$ is locally
compact. Further, $V_j=\pr_j(V)$ is $\sigma$-compact.
Hence also each $V^{<\beta>}$ is locally compact and $\sigma$-compact,
being a finite direct product of copies of the factors $V_1,\ldots, V_n$
(see (\ref{spcas})).
\end{proof}
\begin{la}\label{samsu}
The space $C^\alpha(U,E)$ has the following properties:
\begin{itemize}
\item[{\rm(a)}]
If $\K$ is metrizable and $E$ is complete,
then $C^\alpha(U,E)$ is complete.
\item[{\rm(b)}]
If $\K$ is metrizable and $E$ is sequentially complete,
then $C^\alpha(U,E)$ is sequentially complete.
\item[{\rm(c)}]
If $U$ is $\sigma$-compact and locally compact
and $E$ is metrizable,
then $C^\alpha(U,E)$ is metrizable.
\item[{\rm(d)}]
If $\K$ is an ultrametric field and $E$ is locally convex,
then $C^\alpha(U,E)$ is locally convex.
\item[{\rm(e)}]
If $\K$ is an ultrametric field, $U$ is compact,
$E$ an ultrametric normed space and $\alpha\in \N_0^n$,
then $C^\alpha(U,E)$ is an ultrametric normed
space.
\item[{\rm(f)}]
If $\K$ is an ultrametric field, $U$ is compact,
$E$ an ultrametric Banach space and $\alpha\in \N_0^n$,
then $C^\alpha(U,E)$ is an ultrametric Banach
space.
\item[{\rm(g)}]
Assume that there exists a cover $\cV$ of $U$ by cartesian open subsets $V\sub U$
such that $V^{<\beta>}$ is a k-space for each $\beta\in \N_0^n$ such that $\beta\leq\alpha$.
Then
$C^\alpha(U,E)$ is complete if $E$ is complete,
and $C^\alpha(U,E)$ is sequentially complete if $E$ is sequentially complete.
\item[{\rm(h)}]
Let $E$ be metrizable and
assume that there exists a countable cover $\cV$ of $U$ by cartesian open subsets $V\sub U$
such that $V^{<\beta>}$ is hemicompact
for each $\beta\in \N_0^n$ such that $\beta\leq\alpha$.
Then
$C^\alpha(U,E)$ is metrizable.
\end{itemize}
\end{la}
\begin{proof}
(a) and (b) follow from (g), because every metrizable topological space is a k-space.

(c) By Lemma~\ref{nula1}, $U$ admits a countable cover $\cV$ by cartesian open subsets $V\sub U$
which are locally compact and $\sigma$-compact.
For such $V$ and $\beta\in \N_0^n$ with $\beta\le\alpha$, the set
$V^{<\beta>}$ is locally compact and $\sigma$-compact (by Lemma~\ref{nula2})
and hence hemicompact. Hence (c) follows from (h).

(d) By Lemma~\ref{cloima}, the topological vector space $C^\alpha(U,E)$ is isomorphic
to a vector subspaces of a direct product of the spaces of the form
$C(V^{<\beta>},E)$.
Each of these is locally convex by Lemma~\ref{sammelsu}\,(g), and hence
also $C^\alpha(U,E)$.

(e) Assume first that the compact set $U$ is cartesian,
$U_1\times \cdots\times U_n=U$ with compact subsets $U_1,\ldots, U_n\sub \K$.
We can then take $\cV=\{U\}$ in Lemma~\ref{cloima}.
Hence $C^\alpha(U,E)$ is isomorphic
to a vector subspaces of a finite direct product of spaces of the form
$C(U^{<\beta>},E)$.
Each $C(U^{<\beta>},E)$ is an ultrametric normed space by Lemma~\ref{sammelsu}\,(h),
using that $U^{<\beta>}$ is of the form (\ref{spcas})
and hence compact. Hence also
$C(U^{<\beta>},E)$ is an ultrametric normed space.\\[2.4mm]
We now return to the general case,
assuming only that the compact set $U$ is locally cartesian.
Fix $x\in U$.
Then $x$ has a cartesian open neighbourhood $W=W_1\times\cdots\times W_n\sub U$
and the latter contains a compact neighbourhood~$K$
of $x$ of the form $K=K_1\times\cdots\times K_n$
(see proof of Lemma~\ref{nula1}).
We let $V_x:=K^0$ be the interior of $K$ relative~$W$.
Then $V_x=K_1^0\times\cdots\times K_n^0$,
where $K_j^0$ is the interior of~$K_j$ relative~$W_j$.
Since $W_j$ does not have isolated points and $K_j^0$ is open
in $W_j$, also $K_j^0$ does not have isolated points,
and hence also its closure $\wb{K_j^0}$ does not have isolated points.
Hence $L_x:=\wb{V_x}=\wb{K_1^0}\times\cdots \times \wb{K_n^0}$
is a compact cartesian subset of $\K^n$.
Since $V_x$ is a neighbourhood of $x$ in~$U$,
by compactness of~$U$ there is a finite subset $\Phi\sub U$
such that $U=\bigcup_{x\in \Phi}V_x$.
For each $x\in \Phi$, the restriction maps
\[
\rho_x\colon C^\alpha(U,E)\to C^\alpha(L_x,E)\quad\mbox{and}\quad
\sigma_x\colon C^\alpha(L_x,E)\to C^\alpha(V_x,E)
\]
are continuous linear, by Lemma~\ref{basicres}.
Hence the mappings
\[
\rho:=(\rho_x)_{x\in \Phi}\colon C^\alpha(U,E)\to\prod_{x\in \Phi}C^\alpha(L_x,E)
\]
and
\[
\sigma:=\prod_{x\in \Phi}\sigma_x\colon
\prod_{x\in \Phi}C^\alpha(L_x,E)\to
\prod_{x\in \Phi}C^\alpha(V_x,E), \;\;\; (f_x)_{x\in \Phi}\mto (\sigma_x(f_x))_{x\in \Phi}
\]
are continuous linear.
Because $\sigma\circ\rho$ is a topological embedding (by Lemma~\ref{cloima}),
we deduce that also~$\rho$ is a topological embedding.
Since each of the spaces $C^\alpha(L_x,E)$ is an ultrametric normed space,
so is their finite direct product $\prod_{x\in \Phi}C^\alpha(L_x,E)$
and hence also $C^\alpha(U,E)$.

(f) follows from (a) and (e).

(g) By Lemma~\ref{cloima}, the topological vector space $C^\alpha(U,E)$ is
isomorphic to a closed vector subspace of the product of the spaces
$C(V^{<\beta>},E)$ with $\beta$ and $V$ as described in~(g).
Each of the latter is complete (resp., sequentially complete),
by Lemma~\ref{sammelsu} (d) and (e).
Hence so is $C^\alpha(U,E)$.

(h) Let $B_\alpha$ be the set of all $\beta\in \N_0^n$ such that
$\beta\leq \alpha$. Since $B_\alpha$ and $\cV$ are countable, also
$B_\alpha\times \cV$ is countable.
By Lemma~\ref{sammelsu}\,(c), each of the spaces $C(V^{<\beta>},E)$
is metrizable. Hence also the countable direct product $\prod_{(\beta,V)\in B_\alpha\times \cV}C(V^{<\beta>},E)$
is metrizable and hence also $C^\alpha(U,E)$, being homeomorphic to
a subspace of the latter product by Lemma~\ref{cloima}.
\end{proof}
\begin{la}\label{alphpushf}
If $\lambda\colon E\to F$ is a continuous linear mapping between
topological $\K$-vector spaces,
then the map
\[
C^\alpha(U,\lambda)\colon
C^\alpha(U,E)\to
C^\alpha(U,F),\quad f\mto \lambda\circ f
\]
is continuous and linear.
If $\lambda$ is linear and a topological embedding, then also
$C^\alpha(U,\lambda)$ is linear and a topological embedding.
\end{la}
\begin{proof}
Let $\Delta_{\beta,V}\colon C^\alpha(U,E)\to C(V^{<\beta>},E)$
be as in Definition~\ref{coalpha}
and $\Delta_{\beta,V}^F\colon C^\alpha(U,F)\to C(V^{<\beta>},F)$
be the analogous map.
By Lemma~\ref{linchain},
\begin{equation}\label{leftright}
\Delta_{\beta,V}^F
\circ C^\alpha(U,\lambda)
=C(V^{<\beta>},\lambda)\circ \Delta_{\beta,V}.
\end{equation}
Because the right-hand side of (\ref{leftright}) is continuous, also
each of the maps $\Delta_{\beta,V}^F \circ C^\alpha(U,\lambda)$ is continuous
and hence also $C^\alpha(U,\lambda)$
(as the topology on $C^\alpha(U,F)$ is initial with respect to the mappings
$\Delta_{\beta,V}^F$).\\[2.4mm]
If $\lambda$ is an embedding, then the topology on~$E$ is initial with respect to $\lambda$.
Since $C^\alpha(U,\lambda)$
is injective, it will be an embedding if we can show that the
topology on
$C^\alpha(U,E)$ is initial with respect to
$C^\alpha(U,\lambda)$.
This is a special case of the next lemma.
\end{proof}
\begin{la}\label{alphinit}
If the topology on $E$ is initial with respect to a family $(\lambda_j)_{j\in J}$
of linear maps $\lambda_j\colon E\to E_j$ to topological $\K$-vector spaces~$E_j$,
then the compact-open $C^\alpha$-topology on $C^\alpha(U,E)$
is the initial topology with respect to the linear maps $C^\alpha(U,\lambda_j)\colon C^\alpha(U,E)\to C^\alpha(U,E_j)$,
for $j\in J$.
\end{la}
\begin{proof}
Let $\cO$ be the compact-open $C^\alpha$-topology
on $C^\alpha(U,E)$ and $\cT$ be the initial topology described in the lemma.\\[2.4mm]
Step 1. By Lemma~\ref{inipush}, the topology on $C(V^{<\beta>},E)$ is initial with respect
to the maps $C(V^{<\beta>},\lambda) \colon C(V^{<\beta>},E)\to C(V^{<\beta>},E_j)$.
Thus, by transitivity of initial topologies,
$\cO$ is initial with respect to the maps
$C(V^{<\beta>},\lambda_j)\circ \Delta_{\beta,V}$ with $\beta\in \N_0^n$ such that
$\beta\leq\alpha$, open cartesian subsets $V\sub U$
and $\Delta_{\beta,V}$ as in Definition~\ref{coalpha}.\\[2.4mm]
Step~2. Let $\Delta_{\beta,V}^{E_j}\colon C^\alpha(U,E_j)\to C(V^{<\beta>},E_j)$
be the analogous maps.
The topology on $C^\alpha(U,E_j)$ is initial with respect to the maps $\Delta^{E_j}_{\beta,V}$.
Hence,
by transitivity of initial topologies, the topology~$\cT$ is initial
with respect to the maps $\Delta^{E_j}_{\beta,V}\circ C^\alpha(U,\lambda_j)$.
But $\Delta^{E_j}_{\beta,V}\circ C^\alpha(U,\lambda_j)=
C(V^{<\beta>},\lambda_j)\circ \Delta_{\beta,V}$, by Lemma~\ref{linchain}.
Comparing with Step~1, we see that $\cO=\cT$.
\end{proof}
\begin{la}\label{alphvsub}
If $E_0\sub E$ is a vector subspace,
then the compact-open $C^\alpha$-topology on
$C^\alpha(U,E_0)$ coincides with the topology induced by
$C^\alpha(U,E)$.
\end{la}
\begin{proof}
Apply Lemma~\ref{alphpushf} to the inclusion map $\lambda\colon E_0\to E$.
\end{proof}
\begin{la}\label{alphprod}
If $E=\prod_{j\in J}E_j$ as a topological $\K$-vector
space, with the canonical projections $\pr_j\colon E\to E_j$,
then the mapping
\[
\Psi:=(C^\alpha(U,\pr_j))_{j\in J}\colon C^\alpha(U,E)\to \prod_{j\in J}C^\alpha(U,E_j)
\]
is an isomorphism of topological vector spaces. 
\end{la}
\begin{proof}
Lemma~\ref{inprodu} entails that $\Psi$ is bijective and hence
an isomorphism of vector spaces.
Because the topology on $E$ is initial with respect to the family $(\pr_j)_{j\in J}$,
Lemma~\ref{alphinit} shows that the topology on $C^\alpha(U,E)$
is initial with respect to the maps $C^\alpha(U,\pr_j)$.
Hence $\Psi$ is a topological embedding and hence $\Psi$
is an isomorphism of topological vector spaces.
\end{proof}
\begin{la}\label{plinim}
Let $E=\pl\, E_j$ be as in Lemma~{\rm\ref{plfirst}}, with the bonding maps $\phi_{i,j}\colon E_j\to E_i$
and the limit maps
$\phi_j\colon E\to E_j$. Then the spaces $C^\alpha(U,E_j)$ form a projective system
with bonding maps $C^\alpha(U,\phi_{i,j})$, and
\[
C^\alpha(U,E)=\pl\,C^\alpha(U, E_j),
\]
with the limit maps $C^\alpha(U,\phi_j)$.
\end{la}
\begin{proof}
Let $P:=\prod_{j\in J}E_j$ and $\pr_j\colon P\to E_j$
be the projection onto the $j$-th component.
It is known that the map
\[
\phi:=(\phi_j)_{j\in J}\colon E\to P
\]
is a topological embedding with image $\phi(E)=\{(x_j)_{j\in J}\in P\colon
(\forall i,j\in J)$\linebreak
$i\leq j \impl x_i=\phi_{i,j}(x_j)\}$.
Thus $C^\alpha(U,\phi)$ is a topological embedding (see Lemma~\ref{alphpushf})
and hence also
\[
\Psi\circ C^\alpha(U,\phi) \colon C^\alpha(U,E)\to\prod_{j\in J}C^\alpha(U,E_j),
\]
using the isomorphism $\Psi:=(C^\alpha(U,\pr_j))_{j\in J}\colon C^\alpha(U,P)\to
\prod_{j\in J}C^\alpha(U,E_j)$ of topological vector spaces from Lemma~\ref{alphprod}.
The image of $\Psi\circ C^\alpha(U,\phi)$ is contained in the projective limit
\[
L:=\{(f_j)_{j\in J}\in\prod_{j\in J}C^\alpha(U,E_j)\colon
(\forall i,j\in J)\, i\leq j \impl f_i=\phi_{i,j}\circ f_j\}
\]
of the spaces $C^\alpha(U,E_j)$.
If $(f_j)_{j\in J} \in L$ and $x\in U$, then
$\phi_{i,j}(f_j(x))=f_i(x)$ for all $i\leq j$, whence
there exists $f(x)\in E$ with $\phi_i(f(x))=f_i(x)$.
By Lemma~\ref{plfirst}, we have $f\in C^\alpha(U,E)$.
Now $(\Psi\circ C^\alpha(U,\phi))(f)=(f_j)_{j\in J}$.
Hence $\Psi\circ C^\alpha(U,\phi)$ has image $L$ and hence is an isomorphism of topological vector spaces from
$C^\alpha(U,E)$ onto~$L$. The assertions follow.
\end{proof}
The following observation will be useful in
the proof of Theorem~B.
\begin{la}\label{laavoid}
For each $f\in C^\alpha(U,\K)$ and $a\in E$, we have $f a\in C^\alpha(U,E)$,
and the following linear map is continuous:
\[
\theta\colon E\to C^\alpha(U,E),\quad a\mto f a.
\]
\end{la}
\begin{proof}
The map $m_a\colon \K\to E$, $x\mto xa$ is continuous linear,
whence $f a =$\linebreak
$m_a\circ f\in C^\alpha(U,E)$
(see Lemma~\ref{linchain}).
Let $\Delta_{\beta,V}\colon C^\alpha(U,E)\to C(V^{<\beta>},E)$
be as in Definition~\ref{coalpha} and define
$\Delta_{\beta,V}^\K \colon C^\alpha(U,\K)\to C(V^{<\beta>},\K)$
analogously. Then
\[
\Delta_{\beta,V}\circ \theta\colon E\to C(V^{<\beta>},E)
\]
is continuous. In fact, we can write
$\Delta_{\beta,V}(\theta(a))=
\Delta_{\beta,V}(m_a\circ f)
=m_a\circ \Delta_{\beta,V}^\K(f)
=\Delta_{\beta,V}^\K(f) \cdot c_a$,
where the dot denotes the
continuous module multiplication $C(V^{<\beta>},\K)\times C(V^{<\beta>},E)\to C(V^{<\beta>},E)$
from Lemma~\ref{sammelsu}\,(f)
and the continuous linear map
$E\to C(V^{<\beta>},E)$, $a\mto c_a$ is as in Lemma~\ref{tocon}.
Hence $\theta$ is continuous.
\end{proof}
\begin{rem}
In our terminology,
\cite[Lemma~3.9]{NDI}
asserts that, for each complete ultrametric field $\K$,
locally cartesian subset $U\sub \K^n$
and compact set $K\sub U$ of the form
$K=K_1\times\cdots \times K_n$,
there exists a cartesian open subset $V\sub U$
with $K\sub V$. The next example shows that this
assertion is incorrect.\footnote{In the main cases when $U$ is open \cite[5.12]{Kel}
or cartesian, the assertion is correct.}
Therefore the definition of the topology
of  ``topology of compact cartesian convergence''
in \cite[pp.\,24--25]{Nag} does not work as stated
(instead, one should only consider
compact sets which are contained in some
cartesian open subset $V$
of the domain of definition).
 \end{rem}
\begin{example}
Let $U=(V_1\times V_2)\cup (W_1\times W_2)$, $v$ and $w$
be as in Example~\ref{exa1},
and $K:=\{0\}\times \{v,w\}$.
Then $K$ does not have an open cartesian neighbourhood in~$U$.\\[2.4mm]
[Suppose that $P_1\times P_2\sub U$ is open
and $K\sub P_1\times P_2$.
Then $0\in P_1$ and $v,w\in P_2$.
Since $(p^m,v)\to (0,v)$ as $m\to\infty$
and $(p^m,v)\in U$, there is $m\in \N$ such that $(p^m,v)\in P_1\times P_2$.
Thus $p^m\in P_1$
and hence $(p^m,w)
\in P_1\times P_2\sub U$.
Now $(\Q_p\times W_2)\cap U=W_1\times W_2$.
As this set does not contain $(p^m,w)$, we
have reached a contradiction.]\\[2.4mm]
Alternatively, note that if a locally cartesian set $U\sub \K^n$
satisfies the conclusion of \cite[Lemma~3.9]{NDI},
then $U^{>\beta<}$ is dense in
$U^{<\beta>}$ for each $\beta\in \N_0^n$.
We have seen in Example~\ref{exa1} that~$U$ does not have this property.\footnote{Given $x\in U^{<\beta>}$, define $K_i:=\{x^{(i)}_0,\ldots, x^{(i)}_{\beta_i}\}$
for $i\in \{1,\ldots, n\}$.
Then $K:=K_1\times\cdots\times K_n\sub U$, by definition of $U^{<\beta>}$.
If an open cartesian set $V\sub U$ with $K\sub V$
exists, then $x\in V^{<\beta>}$,
which is in the closure of
$V^{>\beta<}$. Hence $x$ is in the closure of
$U^{>\beta<}$.}
\end{example}
%
%
%
%
%
%
%
\section{Proof of Theorem A}
If $x\in U$, then $f_x:=f(x,\sbull)=f^\vee(x)\in C^\beta(V,E)$.
To see this, let $Q\sub V$ be a cartesian open subset
and $\eta\in \N_0^m$ such that $\eta\leq\beta$.
There is an open cartesian subset $P\sub U$ such that $x\in P$.
It is clear that $(f|_{P\times Q})^{>(0,\eta)<}(x,y)=(f_x|_Q)^{>\eta<}(y)$
for all $y\in Q^{>\eta<}$.
Hence $(f|_{P\times Q})^{<(0,\eta)>}(x,\sbull)\colon Q^{<\eta>}\to E$
is a continuous extension
for $(f_x|_Q)^{>\eta<}$ and thus $f_x$ is $C^\beta$.\\[2.4mm]
We claim: \emph{If $f\in C^{(\alpha,\beta)}(U\times V,E)$,
$P\sub U$ and $Q\sub V$ are cartesian open subsets,
$\gamma\in \N_0^n$ with $\gamma\leq\alpha$ and
$\eta\in \N_0^m$ with $\eta\leq\beta$,
then}
\begin{equation}\label{keyobs}
(\Delta_{\eta,Q}\circ f^\vee|_P)^{>\gamma<}(x)(y)=f^{>(\gamma,\eta)<}(x,y)
\end{equation}
\emph{for all $(x,y)\in P^{>\gamma<}\times Q^{>\eta<}=(Q\times P)^{>(\gamma,\eta)<}$,
where $\Delta_{\eta,Q}$ denotes the map $C^\beta(V,E)\to C(Q^{<\eta>},E)$, $g\mto g^{<\eta>}$.}\\[2.4mm]
If this is true, holding $x$ fixed we deduce that
\[
(\Delta_{\eta,Q}\circ f^\vee)^{>\gamma<}(x)(y)=(f|_{P\times Q})^{<(\gamma,\eta)>}(x,y)
\]
for all $(x,y)\in P^{>\gamma<}\times Q^{<\eta>}$ (by continuity).
Thus
\[
(\Delta_{\eta,Q}\circ f^\vee)^{>\gamma<}(x)=((f|_{P\times Q})^{<(\gamma,\eta)>})^\vee(x)
\]
for all $x\in P^{>\gamma<}$.
As the maps $((f|_{P\times Q})^{<(\gamma,\eta)>})^\vee\colon
P^{<\gamma>}\to C(Q^{<\eta>},E)$ are continuous
(see Proposition~\ref{ctsexp}),
we deduce that each of the maps $\Delta_{\eta,Q}\circ f^\vee|_P$ is $C^\alpha$,
with
\begin{equation}\label{nochein}
(\Delta_{\eta,Q}\circ f^\vee|_P)^{<\gamma>}=((f|_{P\times Q})^{<(\gamma,\eta>})^\vee\,.
\end{equation}
Using Lemmas~\ref{cloima}, \ref{difinsub} and \ref{inprodu}, we see that $f^\vee|_P$ is $C^\alpha$.
Hence $f^\vee$ is $C^\alpha$, by Definition~\ref{defSDS}\,(b).
Since $\Delta_{\eta,Q}$ is continuous linear, Lemma~\ref{linchain}
and (\ref{nochein}) show  that
\begin{equation}\label{prelstag}
\Delta_{\eta,Q}\circ ((f^\vee|_P)^{<\gamma>})=
(\Delta_{\eta,Q}\circ f^\vee|_P)^{<\gamma>}=
((f|_{P\times Q})^{<(\gamma,\eta)>})^\vee\,.
\end{equation}
By Proposition~\ref{ctsexp}, the map
\[
\Psi_{\gamma,\eta,P,Q}\colon
C(P^{<\gamma>}\times Q^{<\eta>},E)\to
C(P^{<\gamma>}, C(Q^{<\eta>},E)),\quad h\mto h^\vee
\]
is a topological embedding.
Let $\Phi$ be as in Theorem~A and consider the mappings
$\Delta_{\gamma,P}\colon C^\alpha(U,C^\beta(V,E))\to C(P^{<\gamma>},C^\beta(V,E))$, $g\mto g^{<\gamma>}$
and
$\Delta_{(\gamma,\eta),P\times Q}
\colon C^{(\alpha,\beta)}(U\times V,E)\to C(P^{<\gamma>}\times Q^{<\eta>},E)$,
$h\mto h^{<(\gamma,\eta)>}$
(where $(P\times Q)^{<(\gamma,\eta)>}=P^{<\gamma>}\times Q^{<\eta>}$).
We can then re-read (\ref{prelstag}) as
\begin{equation}\label{givsall}
C(P^{<\gamma>},\Delta_{\eta,Q})\circ \Delta_{\gamma,P}\circ\Phi = 
\Psi_{\gamma,\eta,P,Q}
\circ \Delta_{(\gamma,\eta),P\times Q}.
\end{equation}
Let $\cO$ be the compact-open $C^{(\alpha,\beta)}$-topology on $C^{(\alpha,\beta)}(U\times V,E)$
and $\cT$ be the initial topology with respect to $\Phi$.
Since $\cO$ is initial with respect to the maps
$\Delta_{(\gamma,\eta),P\times Q}$ and the topology on
$C(P^{<\gamma>}\times Q^{<\eta>},E)$
is initial with respect to $\Psi_{\gamma,\eta,P,Q}$,
the transitivity of initial topologies shows that $\cO$
is initial with respect to the maps $\Psi_{\gamma,\eta,P,Q}\circ
\Delta_{(\gamma,\eta),P\times Q}$.\\[2.4mm]
The topology on $C^\beta(V,E)$ being initial with respect
to the maps $\Delta_{\eta,Q}$, Lemma~\ref{inipush} shows that the topology
on $C(P^{<\gamma>},C^\beta(V,E))$
is initial with respect to the maps $C(P^{<\gamma>},\Delta_{\eta,Q})$.
Also, the topology on $C^\alpha(U,C^\beta(V,E))$ is initial with respect
to the maps $\Delta_{\gamma,P}$.
Hence, by transitivity of initial topologies, $\cT$
is the initial topology with respect to the maps
\[
C(P^{<\gamma>},\Delta_{\eta,Q})\circ \Delta_{\gamma,P}\circ\Phi.
\]
By (\ref{givsall}), these maps coincide with
$\Psi_{\gamma,\eta,P,Q}
\circ \Delta_{(\gamma,\eta),P\times Q}$. Comparing with the preceding description of~$\cO$
as an initial topology, we see that $\cO=\cT$. Since $\Phi$ is also injective,
we deduce that~$\Phi$ is a topological embedding.\\[2.4mm]
If $\K$ is metrizable, we have to show that~$\Phi$ is surjective.
To this end, let $g\in C^\alpha(U,C^\beta(V,E))$.
Then $g\colon U\to C^\beta(V,E)$ is continuous (see Remark~\ref{alphacts}),
entailing that $g$ is continuous also as a map $U\to C(V,E)$
(see Remark~\ref{compaco}\,(d)).
Since $U\times V$ is metrizable and hence a k-space,
we deduce with Proposition~\ref{ctsexp} that
\[
f:=g^\wedge \colon U\times V\to E,\quad f(x,y)=g(x)(y)
\]
is continuous.
Now
\[
g_{\gamma,\eta,P,Q}
:=\Delta_{\eta,Q}\circ \Delta_{\gamma,P}(g)\in C(P^{<\gamma>},C(Q^{<\eta>},E)).
\]
Because $P^{<\gamma>}\times Q^{<\eta>}$ is metrizable
and hence a k-space, Proposition~\ref{ctsexp}
shows that the associated map
\[
g_{\gamma,\eta,P,Q}^\wedge
\colon P^{<\gamma>}\times Q^{<\eta>}\to E,\quad
(x,y)\mto g_{\gamma,\eta,P,Q}(x)(y)
\]
is continuous.
We claim that
\begin{equation}\label{gvsrest}
(f|_{P\times Q})^{>(\gamma,\eta)<}=g_{\gamma,\eta,P,Q}^\wedge|_{P^{>\gamma<}\times Q^{>\eta<}}.
\end{equation}
If this is true, then
$g_{\gamma,\eta,P,Q}^\wedge$ provides a continuous extension
for $(f|_{P\times Q})^{>(\gamma,\eta)<}$
to $P^{<\gamma>}\times Q^{<\eta>}=(P\times Q)^{<(\gamma,\eta)>}$.
Hence $f$ is $C^{(\alpha,\beta)}$.
Moreover, $\Phi(f)=f^\vee=(g^\wedge)^\vee=g$.
Thus, it only remains to prove the claims.\\[2.4mm]
\emph{Proof of the first claim.}
Since $\Delta_{\eta,Q}$ is linear, we have
$(\Delta_{\eta,Q}\circ f^\vee|_P)^{>\gamma<}
=
\Delta_{\eta,Q}\circ (f^\vee|_P)^{>\gamma<}$.
Let $y\in Q$.
Because the point evaluation $\ve_y\colon C^\beta(V,E)\to E$, $g\mto g(y)$
is linear, we have
\[
\ve_y\circ (f^\vee|_P)^{>\gamma<}=(\ve_y\circ f^\vee)^{>\gamma<}=
(x\mto f(x,y))^{>\gamma<}.
\]
Hence
$(f^\vee|_P)^{>\gamma<}(x)(y)=(\ve_y\circ (f^\vee|_P)^{>\gamma<})(x)$
is given by (\ref{badbd}) for $x\in P^{>\gamma<}$, replacing $\beta$ with $\gamma$,
$V$ with $P$ and $f$ with $f(\sbull,y)$ there.
Holding $x\in P^{>\gamma<}$ fixed
in the resulting formula,
we use (\ref{badbd})
to calculate
$\Delta_{\eta,Q}((f^\vee|_P)^{>\gamma<}(x))(y)$
for $y\in Q^{>\eta<}$.
We obtain that
$\Delta_{\eta,Q}((f^\vee|_P)^{>\gamma<}(x))(y)$ is given by
{\small\[
\sum_{j_1=0}^{\gamma_1}\cdots
\sum_{j_n=0}^{\gamma_n}
\sum_{i_1=0}^{\eta_1}\cdots
\sum_{i_m=0}^{\eta_m}
\Big(\prod_{k_1\not=j_1} \frac{1}{x^{(1)}_{j_1}-x^{(1)}_{k_1}}
\cdot\ldots\cdot
\prod_{k_n\not=j_n}\frac{1}{x^{(n)}_{j_n}-x^{(n)}_{k_n}}\cdot\hspace*{12mm}
\]
\begin{equation}\label{soonag}
\hspace*{5mm}\prod_{\ell_1\not=i_1} \frac{1}{y^{(1)}_{i_1}-y^{(1)}_{Ê\ell_1}}
\cdot\ldots\cdot
\prod_{\ell_m\not=i_m}\frac{1}{y^{(m)}_{i_m}-y^{(m)}_{\ell_m}}\Big)
f(x^{(1)}_{j_1},\ldots, x^{(n)}_{j_n},
y^{(1)}_{i_1},\ldots, y^{(m)}_{i_m}).
\end{equation}}
\hspace*{-1.4mm}Using (\ref{badbd}) to calculate $(f|_{P\times Q})^{>(\gamma,\eta)<}(x,y)$,
we get the same formula.\\[2.4mm]
\emph{Proof of the second claim.}
For $y\in Q$ and $x\in P^{>\gamma<}$,
we can calculate $((g|_P)^{>\gamma<}))(x)(y)=(\ve_y\circ ((g|_P)^{>\gamma<}))(x)
=(\ve_y\circ g|_P)^{>\gamma<}(x)=(f(\sbull,y)|_P)^{>\gamma<}(x)$
using (\ref{badbd}). Holding $x$ fixed,
we can apply (\ref{badbd}) again and infer that
$g_{\gamma,\eta,P,Q}^\wedge(x,y)$ $=
(\Delta_{\eta,Q}\circ \Delta_{\gamma,P}(g))(x,y)$ is given for $y\in Q^{>\eta<}$ by
(\ref{soonag}). Using (\ref{badbd}),
the same formula is obtained for
$(f|_{P\times Q})^{>(\gamma,\eta)<}(x,y)$.
This completes the proof of the claim and hence completes the proof
of Theorem~A.\,\Punkt
\section{Proof of Theorem B}
The proof is by induction on~$n\in \N$.\\[2.4mm]
\emph{The case $n=1$.} Thus $\alpha=(r)$ with $r:=\alpha_1$.
For $E=\K$ and $r<\infty$, the assertion was established
in \cite[Proposition~II.40]{Nag},
and one can proceed analogously
if $r<\infty$ and~$E$ is an ultrametric
Banach space.\\[2.4mm]
Now let $E$ be a locally convex space,
$f\in C^r(U,E)$ and $W\sub C^r(U,E)$ be an open $0$-neighbourhood.\\[2.4mm]
The case $r<\infty$.
Let $\wt{E}$ be a completion of~$E$
with $E\sub \wt{E}$. Since $C^r(U,E)$
carries the topology induced by $C^r(U,\wt{E})$
(see Lemma~\ref{alphvsub}), we have $W=\wt{W}\cap C^r(U,E)$
for some open $0$-neighbourhood $\wt{W}\sub
C^r(U,\wt{E})$.
Now $\wt{E}=\pl \,E_q$ for certain ultrametric Banach spaces
$(E_q,\|.\|_q)$, with continuous linear limit maps $\phi_q\colon \wt{E}\to E_q$
with dense image.
Then $C^r(U,\wt{E})=\pl\,C^r(U,E_q)$
with respect to the maps $C^r(U,\phi_q)$ (see Lemma~\ref{plinim}).
After shrinking~$\wt{W}$ (and $W$), we may therefore assume that
$\wt{W}=C^r(U,\phi_q)^{-1}(V)$
for some $q$ and open $0$-neighbourhood $V\sub C^r(U,E_q)$.
By the case of ultrametric Banach spaces,
there exists $g\in \LocPol_{\leq r}(U,E_q)$
such that $\pi_q\circ f - g\in V$.
Thus, there are $\ell\in \N$, clopen subsets $U_1,\ldots, U_\ell\sub U$
and $c_{i,j}\in E_q$ for
$i\in \{0,1,\ldots, r\}$ and
$j\in \{1,\ldots,\ell\}$
such that
\[
g(x)=\sum_{i=0}^r\sum_{j=1}^\ell x^i \one_{U_j}(x)c_{i,j},
\]
where $\one_{U_j}\colon U\to \K$
takes the value $1$ on $U_j$,
and vanishes elsewhere.
Let us write
$X^i \colon U\to \K$ for the function
$x\mto x^i$.
By Lemma~\ref{laavoid},
there are open neighbourhoods
$Q_{i,j}\sub E_q$ of $c_{i,j}$
such that
\[
\phi_q\circ f - \sum_{i=0}^r\sum_{j=1}^\ell X^i \one_{U_j}b_{i,j}\in V,
\]
for all $b_{i,j}\in Q_{i,j}$.
Since $\phi_q(E)$ is dense in~$E_q$,
there exist $a_{i,j} \in E$
such that $\phi_q(a_{i,j})\in Q_{i,j}$.
Hence $h:=\sum_{i=0}^r\sum_{j=1}^\ell X^i  \one_{U_j}a_{i,j}\in \LocPol_{\leq r}(U,E)$
and
\[
\phi_q\circ (f-h)=\phi_q\circ f - \sum_{i=0}^r\sum_{j=1}^\ell X^i  \one_{U_j}\phi_q(a_{i,j})\in V.
\]
Thus
$f-h\in C^r(U,\phi_q)^{-1}(V)=\wt{W}$ and thus
$f-h\in \wt{W}\cap E=W$.\\[2.4mm]
If $r=\infty$, recall from Lemma~\ref{pllower} that $C^\infty(U,E)=\pl_{k\in \N_0}\,C^k(U,E)$.
We may therefore assume that $W=V\cap C^\infty(U,E)$
for some $k\in \N_0$ and open $0$-neighbourhood $V\sub C^k(U,E)$.
By the preceding, there is
$h\in \LocPol_{\leq k}(U,E)$ such that $f-h\in V$.
Then $f-h\in V\cap C^\infty(U,E)=W$.\\[2.4mm]
\emph{Induction step.}
If $n\geq 2$, we write $U=U_1\times U'$ with $U\sub \K$ and
$U'\sub \K^{n-1}$.
We also write $\alpha=(r,\alpha')$ with $r\in \N_0\cup\{\infty\}$
and $\alpha'\in (\N_0\cup\{\infty\})^{n-1}$.
Let $f\in C^\alpha(U,E)$ and $W\sub C^\alpha(U,E)$
be an open $0$-neighbourhood.
By Theorem~A, the map
\[
\Phi\colon C^\alpha(U,E)\to C^r(U_1,C^{\alpha'}(U',E)),\quad f\mto f^\vee
\]
is an isomorphism of topological vector spaces.
Hence $W=\Phi^{-1}(V)$ for some open $0$-neighbourhood
$V\sub C^r(U_1,C^{\alpha'}(U',E))$.
By induction, there exists $g\in \LocPol_{\leq r}(U_1,C^{\alpha'}(U',E))$
with $\Phi(f)-g\in V$.
There are $\ell\in \N$, clopen subsets $R_1,\ldots, R_\ell\sub U_1$
and $c_{i,j}\in C^{\alpha'}(U',E)$ for
$i\in \{0,1,\ldots, r\}$ and
$j\in \{1,\ldots,\ell\}$
such that
\[
g(x)=\sum_{i=0}^r\sum_{j=1}^\ell x^i \one_{R_j}(x)c_{i,j},
\]
using the characteristic function $\one_{R_j}\colon U_1\to \K$.
Write $X^i \colon U_1\to\K$ for the function $x\mto x^i$.
By Lemma~\ref{laavoid},
there are open neighhbourhoods $Q_{i,j}\sub C^{\alpha'}(U',E)$ of $c_{i,j}$
such that
\[
\Phi(f) - \sum_{i=0}^r\sum_{j=1}^\ell X^i  \one_{R_j}b_{i,j}\in V,
\]
for all $b_{i,j}\in Q_{i,j}$.
By induction, we find $b_{i,j}\in Q_{i,j}\cap \LocPol_{\leq \alpha'}(U',E)$.
Then
\[
g\colon U\to E,\quad g(x,x'):=\sum_{i=0}^r\sum_{j=1}^\ell x^i\one_{R_j}(x)b_{i,j}(x')\quad\mbox{for $x\in U_1$, $x'\in U'$}
\]
is an element of $\LocPol_{\leq\alpha}(U,E)$.
Since $\Phi(g)=\sum_{i=0}^r\sum_{j=1}^\ell X^i \one_{R_j}b_{i,j}$,
we have $\Phi(f)-\Phi(g)\in V$ and thus $f-g\in W$.
This completes the proof.\,\Punkt
\section{Proof of Theorem C}
For $x\in \Z_p$ and $\nu\in\N_0$,
as usual we define
\[
\left(
\begin{array}{c}
x\\
\nu
\end{array}
\right):= \frac{x(x-1)\cdots (x-\nu+1)}{\nu!}
\]
(see, e.g., \cite[Definition 47.1]{Sch}).
This is a polynomial function in~$x$ of degree~$\nu$.
Since
$\left(
\begin{array}{c}
x\\
\nu
\end{array}
\right)\in \N_0$ if $x\in\N$, the density of $\N$ in $\Z_p$
and continuity of
$\left(
\begin{array}{c}
\sbull\\
\nu
\end{array}
\right)$
entail that
\begin{equation}\label{inzp}
\left(
\begin{array}{c}
x\\
\nu
\end{array}
\right)\in \Z_p\quad\mbox{for all $\,x\in \Z_p$.}
\end{equation}
For $n\in \N$, $\nu\in \N_0^n$ and $x=(x_1,\ldots, x_n)\in \Z_p^n$,
we consider the product
\[
\left(
\begin{array}{c}
x\\
\nu
\end{array}
\right)
:=
\left(
\begin{array}{c}
x_1\\
\nu_1
\end{array}
\right)
\cdot \ldots\cdot
\left(
\begin{array}{c}
x_n\\
\nu_n
\end{array}
\right)\in \Z_p
.
\]
Let $\cF_n$ be the set of all finite subsets
of $\N_0^n$.
We make $\cF_n$ a directed set by means of inclusion of sets,
$\leq:=\sub$.
\begin{la}\label{detmind}
Let $E$ be a topological vector space over $\Q_p$,
$n\in \N$ and $a_\nu\in E$ for $\nu\in \N_0^n$.
Assume that, for each $x\in \Z_p^n$,
the limit
\begin{equation}\label{26b}
f(x):=
\sum_{\nu\in \N_0^n}\left(
\begin{array}{c}
x\\
\nu
\end{array}
\right)a_\nu:=\lim_{\Phi\in \cF_n}
\sum_{\nu\in \Phi}\left(
\begin{array}{c}
x\\
\nu
\end{array}
\right)a_\nu
\end{equation}
of the net of finite partial sums exists.
Then the family $(a_\nu)_{\nu\in \N_0^n}$ is uniquely determined
by $f$; the coefficient $a_\nu$ can be recovered via
\begin{equation}\label{iscoeff}
a_\nu=\sum_{\mu\leq\nu}
(-1)^{|\nu|-|\mu|}
\left(
\begin{array}{c}
\nu\\
\mu
\end{array}
\right) f(\mu)
\end{equation}
$($where the summation is over all $\mu\in \N_0^n$ such that $\mu\leq\nu)$.
In particular, $a_\nu$ is contained in the additive subgroup of $(E,+)$ generated by
$f(\N_0^n)$.
\end{la}
\begin{proof}
We closely follow the treatment of the one-dimensional scalar-valued case
in the proof of \cite[Theorem~52.1]{Sch}.
For $j\in \{1,\ldots, n\}$, define the operator
$\delta_j\colon E^{\Z_p^n}\to E^{\Z_p^n}$
via
\[
(\delta_j f)(x):=f(x+e_j)-f(x),
\]
where $e_j=(0,\ldots,0,1,0,\ldots,0)$
with only a non-zero $j$-th entry.
Performing the calculations on \cite[pp.\,152--153]{Sch}
in each of the variables $x_n, x_{n-1},\ldots, x_1$ in turn,
we obtain
\begin{eqnarray*}
a_\nu & = & (\delta_1^{\nu_1} \cdots \delta_n^{\nu_n} f)(0,\ldots,0)\\
&= & \sum_{\mu_1=0}^{\nu_1}\cdots
\sum_{\mu_n=0}^{\nu_n}(-1)^{\nu_1-\mu_1}\cdots (-1)^{\nu_n-\mu_n}
f(\mu_1,\ldots,\mu_n),
\end{eqnarray*}
which can be written more compactly as (\ref{iscoeff}).
\end{proof}
\begin{defn}\label{dfnmahcoe}
If $E$ is a vector space over $\Q_p$
and $f\colon \Z_p^n\to E$ a function,
we take (\ref{iscoeff}) as the definition of $a_\nu\in E$
and call $a_\nu$ the \emph{$\nu$-th Mahler coefficient of~$f$}.
The series (net of finite partial sums)
\[
\sum_{\nu\in \N_0^n}
\left(
\begin{array}{c}
\sbull \\
\nu
\end{array}
\right)a_\nu
\]
of functions $\Z_p^n\to E$
is called the \emph{Mahler series of~$f$}.
\end{defn}
\begin{rem}\label{NEWrm1}
Let $f$ be as in Defintion~\ref{dfnmahcoe},
$a_\nu\in E$ be its $\nu$-th Mahler coefficient
and $\lambda\colon E\to F$ be a continuous linear map
to a topological $\Q_p$-vector space~$F$.
Applying $\lambda$ to (\ref{iscoeff}),
we see that $\lambda(a_\nu)$ is the $\nu$-th Mahler coefficient of the continuous
function $\lambda\circ f\colon \Z_p^n\to F$.
\end{rem}
\begin{defn}\label{defnc0}
Let $X$ be a set, $E$ be a locally convex space
over an ultrametric field $\K$
and $w\colon X\to \,]0,\infty[$
be a function.
We define $c_0(X,w,E)$
as the space of all $f\colon X\to E$
such that, for each ultrametric continuous
seminorm $q$ on~$E$ and $\ve>0$,
there exists a finite subset $\Phi\sub X$ such that
\[
(\forall x\in X\setminus \Phi)\quad w(x)q(f(x))<\ve.
\]
We endow $c_0(X,w,E)$
with the locally convex topology defined by the ultrametric seminorms
given by
\[
\|f\|_{w,q}:=\sup \{w(x)q(f(x))\colon x\in X\}\in [0,\infty[.
\]
If $w$ is the constant function~$1$,
we abbreviate $c_0(X,E):=c_0(X,1,E)$
and $\|.\|_{q,\infty}:=\|.\|_{1,q}$.
\end{defn}
\begin{la}\label{suppoc0}
Let $X$, $w$ and $E$ be as in Definition~{\rm\ref{defnc0}}
and $\lambda\colon E\to F$
be a continuous linear map to a locally convex space~$F$
over~$\K$.
Then $\lambda\circ f\in c_0(X,w,F)$
for each $f\in c_0(X,w,E)$, and the map
\[
c_0(X,w,\lambda)\colon c_0(X,w,E)\to
c_0(X,w,F),\quad f\mto \lambda\circ f
\]
is continuous and linear.
If $\lambda$ is a linear topological embedding,
then also $c_0(X,w,\lambda)$.
\end{la}
\begin{proof}
If $q$ is an ultrametric  continuous seminorm on~$F$, then
$q\circ \lambda$ is an ultrametric continuous
seminorm on~$E$.
If $\ve>0$, there exists a finite subset
$\Phi\sub X$ such that $\sup\{w(x)(q\circ\lambda)(f(x))\colon
x\in X\setminus \Phi\}<\ve$.
Since $(q\circ\lambda)(f(x))=q((\lambda\circ f)(x))$,
we see that $\lambda\circ f\in c_0(X,w,F)$.
Because
\[
\|\lambda\circ f\|_{w,q}=\|f\|_{w,q\circ\lambda}\leq
\|f\|_{w,q\circ\lambda},
\]
the linear map $c_0(X,w,\lambda)$ is continuous.\\[2.4mm]
If $\lambda$ is a topological embedding,
then $c_0(X,w,\lambda)$ is injective.
Hence $c_0(X,w,\lambda)$ will be an embedding
if the initial topology on $c_0(X,w,E)$
with respect to $c_0(X,w,\lambda)$
coincides with the given topology.
But this is a special case of the next lemma.
\end{proof}
\begin{la}\label{inic0}
Let $X$, $w$ and $E$ be as in Definition~{\rm\ref{defnc0}}
and assume that the topology on~$E$ is initial with respect to a family
$(\lambda_j)_{j\in J}$ of linear mappings $\lambda_j\colon E\to E_j$
to locally convex spaces~$E_j$.
Then the topology on
$c_0(X,w,E)$ is initial with respect to the family
$(c_0(X,w,\lambda_j))_{j\in J}$.
\end{la}
\begin{proof}
By hypothesis,
the locally convex vector topology on~$E$
can be defined by the seminorms $q\circ\lambda_j$,
for $j\in J$ and $q$ ranging through the ultrametric continuous
seminorms on~$E_j$.
The given locally convex topology $\cO$ on $c_0(X,w,E)$ is therefore
defined by the seminorms of the form $\|.\|_{w,q\circ \lambda_j}$.
On the other hand, the initial topology $\cT$ on $c_0(X,w,E)$
with respect to the maps $c_0(X,w,\lambda_j)$
is the locally convex vector topology defined by the ultrametric
seminorms $\|.\|_{w,q}\circ c_0(X,w,\lambda_j)$.
As these coincide with the seminorms $\|.\|_{w,q\circ \lambda_j}$
from above, $\cO=\cT$ follows.
\end{proof}
\begin{la}\label{c0props}
Let $X$, $w$ and $E$ be as in Definition~{\rm\ref{defnc0}}.
The locally convex space $c_0(X,w,E)$
has the following properties:
\begin{itemize}
\item[{\rm(a)}]
If $E$ is metrizable, then
$c_0(X,w,E)$
is metrizable.
\item[{\rm(b)}]
If $E$ is complete $($resp., sequentially complete$)$,
then also
$c_0(X,w,E)$
is complete $($resp., sequentially complete$)$.
\item[{\rm(c)}]
If $E$ is an ultrametric Banach space,
then also
$c_0(X,w,E)$.
\end{itemize}
\end{la}
\begin{proof}
(a) If $E$ is metrizable, then its vector topology is defined
by a countable set~$Q$ of ultrametric seminorms.
Then $\{\|.\|_{w,q}\colon q\in Q\}$
is a countable set of ultrametric seminorms
that defines the locally convex topology of $c_0(X,w,E)$,
and hence $c_0(X,w,E)$ is metrizable.\\[2.4mm]
(b) If $E$ is complete, let $(h_j)_{j\in J}$
be a Cauchy net in $c_0(X,w,E)$.
Since the point evaluation
$\ve_x\colon c_0(X,w,E)\to E$
is continuous and linear for $x\in X$, we see
that $(h_j(x))_{j\in J}$ is a Cauchy net in~$E$ and hence convergent to some
$h(x)\in E$. Given an ultrametric continuous seminorm~$q$ on~$E$
and $\ve>0$,
there exists $j_0\in J$ such that
$\|h_j-h_k\|_{w,q}\leq\ve$ for all $j,k\in J$ such that $j,k\geq j_0$.
For fixed $x\in X$, we therefore have
$w(x)q(h_j(x)-h_k(x))\leq\ve$ for all $j,k\in J$ such that $j,k\geq j_0$.
Passing to the limit in~$k$,
we see that
\begin{equation}\label{mayneed}
(\forall j \geq j_0)(\forall x\in X)\;\; w(x)q(h_j(x)-h(x))\leq\ve.
\end{equation}
Hence $w(x)q(h(x))\leq w(x)q(h_{j_0})+\ve$
and thus $h\in c_0(X,w,E)$, with $\|h\|_{w,q}\leq \|h_{j_0}\|_{w,q}+\ve$.
Now (\ref{mayneed}) can be read as $\|h_j-h\|_{w,q}\leq\ve$
for all $j\geq j_0$. Thus $h_j\to h$, showing the completeness
of $c_0(X,w, E)$.\\[2.4mm]
(c) If the ultrametric norm
$q:=\|.\|$ defines the vector topology on~$E$,
then $\|.\|_{w,q}$ is an ultrametric norm
that defines the vector topology on $c_0(X,w,E)$
and makes it an ultrametric Banach space
(by (b)).
\end{proof}
We abbreviate $c_0(X,\lambda):=c_0(X,1,\lambda)$.
\begin{la}\label{c0exp}
Let $X$ and $Y$ be sets,
$v\colon X\to \;]0,\infty[$
and  $w\colon Y\to \;]0,\infty[$
be functions and~$E$ be a locally convex space over
an ultrametric field~$\K$.
Define $v\tensor w
\colon X\times Y\to \,]0,\infty[$
via $(v\tensor w)(x,y):=
v(x)w(y)$.
Then $f^\vee(x):=f_x:=f(x,\sbull)\in c_0(Y,w,E)$ for
each $x\in X$ and $f\in c_0(X\times Y,v\tensor w, E)$.
Moreover, $f^\vee\in c_0(X,v,c_0(Y,w,E))$, and the map
\[
\Theta \colon c_0(X\times Y,v\tensor w,E)\to c_0(X,v,c_0(Y,w,E)),\quad f\mto f^\vee
\]
is an isomorphism of topological vector spaces with inverse
\[
\Psi\colon c_0(X,v,c_0(Y,w,E))\to c_0(X\times Y,v\tensor w,E),\quad g\mto g^\wedge,
\]
where $g^\wedge(x,y):=g(x)(y)$.
\end{la}
\begin{proof}
Let $q$ be an ultrametric continuous seminorm on~$E$ and $\ve>0$.
Fix $z\in X$. There exists a finite subset $\Phi\sub X\times Z$ such that
$v(x)w(y)q(f(x,y))<\min\{\ve,v(z)\ve\}$
for all $(x,y)\in (X\times Y)\setminus \Phi$.
After increasing $\Phi$, we may assume that $\Phi=\Phi_1\times \Phi_2$
with finite subsets $\Phi_1\sub X$, $\Phi_2\sub Y$.
For $y\in Y\setminus \Phi_2$,
we then have
$w(y)q(f(z,y))<\ve$,
and we deduce that $f_z\in c_0(Y,w,E)$.
Moreover, for all $x\in E\setminus \Phi_1$
and all $y\in Y$ we have
$v(x)w(y)q(f(x,y))<\ve$.
Passing to the supremum over~$y$,
we obtain $v(x)\|f_x\|_{w,q}\leq\ve$.
Hence $f^\vee\in c_0(X,v,c_0(Y,w,E))$ indeed.
Because
\begin{eqnarray*}
\|f\|_{v\tensor w,q}& = & \sup\{v(x)w(y)q(f(x,y))\colon x\in X,y\in Y\}\\
&=& \sup_{x\in X}v(x)\sup_{y\in Y}w(y)q(f(x,y))
=\sup_{x\in X}v(x)\|f_x\|_{w,q}=\|f^\vee\|_{v,\|.\|_{w,q}},
\end{eqnarray*}
the injective linear map $\Theta$ is a topological embedding.
To see that $\Theta$ is surjective, let $g\in c_0(X,v,c_0(Y,w,E))$.
Define $g^\wedge\colon X\times Y\to E$, $g^\wedge(x,y):=g(x)(y)$.
If we can show that
$g^\wedge \in c_0(X\times Y,v\tensor w,E)$,
then $\Theta(g^\wedge)=(g^\wedge)^\vee=g$
proves the surjectivity.
To this end, let $q$  and $\ve$ be as before.
Then there is a finite subset $\Phi_1\sub X$ such that
$v(x)\|g(x)\|_{w,q}<\ve$ for all $x\in X\setminus \Phi_1$.
For each $x\in \Phi_1$, there exists a finite subset $\Xi_x\sub Y$
such that $w(x)v(y)q(g(x)(y))<\ve$ for all $y\in Y\setminus \Xi_x$.
Set $\Phi_2:=\bigcup_{x\in \Phi_1}\Xi_x$
and $\Phi:=\Phi_1\times \Phi_2$.
If $(x,y)\in (X\times Y)\setminus \Phi$,
then $x\in X\setminus \Phi_1$ (in which case
$v(x)w(y)q(g(x)(y))\leq v(x)\|g(x)\|_{w,q}<\ve$)
or $x\in \Phi_1$ and $y\in Y\setminus \Phi_2$,
in which case $y\in Y\setminus \Xi_x$ and thus
$w(x)v(y)q(g(x)(y))<\ve$ as well.
Thus $g^\wedge\in c_0(X\times Y,v\tensor w,E)$ indeed.
\end{proof}
If $X$ is a compact space, $E$ a locally convex space over an ultrametric field
and $q$ an ultrametric continuous seminorm on~$E$, we obtain an ultrametric
continuous seminorm on $C(X,E)$ via
\[
\|f\|_{q,\infty}:=\sup\{q(f(x))\colon x\in X\}\in [0,\infty[\quad \mbox{for $\,f\in C(X,E)$.}
\]
\begin{la}\label{strongMahcts}
Let $E$ be a locally convex space over~$\Q_p$
and $f\colon \Z_p^n\to E$ be a continuous function,
with Mahler coefficients $a_\nu(f)$ for $\nu\in \N_0^n$.
Then the Mahler series of $f$ converges uniformly to~$f$
$($i.e., it converges to~$f$ in $C(\Z_p^n,E))$.
Further,
$(a_\nu(f))_{\nu\in \N_0^n}\in c_0(\N_0^n,E)$, and the map
\[
A_{n,E}:= (a_\nu)_{\nu\in \N_0^n}\colon C(\Z_p^n,E) \to c_0(\N_0^n,E),\quad f\mto (a_\nu(f))_{\nu\in \N_0^n}
\]
is a linear topological embedding with
\begin{equation}\label{likeisome}
\|A_{n,E}(f)\|_{q,\infty}=\|f\|_{q,\infty}
\end{equation}
for each continuous ultrametric seminorm $q$ on~$E$.
If $E$ is sequentially\linebreak
complete, then $A_{n,E}$ is an isomorphism
of topological vector spaces.
\end{la}
\begin{proof}
Let $\wt{E}$ be a completion of~$E$
with $E\sub \wt{E}$.
The proof is by induction on $n\in \N$.
If $n=1$, then one finds as in \cite[p.\,156, Exercise 52E(v)]{Sch}\footnote{Due to a misprint,
the exercise is labelled Exercice~52.D, like the previous one.}
that $q(a_\nu(f))\to 0$ as $\nu\to\infty$,
for each $f\in C(\Z_p,E)$ and continuous ultrametric seminorm~$q$ on~$E$.
Hence $A_{1,E}(f)\in c_0(\Z_p,E)$.
Since
\begin{equation}\label{reqest}
q\left(
\left(
\begin{array}{c}
x\\
\nu
\end{array}
\right)a_\nu(f)\right)\leq q(a_\nu(f)),
\end{equation}
we deduce that
the Mahler series of $f$ converges
uniformly in $C(\Z_p,\wt{E})$,
to a continuous function $g\colon \Z_p\to \wt{E}$.
As in the cited exercise, one finds that
$f(x)=g(x)$ for $x\in \N_0$
and hence for all $x\in \Z_p$, by continuity.
The Mahler series thus converges
uniformly to~$f$.
We now deduce from (\ref{reqest}) and (\ref{26b}) that
\[
\|f\|_{q,\infty}\leq \sup\{q(a_\nu(f))\colon \nu\in \N_0\}=\|A_{1,E}(f)\|_{q,\infty}.
\]
Conversely, (\ref{iscoeff}) entails that
$q(a_\nu(f))\leq \|f\|_{q,\infty}$ for each $\nu\in \N_0$.
Hence also $\|A_{1,E}(f)\|_{q,\infty}\leq
\|f\|_{q,\infty}$, and (\ref{likeisome}) follows.
It is clear that $A_{1,E}$ is linear,
and Lemma~\ref{detmind} implies that $A_{1,E}$ is injective
(noting that $f$ is given by its Mahler series,
as just shown). Summing up, $A_{1,E}$ is a linear topological embedding.\\[2.4mm]
If $E$ is sequentially complete, let
$b=(b_\nu)_{\nu\in \N_0}\in c_0(\N_0,E)$.
Then the limit $f(x):=
\sum_{\nu\in \N_0}\left(
\begin{array}{c}
x\\
\nu
\end{array}
\right)b_\nu$ exists in~$E$, uniformly in~$x$,
and hence defines a continuous function $f\colon \Z_p\to E$.
By Lemma~\ref{detmind}, $A_{1,E}(f)=(b_\nu)_{\nu\in \N_0}$.
Thus $A_{1,E}$ is also surjective and hence an isomorphism
of topological vector spaces.\\[2.4mm]
Induction step. If $n\geq 2$, abbreviate $F:=C(\Z_p^{n-1},E)$.
This space is sequentially complete if~$E$ is sequentially
complete, by Lemma~\ref{sammelsu}\,(e).
Now $f^\vee\colon \Z_p\to C(\Z_p^{n-1},E)=F$, $x\mto f_x:=f(x,\sbull)$
is continuous, for each $f\in C(\Z_p^n,E)$,
and the map
\[
\Phi\colon C(\Z_p^n,E)\to C(\Z_p,F),\quad f\mto f^\vee
\]
is an isomorphism of topological
vector spaces (see Proposition~\ref{ctsexp}).
By the case $n=1$, the map
\[
A_{1,F}\colon C(\Z_p,F)\to c_0(\N_0,F)
\]
is a linear topological embedding
(and an isomorphism if~$E$ and hence $F$ is sequentially complete).
Since $A_{n-1,E}$ is a linear topological embedding by induction
(and an isomorphism if $E$ is sequentially complete),
also the map
\[
c_0(\N_0,A_{n-1,E})\colon
c_0(\N_0,F)\to c_0(\N_0,c_0(\N_0^{n-1},E))
\]
is a topological embedding
(resp., an isomorphism), by Lemma~\ref{suppoc0}.
Finally,
\[
\Psi\colon c_0(\N_0,c_0(\N_0^{n-1},E))\to c_0(\N_0^n,E),\quad g\mto g^\wedge
\]
(with $g^\wedge(\nu_1,\ldots, \nu_n):=g(\nu_1)(\nu_2,\ldots, \nu_n)$)
is an isomorphism of topological vector spaces, by Lemma~\ref{c0exp}.
Hence $\Psi\circ c_0(\N_0,A_{n-1,E})\circ A_{1,F}\circ\Phi$
is a linear topological embedding
(and an isomorphism if~$E$
is sequentially complete).
We claim that
\begin{equation}\label{NewEQ}
\Psi\circ c_0(\N_0,A_{n-1,E})\circ A_{1,F}\circ\Phi = A_{n,E}.
\end{equation}
If this is true, then all assertions are verified,
except for (\ref{likeisome}).
However, (\ref{likeisome}) readily
follows from (\ref{iscoeff}) and (\ref{reqest}), as in the case $n=1$.\\[2.4mm]
To verify the claim, let $\nu=(\nu_1,\nu')$ with $\nu_1\in \N_0$ and $\nu'\in \N_0^{n-1}$.
For $\mu'\in \N_0^{n-1}$ such that $\mu'\leq\mu$,
the point evaluation
$\ve_{\mu'}\colon C(\N_0^{n-1},E)\to E$,
$g\mto g(\mu')$ is continuous linear.
Using Remark~\ref{NEWrm1}, we deduce that
\begin{eqnarray}
A_{1,F}(f^\vee)(\nu_1)(\mu')
&=&\ve_{\mu'}(A_{1,F}(f^\vee)(\nu_1))
\,=\,A_{1,E}((\ve_{\mu'}\circ f^\vee)(\nu_1)\notag\\
&=&A_{1,E}(f(\sbull,\mu'))(\nu_1)\notag\\
&=& \sum_{\mu_1=0}^{\nu_1}
(-1)^{\nu_1-\mu_1}
\left(
\begin{array}{c}
\nu_1\\
\mu_1
\end{array}
\right)
f(\mu_1,\mu').\label{usefina}
\end{eqnarray}
As a consequence,
\begin{eqnarray*}
\lefteqn{(\Psi\circ c_0(\N_0,A_{n-1,E})\circ A_{1,F}\circ\Phi)(f)(\nu) \,=\,
\Psi(A_{n-1,E}\circ A_{1,F}(f^\vee))(\nu)}\hspace*{7.3mm}\\
&=&(A_{n-1,E}\circ A_{1,F}(f^\vee))(\nu_1)(\nu')
\,=\, A_{n-1,E}(A_{1,F}(f^\vee)(\nu_1))(\nu')\\
&=&
\sum_{\mu'\leq\nu'}
(-1)^{|\nu'|-|\mu'|}
\left(
\begin{array}{c}
\nu'\\
\mu'
\end{array}
\right)
A_{1,F}(f^\vee)(\nu_1)(\mu')\\
&=&
\sum_{\mu\leq\nu}
(-1)^{|\nu|-|\mu|}
\left(
\begin{array}{c}
\nu\\
\mu
\end{array}
\right)
f(\mu)\,=\,
A_{n,E}(f)(\nu),
\end{eqnarray*}
using (\ref{usefina})
to obtain the penultimate equality.
\end{proof}
\begin{numba}\label{defnwalph}
For $\alpha\in \N_0^n$,
define $w_\alpha\colon \N_0^n\to \;]0,\infty[$ via
$w_\alpha(\nu):=\nu^\alpha:=\nu_1^{\alpha_1}\cdots
\nu_n^{\alpha_n}$.
\end{numba}
\begin{prop}\label{isweightedf}
In the situation of Theorem~C,
the map
\[
A^\alpha_E\colon C^\alpha(\Z_p^n,E)\to c_0(\N_0^n,w_\alpha,E)
\]
taking a function to its Mahler coefficients
is an isomorphism of topological vector spaces.
\end{prop}
{\bf Proof of Theorem~C and Proposition~\ref{isweightedf}.}
The final assertion of Theorem~C is a special case
of Corollary~\ref{fico} (see also Remark~\ref{forThmC}),
whence we omit its proof here.
The other assertions are proved by induction on $n\in \N$.\\[2.4mm]
The case $n=1$. Write $\alpha=(r)$ with $r\in \N_0$.
If $f\in C(\Z_p,E)$, let $(a_\nu(f))_{\nu\in \N_0}$
be the sequence of Mahler coefficients of~$f$.
Using ultrametric continuous seminorms
instead of $|.|$, we can repeat the proof of \cite[Theorem~54.1]{Sch}
(devoted the scalar-valued case)
and find that $f$ is $C^r$ if and only if $q(a_\nu(f))\nu^r\to 0$
as $\nu\to\infty$, for each ultrametric continuous seminorm~$q$
on~$E$. In particular, we can define a map
\[
A^r_E \colon C^r(\Z_p,E)\to c_0(\N,w_r,E),\quad f\mto (a_\nu(f))_{\nu\in \N_0}.
\]
It is clear that $A^r_E$ is linear, and it is injective
(by Lemma~\ref{detmind}).
To see that $A^r_E$ is also surjective
(and hence an isomorphism of vector spaces),
let $b=(b_\nu)_{\nu\in \N_0}\in c_0(\N,w_r,E)$.
Since $w_r\geq 1$, we have $b\in c_0(\N,E)$.
As~$E$ is assumed sequentially complete,
this entails the existence of a continuous function $f\colon  \Z_p\to E$
with Mahler coefficients $b_\nu$ (Lemma~\ref{strongMahcts}).
But then $f\in C^r(\Z_p,E)$ by the characterization of $C^r$-functions
just obtained, and now $b=A^r_E(f)$ by construction of~$f$.\\[2.4mm]
To see that $A^1_E$ is an isomorphism of topological vector spaces,
it only remains to show that $A^1_E$ is a topological embedding.
We may assume that $E$ is complete.\footnote{If $\wt{E}$ is a completion of~$E$ containing~$E$,
then $C^r(\Z_p,E)$ embeds into $C^r(\Z_p,\wt{E})$ (see Lemma~\ref{alphpushf})
and $c_0(\N_0,w_r,E)$ embeds into $c_0(\N_0,w_r,E)$ (Lemma~\ref{suppoc0}),
whence $A^1_E$ will be an embedding if $E^1_{\wt{E}}$ is so.}
Thus $E=\pl\,E_q$ for some ultrametric Banach spaces~$E_q$, with
limit maps $\pi_q\colon E\to E_q$.
Then the topology on $C^r(\Z_p,E)$ is initial with respect to the maps
$C^r(\Z_p,\pi_q)$ (Lemma~\ref{alphinit})
and the topology on
$c_0(\N_0,w_r,E)$ is initial with respect to
the maps $c_0(\N_0,w_r,\pi_q)$ (Lemma~\ref{inic0}).
Hence, using the transitivity of initial topologies,
the initial topology $\cT$ on $C^r(\Z_p,E)$
with respect to the map $A^1_E$
is also initial with respect to the maps
$c_0(\N_0,w_r,\pi_q)\circ A^1_E=A^1_{E_q}\circ C^r(\Z_p,\pi_q)$
(where the equality of both sides comes from Remark~\ref{NEWrm1}).
If we can show that each $A^1_{E_q}$ is an embedding,
then the topology on $C^r(\Z_p,E_q)$ is initial
with respect to $A^1_{E_q}$ and hence
the compact-open $C^r$-topology~$\cO$ on $C^r(\Z_p,E)$
is initial with respect to the maps $A^1_{E_q}\circ C^r(\Z_p,\pi_q)$.
It therefore coincides with~$\cT$, whence the bijective\linebreak
map $A^1_E$ is
a homeomorphism (and hence an isomorphism of topological\linebreak
vector spaces).\\[2.4mm]
By the preceding, it suffices to show that $A^1_E$ is an isomorphism
of topological vector spaces if~$E$ is an ultrametric Banach space.
In this case, $C^r(\Z_p,E)$ is an ultrametric Banach space
(Lemma~\ref{samsu}\,(f)).
Also $c_0(\N_0,w_r,E)$ is an ultrametric Banach space (Lemma~\ref{c0props}).
Since $A^1_E$ is an isomorphism of vector spaces,
the Closed Graph Theorem
(see, e.g., \cite[Proposition~8.5]{PSn}) implies
that $A^1_E$ will be an isomorphism of topological vector spaces
if its graph
\[
\Gamma:=\{(f,A^1_E(f))\colon f\in C^r(\Z_p,E)\}
\]
is closed in $C^r(\Z_p,E)\times c_0(\N_0,w_r,E)$.
To verify the latter, let $(f_i)_{i\in \N}$ be a sequence in
$C^r(\Z_p,E)$ which converges to some $f\in
C^r(\Z_p,E)$, such that $A^1_E(f_i)$ converges
in $c_0(\N_0,w_r,E)$
to some $b=(b_\nu)_{\nu\in \N_0}\in c_0(\N_0,w_r,E)$.
The topology on $c_0(\N_0,w_r,E)$
being finer than the topology of pointwise convergence,
the point evaluation
\[
\ve_\nu\colon c_0(\N_0,w_r,E)\to E,\quad (c_\mu)_{\mu\in\N_0}\mto c_\nu
\]
is continuous, for each $\nu\in \N_0$.
Hence $A^1_E(f_i)\to b$ implies $a_\nu(f_i)\to b_\nu$.
But $a_\nu(f_i)\to a_\nu(f)$,
as is clear from (\ref{iscoeff}).
Hence $f$ has the Mahler coefficients $b_\nu$
and thus $(f,b)=(f,A^1_E(f))\in \Gamma$, proving that $\Gamma$
is closed.\\[2.4mm]
Induction step. If $n\geq 2$, write $\alpha=(\alpha_1,\alpha')$ with $\alpha_1\in \N_0$ and
$\alpha' \in \N_0^{n-1}$.
Since $E$ is sequentially complete and $\Q_p$ is metrizable,
$H:=C^{\alpha'}(\Z_p^{n-1},E)$ is sequentially complete (Lemma~\ref{samsu}\,(b)).
The map
\[
\Phi_\alpha\colon C^\alpha(\Z_p^n,E)\to C^{\alpha_1}(\Z_p,C^{\alpha'}(\Z_p^{n-1},E)),
\quad f\mto f^\vee
\]
is an isomorphism of topological vector spaces, by Theorem~A.
Since~$H$ is sequentially complete,
$A^{\alpha_1}_H \colon C^{\alpha_1}(\Z_p,C^{\alpha'}(\Z_p^{n-1},E))\to
c_0(\N_0,w_{\alpha_1},C^{\alpha'}(\Z_p^{n-1},E))$ is an isomorphism of topological
vector spaces (by the case $n=1$).
Because~$E$ is sequentially complete, the map
$A^{\alpha'}_E\colon C^{\alpha'}(\Z_p^{n-1},E)\to c_0(\N_0^{n-1},w_{\alpha'},E)$
is an isomorphism of topological vector spaces, by induction.
Finally, the map $\Psi_\alpha\colon c_0(\N_0,w_{\alpha_1},c_0(\N_0^{n-1},w_{\alpha'},E))\to
c_0(\N_0^n,w_\alpha,E)$, $h\mto h^\wedge$
is an isomorphism of topological vector spaces, by Lemma~\ref{c0exp}
(using that $w_{\alpha_1}\tensor w_{\alpha'}=w_\alpha$).
Hence also the map
\[
\Theta_\alpha:=\Psi_\alpha\circ c_0(\N_0,w_{\alpha_1},A^{\alpha'}_E)\circ A^{\alpha_1}_H \circ \Phi_\alpha\colon
C^\alpha(\Z_p^n,E)\to c_0(\N_0^n,w_\alpha,E)
\]
is an isomorphism of topological vector spaces.
Note that $\Theta_\alpha$ is a restriction of the composition
$\Theta_0=A_{n,E}$ considered in~(\ref{NewEQ}).
Hence $\Theta_\alpha(f)$ is the family
of Mahler coefficients of~$f$, if $f\in C^\alpha(\Z_p^n,E)$.
Thus $(a_\nu(f))_{\nu\in \N_0^n}=\Theta_\alpha(f)\in c_0(\N_0^n, w_\alpha,E)$
in this case, i.e., (\ref{decaycond}) holds.
Conversely, assume that $f\in C(\Z_p^n,E)$
and (\ref{decaycond}) is satisfied, i.e.,
$b:=(a_\nu(f))_{\nu\in \N_0^n}\in c_0(\N_0^n,w_\alpha, E)$.
Then $g:=\Theta_\alpha^{-1}(b)\in C^\alpha(\Z_p^n,E)$.
Since $f$ and $g$ have the same Mahler coefficients, $f=g\in C^\alpha(\Z_p^n,E)$
follows.\,\Punkt
\section{{\boldmath $C^\alpha$}-maps on subsets of {\boldmath$\K^{n_1}\times\cdots\times \K^{n_\ell}$}}
In this section,
$\K$ is a topological field and
$E$ a topological $\K$-vector space.
\begin{defn}\label{defmocom}
Let $n_1,\ldots, n_\ell\in \N$,
$n:=n_1+\cdots+n_\ell$,
$\alpha\in (\N_0\cup\{\infty\})^\ell$
and $U\sub \K^n$ be a locally cartesian subset.
We let $C^\alpha(U,E)$ be the space of
all $C^\alpha$-maps $f\colon U\to E$
(as defined before Theorem~D).
Thus
\begin{equation}\label{hencprl1}
C^\alpha(U,E)=\bigcap_{\beta\in N_\alpha}C^\beta(U,E),
\end{equation}
where $N_\alpha$ is the set of all $\beta=(\beta_1,\ldots,\beta_\ell)\in \N_0^n=\N_0^{n_1}\times\cdots\times\N_0^{n_\ell}$
such that $|\beta_j|\leq\alpha_j$ for all $j\in \{1,\ldots, \ell\}$.
We endow $C^\alpha(U,E)$ with the initial topology $\cO$ with respect to the inclusion maps
\[
\iota_\beta\colon C^\alpha(U,E)\to C^\beta(U,E),
\]
using the compact-open $C^\beta$-topology on the right-hand side.
We call $\cO$ the compact-open $C^\alpha$-topology on
$C^\alpha(U,E)$.
\end{defn}
If $n_j=1$ for each $j$, then the new definition of the compact-open $C^\alpha$-topology on
$C^\alpha(U,E)$ coincides with the earlier one (cf.\ Lemma~\ref{pllower}).\\[2.4mm]
Taking $\ell=1$, we have $\alpha=(r)$ with some $r\in \N_0\cup\{\infty\}$.
Then the $C^{(r)}$-maps are precisely the $C^r$-maps,
defined as follows:
\begin{defn}
Let $\K$ be a topological field,
$U\sub \K^n$ be a locally cartesian subset, $E$
a topological $\K$-vector space and $r\in \N_0\cup\{\infty\}$.
A map $f\colon U\to E$ is called $C^r$ if it is $C^\alpha$
for all $\alpha\in \N_0^n$ such that $|\alpha|\leq r$.
\end{defn}
\begin{rem}
In the cases when $U$ is open or
$U$ is cartesian, this corresponds to the
Definition of $C^r_{SDS}$-maps given
by the author in \cite{CMP} (based on ideas from \cite{Sch} and \cite{DSm}).
In this case, the compact-open $C^r$-topology
(without\linebreak
that name)
was already introduced in the
preprint version
of \cite{CMP} (see\linebreak
Definition~B.1 in {\tt arXiv:math/0609041v1}),
and (for $U$ open) shown to\linebreak
coincide
with the topology on $C^r(U,E)$ introduced earlier
in~\cite{ZOO}\footnote{The latter is even defined if $U$ is an open set in
an arbitrary topological vector space.}
(see\linebreak
Theorem~B.2 in the preprint version of~\cite{CMP}).
The step to locally cartesian sets was performed in~\cite{NDI}.
\end{rem}
\begin{la}\label{multaclos}
The mapping
\begin{equation}\label{againclim}
\iota:=(\iota_\beta)_{\beta\in N_\beta} \colon
C^\alpha(U,E)\to \prod_{\beta\in N_\beta} C^\beta(U,E)
\end{equation}
is a linear topological embedding with closed image.
\end{la}
\begin{proof}
Because the topology on $P:=\prod_{\beta\in N_\beta} C^\beta(U,E)$
is initial with the respect to the projections  $\pr_j\colon P\to C^\beta(U,E)$,
the initial topology~$\cT$ on $C^\alpha(U,E)$ with respect to~$\iota$
coincides with the initial topology with respect to
the maps $\pr_j\circ\iota=\iota_\beta$ (by transitivity of initial topologies),
and hence with~$\cO$.
Thus, being also injective, $\iota$ is a topological embedding.
The linearity is clear. For each $\beta\in N_\alpha$,
the inclusion map $\iota_{0,\beta}\colon C^\beta(U,E)\to C(U,E)$
is continuous.
Since
\begin{eqnarray*}
\im(\iota) & = & \{ (f_\eta)_{\eta\in N_\alpha} \colon (\forall \beta,\gamma\in N_\alpha)\, f_\beta=f_\gamma\}\\
&=& \bigcap_{\beta,\gamma\in N_\alpha}\{f\in P\colon (\iota_{0,\beta}\circ\pr_\beta)(f)=
(\iota_{0,\gamma}\circ\pr_\gamma)(f)\},
\end{eqnarray*}
$\im(\iota)$ is closed.
\end{proof}
\begin{rem}\label{gentypeCalph}
Let $\alpha\in (\N_0\cup\{\infty\})^\ell$
and $U\sub \K^n$ be a locally cartesian subset with $n=n_1+\cdots+ n_\ell$.
It is essential for the proofs of Propositions~\ref{betterdensloc} and \ref{denseglobpo}
that various results from Section~\ref{secthetop}
remain valid for the more general $C^\alpha$-maps on $U$
in the sense of Definition~\ref{defmocom}:
\begin{itemize}
\item[(a)]
Lemma~\ref{basicres} remains valid.\\[2.4mm]
[For $\beta\in N_\alpha$, let $\iota_\beta\colon C^\alpha(U,E)\to C^\beta(U,E)$
and $\iota^S_\beta\colon C^\alpha(S,E)\to C^\beta(S,E)$
be the inclusion maps and $\rho_\beta\colon C^\beta(U,E)\to C^\beta(S,E)$
be the restriction map, which is continuous linear by Lemma~\ref{basicres}.
Then $\iota^S_\beta\circ \rho=\rho_\beta\circ \iota_\beta$
is continuous and hence also $\rho$.]
\item[(b)]
Lemma~\ref{covarbop} remains valid.\\[2.4mm]
[Let $\rho_a$ be as in the lemma, $\rho_{\beta,a}\colon C^\beta(U,E)\to C^\beta(U_a,E)$
be the restriction map and
$\iota_{\beta,a}\colon C^\alpha(U_a,E)\to C^\beta(U,E)$ the inclusion map
for $\beta\in N_\alpha$, $a\in A$.
As a consequence of Lemma~\ref{covarbop},
the compact-open $C^\alpha$-topology on $C^\alpha(U,E)$
is initial with respect to the maps $\rho_{\beta,a}\circ \iota_\beta=\iota_{\beta,a}\circ \rho_a$
and hence also with respect to the maps $\rho_a$.]
\item[(c)]
Lemma~\ref{pllower} carries over in the sense that\vspace{-1.5mm}
$C^\alpha(U,E)=\pl_{\beta\in B_\alpha}\, C^\beta(U,E)$,
where $B_\alpha$ is the set of all $\beta\in \N_0^\ell$
with $\beta\leq\alpha$.\\[2.4mm]
[Consider the inclusion maps $\iota_\gamma\colon C^\alpha(U,E)\to
C^\gamma(U,E)$ for $\gamma\in N_\alpha$,
as well as the inclusion maps
$\iota_{\beta,\gamma}\colon C^\gamma(U,E)\to C^\beta(U,E)$
for $\gamma\in N_\alpha$ and $\beta\in B_\gamma$.
We have $C^\alpha(U,E)=\bigcap_{\beta\in B_\alpha}\, C^\beta(U,E)$,
and the initial topology on $C^\alpha(U,E)$
with respect to the inclusion maps
$\phi_\beta\colon C^\alpha(U,E)\to
C^\beta(U,E)$ for $\beta\in B_\alpha$
coincides with the initial topology with
respect to the maps
$\iota_{\gamma,\beta}\circ \phi_\beta=\iota_\gamma$,
for $\beta\in B_\alpha$ and $\gamma\in N_\beta\sub N_\alpha$.
Since $\bigcup_{\beta\in B_\alpha}N_\beta=N_\alpha$,
this is the compact-open $C^\alpha$-topology.]
\item[(d)]
Lemma~\ref{alphpushf} remains valid.\\[2.4mm]
[For $\beta\in N_\alpha$,
consider the inclusion maps
$\iota^E_\beta\colon C^\alpha(U,E)\to C^\beta(U,E)$
and $\iota^F_\beta \colon C^\alpha(U,F)\to C^\beta(U,F)$.
Since $\iota^F_\beta\circ C^\alpha(U,\lambda)=C^\beta(U,\lambda)\circ\iota^E_\beta$ is continuous by Lemma~\ref{alphpushf},
$C^\alpha(U,\lambda)$ is continuous.
If $\lambda$ is an embedding, then the topology on $C^\beta(U,E)$ is initial
with respect to $C^\beta(U,\lambda)$,
whence the compact-open $C^\alpha$-topology on $C^\alpha(U,E)$
is initial with respect to the maps $C^\beta(U,\lambda)\circ \iota^E_\beta=\iota^F_\beta\circ
C^\alpha(U,\lambda)$,
which is also initial with respect to $C^\alpha(U,\lambda)$.
The latter map being also injective, it is a topological embedding.]
\item[(e)]
Lemma~\ref{alphprod} remains valid.\\[2.4mm]
[For $\beta\in N_\alpha$,
consider the inclusion maps
$\iota^E_\beta\colon C^\alpha(U,E)\to C^\beta(U,E)$
and $\iota^{E_j}_\beta \colon C^\alpha(U,E_j)\to C^\beta(U,E_j)$.
The initial topology on $C^\alpha(U,E)$
with respect to the maps
$C^\alpha(U,\pr_j)$ coincides with the initial topology with respect to
the maps $\iota^{E_j}_\beta\circ
C^\alpha(U,\pr_j)=C^\beta(U,\pr_j)\circ \iota^E_{\beta}$.
In view of Lemma~\ref{alphprod},
it is therefore also initial with
respect to the maps
$\iota^E_{\beta}$, and hence coincides with the compact-open
$C^\alpha$-topology.]
\item[(f)]
Lemma~\ref{plinim} remains valid.\\[2.4mm]
[Using (d) and (e) from this remark
instead of Lemmas~\ref{alphvsub} and \ref{alphprod},
we can repeat the proof of Lemma~\ref{plinim}.]
\item[(g)]
Lemma~\ref{laavoid} remains valid.\\[2.4mm]
[Let $\theta$ be as in the lemma and
$\theta_\beta\colon E\to C^\beta(U,E)$ be the corresponding map
for $\beta\in N_\alpha$. Then $\iota_\beta\circ \theta=\theta_\beta$
is continuous, by Lemma~\ref{laavoid}.
Hence $\theta$ is continuous.]
\end{itemize}
\end{rem}
%
%
%
%
\begin{la}\label{onesttwost}
Let $U$ be a locally cartesian subset of $\K^{n_1}\times\cdots\times \K^{n_\ell}=\K^n$
with $n_1,\ldots, n_\ell\in \N$ and $n=n_1+\cdots+ n_\ell$,
and $\alpha\in (\N_0\cup\{\infty\})^\ell$.
Let $\cV$ be a cover of $U$ by cartesian open subsets $V\sub U$,
and $\cU$ be a basis of open $0$-neighbourhoods of~$E$.
Then a basis of open $0$-neighbourhoods in $C^\alpha(U,E)$
is given by finite intersections of sets of the form
\begin{equation}\label{explsbase}
\{f\in C^\alpha(U,E)\colon (f|_V)^{<\beta>}\in \lfloor K,Q\rfloor\}
\end{equation}
with $V\in \cV$, $Q\in \cU$, $\beta\in N_\alpha$
and compact subsets $K\sub V^{<\beta>}$ of the form
$K=K_1^{1+\beta_1}\times\cdots\times K_n^{1+\beta_n}$
with compact subsets $K_1,\ldots, K_n\sub \K$.
\end{la}
\begin{proof}
Step 1.
Let $V=V_1\times\cdots\times V_n$
be an open cartesian subset of~$U$
and $\beta\in N_\alpha$.
For $i\in \{1,\ldots, n\}$, let
$\pi_i\colon \K^{1+\beta_1}\times\cdots\times \K^{1+\beta_n}\to \K^{1+\beta_i}$
be the projection onto the $i$-th component.
For $j\in \{1,\ldots, 1+\beta_i\}$, let
$\pi_{j,i}\colon \K^{1+\beta_i}\to \K$
be the projection onto the $j$-th component.
If $K$ is a compact subset of
$V^{<\beta>}=V_1^{1+\beta_1}\times \cdots\times V_n^{1+\beta_n}$,
then also the larger set $K_1^{1+\beta_1}\times \cdots\times K_n^{1+\beta_n}\sub
V^{<\beta>}$ is compact, where $K_i:=\bigcup_{j=1}^{1+\beta_i}\pi_{j,i}(\pi_i(K))$.\\[2.4mm]
Step 2.
In view of Step~1 and Lemma~\ref{sammelsu}\,(b),
the sets $\lfloor K,Q\rfloor$
with $Q\in \cU$ and $K\sub V^{<\beta>}$ a compact
set of the form
$K_1^{1+\beta_1}\times \cdots\times K_n^{1+\beta_n}$
form a basis $\cU_{\beta,V}$ of open $0$-neighbourhoods
in $C(V^{<\beta>},E)$.\\[2.4mm]
Step~3.
As a consequence of Definition~\ref{defmocom} and Lemma~\ref{covcart},
the topology on $C^\alpha(U,E)$ is initial with respect
to the maps $\Delta_{\beta,V}\circ \iota_\beta$ with $\beta\in N_\alpha$
and $V\in \cV$, whence the injective map
\[
\Delta\colon C^\alpha(U,E)\to \prod_{\beta\in N_\alpha, V\in \cV} C(V^{<\beta>},E),\quad
f\mto (\Delta_{\beta,V}(f))_{\beta\in N_\alpha, V\in \cV}
\]
is a topological embedding.
A basis of open $0$-neighbourhoods of the direct product
is given by sets of the form
$W=\prod_{\beta,V}W_{\beta,V}$, where
$W_{\beta,V}=C(V^{<\beta>},E)$
for all but finitely
many $(\beta,V)$, and the remaining ones
are in the basis $\cU_{\beta,V}$ of open $0$-neighbourhoods in 
$C(V^{<\beta>},E)$.
Because $\Delta$ is a topological
embedding, it follows that the pre-images $\Delta^{-1}(W)$
form a basis of $0$-neighbourhoods in $C^\alpha(U,E)$.
These are the finite intersections described in the lemma.
\end{proof}
\section{Proof of Theorem D}
For $\gamma\in N_\alpha$
and $\eta\in N_\beta$,
the map
\[
\Phi_{\gamma,\eta}\colon C^{(\gamma,\eta)}(U\times V,E)\to C^\gamma(U,C^\eta(V,E)),\quad
f\mto f^\vee
\]
is a linear topological embedding, by Theorem~A.
Because the topology on $C^\beta(V,E)$ is initial with respect
to the inclusion maps
\[
\iota_\eta\colon C^\beta(V,E)\to C^\eta(V,E),
\]
the topology on $C^\gamma(U,C^\eta(V,E))$ is initial with respect to the
(inclusion) maps
\[
C^\gamma(U,\iota_\eta)\colon C^\gamma(U,C^\beta(V,E))\to C^\gamma(U,C^\eta(V,E))
\]
(see Lemma~\ref{alphinit}).
By definition, the topology on $C^{(\alpha,\beta)}(U\times V,E)$
is initial with respect to the inclusion maps
$\lambda_{\gamma,\eta}\colon C^{(\alpha,\beta)}(U\times V,E)\to
C^{(\gamma,\eta)}(U\times V,E)$, for
$(\gamma,\eta)\in N_{(\alpha,\beta)}=N_\alpha\times N_\beta$.
Likewise, the topology on
$C^\alpha(U,C^\beta(V,E))$ is initial with respect to the inclusion maps
\[
\Lambda_\gamma\colon C^\alpha(U,C^\beta(V,E))\to C^\gamma(U,C^\beta(V,E))\quad
\mbox{for $\,\gamma\in N_\alpha$.}
\]
Let $f\in C^{(\alpha,\beta)}(U\times V,E)$
and $x\in U$.
Let $\gamma\in N_\alpha$.
For each $\eta\in N_\beta$, the map $f$ is $C^{(\gamma,\eta)}$,
whence $f_x\in C^\eta(V,E)$,
by Theorem~A. Hence $f_x$ is $C^\beta$.
Since $f$ is $C^{(\gamma,\eta)}$, the map
$f^\vee$ is $C^\gamma$
as a map to $C^\eta(V,E)$ (by Theorem~A).
Considering now $f^\vee$ as a map to $C^\beta(V,E)$,
this means that $\iota_\eta\circ f^\vee$ is $C^\gamma$
for each $\eta\in N_\beta$.
Using Lemmas~\ref{multaclos}, \ref{inprodu} and \ref{difinsub},
this implies that $f^\vee\colon U\to C^\beta(V,E)$ is~$C^\gamma$.
Hence, because $\gamma\in N_\alpha$ was arbitrary,
$f^\vee\colon U\to C^\beta(V,E)$ is $C^\alpha$.\\[2.4mm]
It is clear that $\Phi$ is injective and linear.
Let $\cT$ be the initial topology on $C^\alpha(U,C^\beta(V,E))$
with respect to~$\Phi$.
By the transitivity of initial topologies,
this topology coincides with the initial topology with respect to the maps
$\Lambda_\gamma\circ \Phi$, for $\gamma\in N_\alpha$.
Again by transitivity, this topology coincides with the initial topology
with respect to the maps
$C^\gamma(U,\iota_\eta)\circ \Lambda_\gamma\circ\Phi$,
for
$\gamma\in N_\alpha$, $\eta\in N_\beta$.
But $C^\gamma(U,\iota_\eta)\circ \Lambda_\gamma\circ\Phi
=\Phi_{\gamma,\eta}\circ\lambda_{\gamma,\eta}$.
Hence~$\cT$ coincides with the initial topology
with respect to the maps $\Phi_{\gamma,\eta}\circ\lambda_{\gamma,\eta}$.
By transitivity of initial topologies, this topology
is initial with respect to the maps $\lambda_{\gamma,\eta}$
and hence coincides with the compact-open $C^{(\alpha,\beta)}$-topology on
$C^{(\alpha,\beta)}(U\times V,E)$ Thus $\Phi$ is a linear topological embedding.\\[2.4mm]
If $\K$ is metrizable and $g\in C^\alpha(U,C^\beta(V,E))$,
define $g^\wedge\colon U\times V\to E$, $g^\wedge(x,y):=g(x)(y)$.
Let $\gamma\in N_\alpha$ and $\eta\in N_\beta$.
Then $g$ is a $C^\gamma$-map to $C^\beta(V,E)$
and also (because the inclusion map $\iota_\eta$ is continuous linear)
a $C^\gamma$-map to $C^\eta(V,E)$. Hence $g^\wedge=\Phi_{\gamma,\eta}^{-1}(g)$
is $C^{(\gamma,\eta)}$. Thus $g^\wedge$ is $C^{(\alpha,\beta)}$
and since $g=\Phi(g^\wedge)$, the linear embedding $\Phi$ is surjective
and hence an isomorphism of topological vector spaces.\,\Punkt
\section{Further extensions and consequences}
We collect consequences and generalizations
of the previous results on Mahler expansions.
Afterwards, we return to the approximation
of functions by locally polynomial functions.
We also obtain results concerning approximation by polynomial functions.
Certain weighted $c_0$-spaces will be used:
\begin{defn}\label{defwset}
Let $X$ be a set, $\cW$ be a non-empty set of functions\linebreak
$w\colon X\to \,]0,\infty[$ and $E$ be a locally convex space over an
ultrametric field~$\K$.
We set
\[
c_0(X,\cW,E)
:=\bigcap_{w\in \cW}c_0(X,w,E)
\]
and endow this space with the locally convex vector topology~$\cO$
defined by the set of ultrametric seminorms $\|.\|_{w,q}$,
for $w\in \cW$ and continuous ultrametric seminorms $q$ on~$E$.
Thus $\cO$ is the initial topology with respect to the inclusion maps
\begin{equation}\label{groLa}
\Lambda_w\colon c_0(X,\cW,E)\to c_0(X,w,E),\quad\mbox{for $\,w\in \cW$.}
\end{equation}
\end{defn}
\begin{prop}\label{Mahbetter}
Let $E$ be a sequentially complete locally convex $p$-adic vector space,
$n_1,\ldots, n_\ell\in \N$,
$\alpha\in (\N_0\cup\{\infty\})^\ell$,
and $N_\alpha$ be the set of all
$\beta=(\beta_1,\ldots,\beta_\ell)\in \N_0^{n_1}\times \cdots\times \N_0^{n_\ell}=\N_0^n$
such that $|\beta_j|\leq\alpha_j$ for all $j\in \{1,\ldots, \ell\}$.
Let $f\colon (\Z_p)^n\to E$ be a continuous function
and $a_\nu(f)\in E$ be the Mahler coefficients of~$f$.
Then $f$ is $C^\alpha$ if and only if
\begin{equation}\label{decaycond2}
q(a_\nu(f)) \nu^\beta\to 0\quad\mbox{as $\,|\nu|\to\infty$}
\end{equation}
for each $\beta\in N_\alpha$ and ultrametric continuous seminorm $q$ on~$E$.
Moreover, the map
\[
A^\alpha_E\colon C^\alpha(\Z_p^n,E)\to c_0(\N_0^n,\cW_\alpha,E),\quad f\mto (a_\nu(f))_{\nu\in \N_0^n}
\]
is an isomorphism of topological vector spaces,
for $\cW_\alpha:=\{w_\beta\colon \beta\in N_\alpha\}$
with $w_\beta$ as in {\rm\ref{defnwalph}}.
\end{prop}
\begin{proof}
The continuous map $f$ is $C^\alpha$
if and only if it is $C^\beta$ for each $\beta\in N_\alpha$
(see Definition~\ref{defmocom}).
By Theorem~A, this holds if and only if
$(a_\nu(f))_{\nu\in \N_0^n}\in c_0(\N_0^n,w_\beta,E)$
for each $\beta\in N_\alpha$.
By Definition~\ref{defwset},
this condition is equivalent to $(a_\nu(f))_{\nu\in \N_0^n}\in c_0(\N_0^n,\cW_\alpha,E)$.
By Proposition~\ref{detmind}, the map $A^\alpha_E$ is injective.
To verify that $A^\alpha_E$ is surjective, let
$b=(b_\nu)_{\nu\in \N_0^n}\in c_0(\N_0^n,\cW_\alpha,E)$.
Then $b\in c_0(\N_0^n,w_\beta,E)$
for each $\beta\in N_\alpha$.
By Proposition~\ref{isweightedf}, we find $f_\beta\in C^\beta(\Z_p^n,E)$
with Mahler coefficients~$b_\nu$.
The uniqueness in Proposition~\ref{detmind}
entails that $f:=f_\beta$ is independent of~$\beta$.
Hence $f\in \bigcap_{\beta\in N_\alpha} C^\beta(\Z_p^n,E)=C^\alpha(\Z_p^n,E)$,
and $A^\alpha_E(f)=b$ by construction. Hence $A^\alpha_E$ is surjective
and hence an isomorphism of vector spaces.
For $\beta\in N_\alpha$, let $\iota_\beta\colon C^\alpha(\Z_p^n,E)\to C^\beta(\Z_p^n,E)$
be the inclusion map. By Proposition~\ref{isweightedf}, the topology on $C^\beta(\Z_p^n,E)$ is initial
with respect to the map
$A^\beta_E\colon C^\beta(\Z_p\ n,E)\to c_0(\N_0,w_\beta,E)$.
Hence the compact-open $C^\alpha$-topology~$\cO$ on $C^\alpha(\Z_p^n,E)$
is initial with respect to the maps $A^\beta_E\circ \iota_\beta$,
for $\beta\in N_\alpha$.
Let $\cT$ be the topology on $C^\alpha(\Z_p^n,E)$ which is initial with respect to~$A^\alpha_E$.
By transitivity of initial topologies, $\cT$ is also initial with respect to the maps
$\Lambda_{w_\beta}\circ A^\alpha_E=A^\beta_E\circ \iota_\beta$
(using notation as in~(\ref{groLa})),
and hence coincides with~$\cO$.
\end{proof}
\begin{cor}\label{fico}
In the situation of Proposition~{\rm\ref{Mahbetter}},
let $f\colon (\Z_p)^n\to E$ be a continuous function
and $a_\nu(f)\in E$ be the Mahler coefficients of~$f$.
Let $N_\alpha'$ be
the set of all $\beta=(\beta_1,\ldots,\beta_\ell)\in N_\alpha$ such that,
for each $i\in \{1,\ldots,\ell\}$, the $n_i$-tupel
$\beta_i\in \N_0^{n_i}$ has at most one
non-zero component.
Then $f$ is $C^\alpha$ if and only if
\begin{equation}\label{decaycond3}
q(a_\nu(f)) \nu^\beta\to 0\quad\mbox{as $\,|\nu|\to\infty$}
\end{equation}
for each $\beta\in N_\alpha'$ and ultrametric continuous seminorm~$q$
on~$E$.
Moreover, the map
\[
C^\alpha(\Z_p^n,E)\to c_0(\N_0^n,\cW_\alpha',E),\quad f\mto (a_\nu(f))_{\nu\in \N_0^n}
\]
is an isomorphism of topological vector spaces,
for $\cW_\alpha':=\{w_\beta\colon \beta\in N_\alpha'\}$.
\end{cor}
\begin{proof}
Let $J$ be the set of all $j=(j_1,\ldots, j_\ell)$
such that $j_i\in \{1,\ldots, n_i\}$ for $i\in\{1,\ldots,\ell\}$.
If $\beta\in N_\alpha$ and $j\in J$, define $\beta[j]\in N_\alpha$
via $\beta[j]:=(\eta_1,\ldots, \eta_\ell)$,
where $\eta_i\in \N_0^{n_i}$ has $|\beta_i|$ as its $j_i$-th component,
while all other components are~$0$.
If $\nu=(\nu_1,\ldots, \nu_\ell)\in \N_0^{n_1}\times\cdots\times \N_0^{n_\ell}$
with $\nu_i=(\nu_{i,1},\ldots, \nu_{i,n_i})$,
then
\[
\nu_i^{\beta_i}\leq\max\{\nu_{i,1},\ldots,\nu_{i,n_i}\}^{|\beta_i|}
\]
and hence
\begin{equation}\label{wilreu}
\nu^\beta \leq  \max\{\nu_{1,j_i}^{|\beta_1|}\cdots\nu_{\ell,j_\ell}^{|\beta_\ell|}\colon
(j_1,\ldots, j_\ell)\in J\}.
\end{equation}
Hence, if (\ref{decaycond3}) is satisfied for all $\beta\in N_\alpha'$,
then also for all $\beta\in N_\alpha$.
The converse is obvious as $N_\alpha'\sub N_\alpha$.
Thus Proposition~\ref{Mahbetter}
shows that indeed $f$ is $C^\alpha$
if and only if (\ref{decaycond3}) is satisfied for all $\beta\in N_\alpha'$.
By (\ref{wilreu}), we have
\begin{equation}\label{wilaso}
w_\beta\leq  \max\{ w_{\beta [ j ] }  \colon j\in J \}
\end{equation}
(as a pointwise maximum), entailing that
$c_0(\N_0^n,\cW_\alpha,E) = c_0(\N_0^n,\cW_\alpha',E)$
as a vector space.
We also deduce from (\ref{wilaso}) that
\[
\| . \|_{w_\beta,q} \leq \max\{ \|.\|_{ w_{\beta [ j ] ,q}} \colon j\in J \}
\]
(as a pointwise maximum) for each ultrametric continuous seminorm $q$ on~$E$.
Thus
$c_0(\N_0^n,\cW_\alpha,E)=c_0(\N_0^n,\cW_\alpha',E)$
as a topological vector space.
The final assertion therefore follows from the final conclusion of Proposition~\ref{Mahbetter}.
\end{proof}
\begin{rem}\label{forThmC}
To deduce the final assertion of Theorem~C from Corollary~\ref{fico}, note that
\[
w_{r e_i}(\nu)
=\nu_i^r \leq |\nu|^r\leq n^r\max\{\nu_1^r,\ldots,\nu_n^r\}
\leq n^r \max\{w_{r e_i}(\nu)\colon i\in\{1,\ldots, n\}\}
\]
for $\nu\in \N_0^n$.
\end{rem}
\begin{cor}
Let $E$ be a sequentially complete locally convex
space over $\Q_p$ and $f\colon U\to E$ be a function
on an open subset $U\sub \Z_p^n$.
Let $k\in \N_0$. Then $f$ is $C^k$ if and only if
$f$ is $C^{(k,0,\ldots,0)}$, $C^{(0,k,0,\ldots,0)}$, $\ldots\;$,
$C^{(0,\ldots,0,k,0)}$ and
$C^{(0,\ldots,0,k)}$.
\end{cor}
\begin{proof}
Each point $x\in U$ has an open neighbourhood of the form
$x+p^m\Z_p^n$ for some $m\in \N_0$.
The assertions now follow
from Corollary~\ref{fico},
applied (with $\ell=1$ and $\alpha=(k)$)
to the function $g\colon \Z_p^n\to E$, $g(y):= f(x+p^m y)$.
\end{proof}
Compare already \cite[p.\,140]{DSm} for the characterization
of $C^1$-functions\linebreak
$f\colon \Z_p^2\to\Q_p$
via (\ref{decaycond3}) (with $\ell=1$).
Related results for $C^k$-functions $\Z_p^n\to\K$
(where $\K$ is a finite field extension
of $\Q_p$) and topologies on
spaces of such functions can also
be found in \cite[54--65]{Nag}.\\[2.4mm]
The following lemma will help us to pass from compact cartesian sets to compact locally cartesian
(and more general) sets in some results.
\begin{la}\label{deccompinct}
Let $(\K,|.|)$ be an ultrametric field, $n\in \N$,
$U\sub \K^n$ be a compact locally cartesian set
and $\cX$ be an open cover of $U$.
Then the following holds:
\begin{itemize}
\item[{\rm(a)}]
There exist $m\in \N$ and disjoint compact, cartesian, relatively open subsets $W_1,\ldots, W_m$
of~$U$ subordinate to~$\cX$ such that $U=W_1\cup\cdots\cup W_m$.
\item[{\rm(b)}]
The sets $W_1,\ldots, W_m$ in {\rm(a)}
can be chosen in such a way that
$W_j=U\cap Q_j$ for certain disjoint clopen cartesian subsets
$Q_1,\ldots, Q_m$ of~$\K^n$.
%
\end{itemize}
\end{la}
\begin{proof}
(a) Consider the ultrametric $D$ on $\K^n$ defined via
\[
D(x,y):=\max\{|x_1-y_1|,\ldots,|x_n-y_n|\}
\]
for $x=(x_1,\ldots, x_n)$
and $y=(y_1,\ldots, y_n)$ in $\K^n$.
Let $d:=D|_{U\times U}$ be the metric induced on~$U$.
Because $U$ is locally cartesian, there exists a cover $\cV$ of $U$
consisting of cartesian, relatively open subsets $V\sub U$.
Let $r>0$ be a Lebesgue number
for the open cover $\cV$ of the compact metric space~$U$.
Thus, for each $x\in U$, there exists $V_x\in\cV$
such that $B^d_r(x)\sub V_x$. Write $V_x=V_1\times\cdots\times V_n$ with $V_1,\ldots,V_n\sub\K$.
After decreasing~$r$ if necessary, we may assume that $r$ also is a Lebesgue
number for~$\cX$.
Then
\begin{eqnarray*}
B^d_r(x) &=& U\cap B^D_r(x)=V_x\cap B^D_r(x)
=(V_1\times\cdots\times V_n)\cap
(B^\K_r(x_1)\times B^\K_r(x_n))\\
&=& (V_1\cap B^\K_r(x_1))\times\cdots\times (V_n\cap B^\K_r(x_n))
\end{eqnarray*}
is a cartesian, relatively open subset of~$U$.
The ultrametric equality implies that,
if $x,y\in U$, then either
$B^d_r(x)=B^d_r(y)$
or $B^d_r(x)\cap B^d_r(y)=\emptyset$.
Hence
\[
P:=\{ B^d_r(x)\colon x\in U\}
\]
is a partition of $U$ into disjoint
open and compact cartesian sets.
By compactness of~$U$, the set~$P$ is finite,
say $P=\{U_1,\ldots, U_m\}$ with disjoint sets
$U_1,\ldots, U_m$. Finally,
because $r$ is a Lebesgue number for~$\cX$,
each $B^d_r(x)$ (and hence each $U_j$)
is contained in some set $X\in \cX$. Hence $U_1,\ldots, U_m$ is subordinate to~$\cX$.\\[2.4mm]
(b) If $x,y\in U$ and $B^D_r(x)
\cap B^D_r(y)\not=\emptyset$, then
$B^D_r(x) =B^D_r(y)$ (as a consequence of the ultrametric inequality)
and hence $B^d_r(x) =B^d_r(y)$.
Therefore, if $U_j=B^d_r(x_j)$
in the proof of (a) (for $j\in \{1,\ldots, m\}$),
then also the clopen subsets $Q_1:=B^D_r(x_1),\ldots,
Q_m:=B^D_r(x_m)$ of $\K^n$ are disjoint.
By construction, $U_j=U\cap Q_j$.
%
\end{proof}
%
%
%
%
\begin{la}\label{justperfect}
Let $(X,d)$ be a complete metric space which is perfect $($i.e., without isolated points$)$.
Then every compact subset $K$ of $X$ is contained in a perfect
compact subset $L$ of~$X$.
\end{la}
\begin{proof}
We may assume that $K\not=\emptyset$.
Being compact and metrizable, $K$ has a countable dense subset.
Thus, we find $x_n\in K$ for $n\in \N$ such that
$D:=\{x_n\colon n\in \N\}$ is dense in~$K$.
Define $F_1:=\{x_1\}$.
For $x\in F_1\cup\{x_2\}$, choose $y_{2,x}\in X\setminus\{x\}$
such that $d(x,y_{2,x})<2^{-2}$, and set
\[
F_2:=F_1\cup\{x_2\}\cup \{y_{2,x}\colon x\in F_1\cup\{x_2\}\}.
\]
Recursively, if $F_{n-1}$ has already been defined,
choose $y_{n,x}\in X\setminus\{x\}$
for $x\in F_{n-1}\cup\{x_n\}$, 
such that $d(x,y_{n,x})<2^{-n}$, and set
\[
F_n:=F_{n-1}\cup\{x_n\}\cup \{y_{n,x}\colon x\in F_{n-1}\cup\{x_n\}\}.
\]
Because $F:=\bigcup_{n\in \N}F_n$ contains~$D$,
its closure $L:=\wb{F}$ contains~$K$ as a subset.
To see that~$L$ does not have isolated points,
it suffices to show this for~$F$.
Now each $x\in F$ is an element of~$F_m$ for some $m\in \N$.
Then $x\in F_n$ for each $n>m$,
and $F_n$ contains the point $y_{n,x}\not=x$
with $d(x,y_{n,x})<2^{-n}$. Hence $x$ is not isolated in~$F$.\\[2.4mm]
Note that each set $F_n$ is finite
and $F_1\sub F_2\sub \cdots$.
For each $m\in \N$, we have
\[
F_{m+1}\sub B^d_{2^{-(m+1)}}(F_m)\cup B^d_{2^{-(m+1)}}(K)=
B^d_{2^{-(m+1)}}(F_m\cup K)
\]
and hence, by a simple induction based on the triangle inequality,
\[
F_{m+n}\sub B^d_{2^{-(m+1)}+\cdots+ 2^{-(m+n)}}(F_m\cup K)
\sub
B^d_{2^{-m}}(F_m\cup K)
\]
for each $n\in \N$. Therefore
\begin{equation}\label{Fok}
F\sub B^d_{2^{-m}}(F_m\cup K).
\end{equation}
Because $L$, being closed in~$X$, is complete in the induced metric,
it will be compact if we can show that it is precompact.
To prove this, let $\ve>0$.
Choose $m\in \N$ such that $2^{-m}<\frac{\ve}{2}$.
Because $K$ is compact, there is a finite subset $\Phi\sub K$
such that $K\sub \bigcup_{x\in\Phi} B^d_{\ve/2}(x)$.
Then $\Phi\cup F_m$ is a finite subset of~$L$
and we claim that
\begin{equation}\label{thatsit}
L\sub \bigcup_{x\in \Phi\cup F_m}B^d_\ve(x)
\end{equation}
(whence $L$ is precompact and the proof is complete).
To prove the claim, it suffices to show that
\begin{equation}\label{ggen}
F\sub \bigcup_{x\in \Phi\cup F_m}\wb{B}^d_{2^{-m}+\frac{\ve}{2}}(x)
\end{equation}
(as $F$ is dense in~$L$ and the right-hand side is closed
and contained in that of (\ref{thatsit})).
Let $y\in F$. By (\ref{Fok}), we have
$y\in B^d_{2^{-m}}(x)$ for some $x\in F_m\cup K$.
If $x\in F_m$, then $y$ is in the right-hand side of~(\ref{ggen}).
If $x\in K$, then $x\in B^d_{\ve/2}(z)$ for some $z\in \Phi$
and thus $y\in B^d_{2^{-m}}(x)\sub B^d_{2^{-m}+\frac{\ve}{2}}(z)$,
which is a subset of the union in (\ref{ggen}).
\end{proof}
\begin{la}\label{verylast}
Let $(\K,|.|)$ be a complete ultrametric field and $U\sub \K^n$ be a locally closed,
locally cartesian subset. Then the following holds:
\begin{itemize}
\item[{\rm(a)}]
If $x\in U$, then every neighbourhood $W\sub U$
of~$x$ contains an open cartesian neighbourhood  $V\sub U$
of~$x$ which is closed in~$\K^n$.
\item[{\rm(b)}]
If $V\sub U$ is an open cartesian subset which is closed in~$\K^n$,
then every compact subset $K\sub V$ is contained
in a compact cartesian subset $L\sub V$.
\item[{\rm(c)}]
Every compact set $K\sub U$ is contained in a compact, locally cartesian subset
of~$U$.
\item[{\rm(d)}]
Consider $U$ as a subset of $\K^{n_1}\times\cdots\times\K^{n_\ell}$
with $n_1+\cdots+ n_\ell=n$, and let $\alpha\in (\N_0\cup\{\infty\})^\ell$.
Then the topology on $C^\alpha(U,E)$ is initial with respect to the restriction maps
\[
\rho_K\colon C^\alpha(U,E)\to C^\alpha(K,E),
\]
for $K$ ranging through the set of all
compact cartesian subsets of~$U$.
\item[{\rm(e)}]
Every $0$-neighbourhood in $C^\alpha(U,E)$ contains
a $0$-neighbourhood of the form $\rho_K^{-1}(Q)$
for some compact, locally cartesian subset $K\sub U$
and $0$-neighbourhood $Q\sub C^\alpha(K,E)$.
\end{itemize}
\end{la}
\begin{proof}
(a) Because~$U$ is locally cartesian, after shrinking~$W$,
we may assume that~$W$ is cartesian.
Since~$U$ is locally closed, there exists a neighbourhood $C\sub U$ of~$x$
such that $C\sub W$.
Let $C^0$ be the interior of $C$ relative~$U$.
There exists a clopen cartesian subset $Q\sub \K^n$
with $x\in Q$ (e.g., a small ball around~$x$ with respect to the maximum norm)
such that $U\cap Q\sub C^0$.
Then $V:=U\cap Q$ is open in~$U$
and $V=C\cap V=C\cap U\cap Q=C\cap Q$ is closed in~$\K^n$.
Moreover, $V=W\cap Q$ is cartesian (see \ref{nicernbhd}).

(b) We have $V=V_1\times\cdots\times V_n$ with closed subsets
$V_1,\ldots, V_n\sub \K$ without isolated points.
Being closed in~$\K$, the sets $V_j$ are complete.
Let $\pr_j\colon \K^n\to \K$ be the projection onto the $j$-th component,
for $j\in \{1,\ldots, n\}$. Then $K_j:=\pr_j(K)$ is a compact subset of~$V_j$
and hence contained in a compact perfect subset $L_j\sub V_j$,
by Lemma~\ref{justperfect}.
Now $L:=L_1\times\cdots\times L_n$ is a compact, cartesian
subset of~$V$ and contains~$K$.

(c) Let $\cV$ be the set of all open, cartesian subsets of~$U$ which are closed in $\K^n$.
By (a), $\cV$ is an open cover of~$U$.
Define the ultrametric $D$ on $\K^n$
as in the proof of Lemma~\ref{deccompinct}\,(a).
Let $d:=D|_{U\times U}$ be the metric induced on~$U$
and $d_K:=d|_{K\times K}$ be the metric induced on~$K$.
We let $\delta>0$ be a Lebesgue number for the
open cover $\cV$ of~$K$.
Thus,
for each $x\in K$, there exists $V_x\in \cV$ such that $B^d_\delta(x)\sub V_x$.
Because $K$ is compact and $d_K$ is an ultrametric,
we have $\{B^{d_K}_\delta(x)\colon x\in K\}=\{W_1,\ldots, W_m\}$
with finitely many disjoint set $W_1,\ldots, W_m$.
We have $W_j=B^{d_K}_\delta(x_j)$ for some $x_j\in K$.
Then also the sets $Y_j:=B^d_\delta(x_j)$
are disjoint for $j\in \{1,\ldots, m\}$.
Note that $Y_j=B^D_\delta(x_j)\cap V_{x_j}$
is cartesian (see~\ref{nicernbhd}), open in $U$ and closed in~$\K^n$.
The sets $K_j:=K\cap Y_j$ being compact,
(b) provides a compact cartesian subset $L_j\sub Y_j$
such that $K_j\sub L_j$. Because the cartesian sets $L_j$ are disjoint and
closed, it is clear that $L:= L_1\cup\cdots\cup L_m$
is locally cartesian. Moreover, $L$ is compact,
a subset of~$U$, and contains $K=K_1\cup\cdots\cup K_m$.

(d) Let $\cT$ be the initial topology on $C^\alpha(U,E)$ with respect to the mappings
$\rho_K$.
By Remark~\ref{gentypeCalph}\,(a),
the compact-open $C^\alpha$-topology $\cO$
on $C^\alpha(U,E)$
makes each of the maps $\rho_K$ continuous and linear.
Hence $\cT\sub\cO$.
For the converse, let $W\sub C^\alpha(U,E)$ be a $0$-neighbourhood.
Let $\cV$ be the set of all open cartesian subsets of~$U$
which are closed in~$\K^n$.
By (a), $\cV$ is a cover of~$U$.
Thus Lemma~\ref{onesttwost}
provides $\beta_1,\ldots,\beta_m\in N_\alpha$
and $V_1,\ldots, V_m\in \cV$ with
$\beta_j=(\beta_{j,1},\ldots,\beta_{j,n})$
and $V_j=V_{j,1}\times\cdots\times V_{j,n}$ for $j\in \{1,\ldots, m\}$;
compact
subsets $K_j\sub V_j^{<\beta_j>}$
such that $K_j=K_{j,1}^{1+\beta_{j_1}}\times\cdots\times K_{j,n}^{1+\beta_{j,n}}$
with compact subsets $K_{j,i}\sub \K$
for $i\in \{1,\ldots, n\}$;
and open $0$-neighbourhoods $Q_1,\ldots, Q_m\sub E$
such that
\[
P:=\bigcap_{j=1}^m\Delta_{\beta_j,V_j}^{-1}(\lfloor K_j,Q_j\rfloor)\sub W,
\]
where $\Delta_{\beta_j,V_j}\colon C^\alpha(U,E)\to C((V_j)^{<\beta_j>},E)$, $f\mto (f|_{V_j})^{<\beta_j>}$.
Then $K_{j,i}$ is a compact subset of $V_{j,i}$
and hence $C_j:=K_{j,1}\times\cdots\times K_{j,n}$
a compact subset of~$V_j$.
By (b), there exists a compact cartesian subset $L_j=L_{j,1}\times\cdots\times L_{j,n}
\sub V_j$
such that $C_j\sub L_j$ and thus $K_{j,i}\sub L_{j,i}$ for all~$i$.
Consider the continuous linear maps
$\delta_{\beta_j,L_j}\colon C^\alpha(L_j,E)\to C(L_j^{<\beta_j>},E)$, $f\mto f^{<\beta_j>}$.
Then $P_j:=\delta_{\beta_j,L_j}^{-1}(\lfloor K_j,Q_j\rfloor)$ is
an open $0$-neighbourhood in $C^\alpha(L_j,E)$
and $P=\bigcap_{j=1}^m\rho_{L_j}^{-1}(P_j)$ is open with respect to~$\cT$.
Hence $\cO\sub\cT$ and $\cO=\cT$ follows.

(e) By (d), each $0$-neighbourhood $W\sub C^\alpha(U,E)$
contains a $0$-neighbour\-hood of the form
$P:=\rho_{K_1}^{-1}(Q_1)\cap\cdots\cap \rho_{K_m}^{-1}(Q_m)$
for some $m\in \N$,
compact cartesian sets $K_j\sub U$
and $0$-neighbourhoods $Q_j\sub C^\alpha(K_j,E)$,
for $j\in \{1,\ldots, m\}$.
By~(c), $K_1\cup\cdots\cup K_m$ is contained in a compact, locally\linebreak
cartesian
subset $K\sub U$. Let $\rho_K\colon C^\alpha(U,E)\to C^\alpha(K,E)$
and $\rho_{K_j,K}\colon$ $C^\alpha(K,E)\to C^\alpha(K_j,E)$
be the restriction maps, which are continuous\linebreak
linear by Remark~\ref{gentypeCalph}\,(a).
Then $Q:=\bigcap_{j=1}^m  \rho_{K_j,K}^{-1}(Q_j)$
is a $0$-neighbourhood in $C^\alpha(K,E)$ and
$\rho_{K_j,K}\circ\rho_K=\rho_{K_j}$ entails that $P=\rho_K^{-1}(Q)$.
\end{proof}
\begin{numba}
Let $\K$ be a field, $E$ be a $\K$-vector space,
$n_1,\ldots, n_\ell\in \N$,
$n:=n_1+\cdots+n_\ell$ and
$\alpha\in (\N_0\cup\{\infty\})^\ell$.
A function $p \colon U \to E$ on a subset $U\sub \K^n$
is called a \emph{polynomial function of multidegree $\leq \alpha$}
if there exist $a_\beta\in E$ for multi-indices $\beta=(\beta_1,\ldots,\beta_\ell)\in \N_0^{n_1}\times\cdots\times
\N_0^{n_\ell}=\N_0^n$
with $|\beta_j|\leq\alpha_j$ for $j\in \{1,\ldots,\ell\}$,
such that $a_\beta=0$
for all but finitely many $\beta$ and
\[
p(x)=\sum_{\beta\leq\alpha}x^\beta a_\beta\quad\mbox{for all $\,x=(x_1,\ldots, x_n)\in U$.}
\]
We write $\Pol_{\leq \alpha}(U,E)$
for the space of all such~$p$.
If $(\K,|.|)$ is a valued field,
$U\sub \K^n$ a subset and
$E$ a topological $\K$-vector space, we say that a function
$f \colon U\to E$
is \emph{locally polynomial} of multidegree $\leq \alpha$
if each $x\in U$ has an open neighbourhood $V$ in~$U$
such that $f|_V=p$
for some polynomial function $p\colon \K^n\supseteq V \to E$
of multidegree $\leq \alpha$.
We write $\text{LocPol}_{\leq\alpha}(U,E)$
for the space of all locally polynomial $E$-valued
functions of multidegree $\leq\alpha$ on~$U$.
\end{numba}
\begin{la}\label{extlocp}
Let $(\K,|.|)$ be an ultrametric field,
$U\sub \K^{n_1}\times\cdots\times\K^{n_\ell}=\K^n$ be a compact locally cartesian set,
where $n_1,\ldots n_\ell\in \N$ and $n=n_1+\cdots + n_\ell$.\linebreak
Let
$\alpha\in (\N_0\cup\{\infty\})^\ell$
and $f\in \LocPol_{\leq\alpha}(U,E)$.
Then there exists a function
$g\in \LocPol_{\leq\alpha}(\K^n,E)$
such that $g|_U=f$.
\end{la}
\begin{proof}
We may assume that $U\not=\emptyset$.
Let $\cX$ be the set of all open subsets $X\sub U$
such that $f|_X\in \Pol_{\leq\alpha}(X,E)$.
By Lemma~\ref{deccompinct},
there exist disjoint clopen subsets $Q_1,\ldots, Q_m$ of $\K^n$
such that the intersections $W_j:=Q_j\cap U$
form an open cover of~$U$ for $j\in \{1,\ldots, m\}$
and  $W_j\sub X_j$ for some $X_j\in \cX$.
Thus, there are $p_j\in \Pol_{\leq\alpha}(\K^n,E)$
such that $p_j|_{W_j}=f|_{W_j}$.
Define $g\colon \K^n\to E$ via
$g(x):=p_j(x)$ if $x\in Q_j$, $g(x):=0$
if $x\in \K^n\setminus(Q_1\cup\cdots\cup Q_m)$.
Because the sets $Q_m$ are clopen and disjoint,
we have $g\in \LocPol_{\leq\alpha}(\K^n,E)$.
Also $g|_U=f$, using that
$g|_{W_j}=p_j|_{W_j}=f|_{W_j}$
and $U=W_1\cup\cdots\cup W_m$.
\end{proof}
Then the following analogue of Theorem~B holds:
\begin{prop}\label{betterdensloc}
For every  complete ultrametric field $\K$,
locally convex topological $\K$-vector space~$E$,
$n_1,\ldots, n_\ell\in \N$,
$\alpha\in (\N_0\cup\{\infty\})^\ell$
and\linebreak
locally closed, locally cartesian subset $U\sub \K^{n_1}\times\cdots\times\K^{n_\ell}=\K^n$ with\linebreak
$n:=n_1+\cdots+n_\ell$,
the space $\LocPol_{\leq \alpha}(U,E)$
of $E$-valued locally polynomial\linebreak
functions of
multidegree $\leq\alpha$
is dense in $C^\alpha(U,E)$.
\end{prop}
\begin{proof}
Step~1: If $U$ is a compact cartesian subset of~$\K^n$,
we can prove the assertion by induction on~$n\in \N$:\\[2.4mm]
\emph{The case $n=1$.} Then $\alpha=(r)$ with $r:=\alpha_1$.
For $E=\K$ and $r<\infty$, the assertion was established
in \cite[Proposition~II.40]{Nag},
and one can proceed analogously
if $r<\infty$ and~$E$ is an ultrametric
Banach space.
We can now pass to general locally convex spaces and
$r\in \N_0\cup\{\infty\}$
as in the proof of Theorem~B,
using the analogues of Lemmas~\ref{pllower},
\ref{alphpushf}
(subsuming \ref{alphvsub}),
\ref{plinim} and
\ref{laavoid}
for $C^\alpha$-maps on $U\sub
\K^{n_1}\times\cdots\times\K^{n_\ell}$
described in Remark~\ref{gentypeCalph} (c), (d), (f)
and (g), respectively.\\[2.4mm]
\emph{Induction step.}
If $n\geq 2$, we write $U=U_1\times U'$ with
compact cartesian subsets $U_1\sub \K^{n_1}$
and $U'\sub \K^{n_2}\times\cdots\times \K^{n_\ell}$.
We also write $\alpha=(r,\alpha')$ with $r\in \N_0\cup\{\infty\}$
and $\alpha'\in (\N_0\cup\{\infty\})^{\ell-1}$.
Using Theorem~D instead of Theorem~A
and replacing $\sum_{i=0}^r$ with
$\sum_{|\nu|\leq r}$ and $i$ with $\nu$
(where the summation is over all $\nu\in \N_0^{n_1}$
such that $|\nu|\leq r$),
we can perform the induction step
as in the proof of Theorem~B.\\[2.4mm]
Step 2: If $U$ is a compact, locally cartesian subset
of~$\K^n$, then $U=W_1\cup\cdots\cup W_m$
is a disjoint union of certain compact, relatively open,
cartesian subsets $W_1,\ldots, W_m$ of $U$
(see Lemma~\ref{deccompinct}).
As a consequence of Lemma~\ref{islocal}
and Remark~\ref{gentypeCalph}\,(b), the map
\[
\rho=(\rho_j)_{j=1}^m\colon C^\alpha(U,E)\to
\prod_{j=1}^m
C^\alpha(W_j,E),\quad f\mto (f|_{W_j})_{j=1}^m
\]
is an isomorphism of topological vector spaces.
Let $f\in C^\alpha(U,E)$ and $P\sub
C^\alpha(U,E)$ be a neighbourhood of~$f$.
By the preceding, we may assume
that $P=\rho^{-1}(P_1\times\cdots\times P_m)$
with neighbourhoods $P_j\sub C^\alpha(W_j,E)$ of $f|_{W_j}$
for $j\in\{1,\ldots, m\}$.
By Step~1, for each $j$ there exists some
$g_j\in P_j\cap \LocPol_{\leq\alpha}(W_j,E)$.
We define $g\colon U\to E$ via $g(x):=g_j(x)$ if $x\in W_j$.
Then $g\in P$ and $g\in \LocPol_{\leq\alpha}(U,E)$,
showing that the latter is dense.\\[2.4mm]
Step~3. Finally, let $U\sub \K^n$ be a locally closed,
locally cartesian subset,
$f\in C^\alpha(U,E)$ and $P\sub C^\alpha(U,E)$
be a $0$-neighbourhood.
By Lemma~\ref{verylast}\,(e),
we may assume that
\[
P=\rho_K^{-1}(Q)
\]
for some compact, locally cartesian subset $K$ of $U$
and open $0$-neighbourhood $Q\sub C^\alpha(K,E)$.
By Step~2, there exists $g\in (f|_K+Q)\cap
\LocPol_{\leq\alpha}(K,E)$.
Using Lemma~\ref{extlocp},
we find $h\in \LocPol_{\leq\alpha}(U,E)$
such that $h|_K=g$.
Then $h\in \rho_K^{-1}(f|_K+Q)=f+P$
and thus $\LocPol_{\leq\alpha}(U,E)$ is dense.
\end{proof}
\begin{prop}\label{denseglobpo}
For every complete ultrametric field $\K$,
locally convex topological $\K$-vector space~$E$,
$n_1,\ldots, n_\ell\in \N$,
$\alpha\in (\N_0\cup\{\infty\})^\ell$
and\linebreak
locally closed, locally
cartesian subset $U\sub \K^{n_1}\times\cdots\times \K^{n_\ell}=\K^n$ with\linebreak
$n:=n_1+\cdots+n_\ell$,
the space $\Pol(U,E)$
of $E$-valued polynomial
functions
is dense in $C^\alpha(U,E)$.
\end{prop}
\begin{proof}
Step~1. If $U\sub \K^n$ is a compact cartesian set,
the proof is by induction on~$\ell$.\\[2.4mm]
If $\ell=1$, then $\alpha=(r)$
with $r:=\alpha_1$.
For $E=\K$ and $r<\infty$, the assertion was established
in \cite[Corollary~II.42]{Nag},
and one can proceed analogously
if $r<\infty$ and~$E$ is an ultrametric
Banach space.\\[2.4mm]
Now let $E$ be a locally convex space,
$f\in C^r(U,E)$ and $W\sub C^r(U,E)$ be an open
$0$-neighbourhood.\\[2.4mm]
The case $r<\infty$.
Let $\wt{E}$
and $\phi_q\colon\wt{E}\to E_q$
be as in the proof of Theorem~B.
Repeating the arguments
(with Remark~\ref{gentypeCalph} (d) and (f)
instead of
Lemmas \ref{alphvsub} and \ref{plinim}),
we may assume that
$W=C^r(U,E)\cap\wt{W}$
for some open $0$-neighbourhood $\wt{W}\sub
C^r(U,\wt{E})$,
and
$\wt{W}=C^r(U,\phi_q)^{-1}(V)$
for some $q$ and open $0$-neighbourhood $V\sub C^r(U,E_q)$.
By the case of ultrametric Banach spaces,
there exists $g\in \Pol(U,E_q)$
such that $\phi_q\circ f - g\in V$.
Thus, there exists a finite subset $\Theta\sub \N_0^n$
and $c_\nu\in E_q$ for
$\nu\in \Theta$ such that
\[
g(x)=\sum_{\nu\in \Theta} x^\nu c_\nu.
\]
By Lemma~\ref{laavoid},
there are open neighbourhoods
$Q_\nu\sub E_q$ of $c_\nu$
such that
\[
\phi_q\circ f - \sum_{\nu\in \Theta}x^\nu b_\nu\in V,
\]
for all $b_\nu\in Q_\nu$.
Since $\phi_q(E)$ is dense in~$E_q$,
there exist $a_\nu \in E$
such that $\phi_q(a_\nu)\in Q_\nu$.
Hence $h:=\sum_{\nu\in \Theta}x^\nu a_\nu\in \Pol(U,E)$
and
\[
\phi_q\circ (f-h)=\phi_q\circ f - \sum_{\nu\in\Theta} x^\nu
a_\nu \in V.
\]
Thus
$f-h\in C^r(U,\phi_q)^{-1}(V)=\wt{W}$ and thus
$f-h\in \wt{W}\cap E=W$.\\[2.4mm]
If $r=\infty$,
we can argue as in the proof of Theorem~B,
replacing Lemma~\ref{pllower}
with Remark~\ref{gentypeCalph}\,(c).\\[2.4mm]
\emph{Induction step.}
If $\ell\geq 2$, we write $U=U_1\times U'$ with compact cartesian sets $U\sub \K^{n_1}$ and
$U'\sub \K^{n_2}\times\cdots\times \K^{n_\ell}$.
We also write $\alpha=(r,\alpha')$ with $r\in \N_0\cup\{\infty\}$
and $\alpha'\in (\N_0\cup\{\infty\})^{\ell-1}$.
Let $f\in C^\alpha(U,E)$ and $W\sub C^\alpha(U,E)$
be an open $0$-neighbourhood.
By Theorem~D, the map
\[
\Phi\colon C^\alpha(U,E)\to C^r(U_1,C^{\alpha'}(U',E)),\quad f\mto f^\vee
\]
is an isomorphism of topological vector spaces.
Hence $W=\Phi^{-1}(V)$ for some open $0$-neighbourhood
$V\sub C^r(U_1,C^{\alpha'}(U',E))$.
By induction, there exists
$g\in \Pol(U_1,C^{\alpha'}(U',E))$
with $\Phi(f)-g\in V$.
Thus, there is a finite subset $\Theta\sub \N_0^{n_1}$
and $c_\nu \in C^{\alpha'}(U',E)$ for
$\nu\in \Theta$
such that
\[
g(x)=\sum_{\nu\in\Theta} x^\nu c_\nu.
\]
By Remark~\ref{gentypeCalph}\,(g),
there are open neighbourhoods $Q_\nu\sub C^{\alpha'}(U',E)$
of $c_\nu$
such that
\[
\Phi(f) - \sum_{\nu\in \Theta}X^\nu b_\nu\in V
\]
for all $b_\nu\in Q_\nu$,
where $X^\nu\colon U_1\to\K$, $x\mto x^\nu$.
By induction, we find $b_\nu\in Q_\nu\cap \Pol(U',E)$.
Then
\[
g\colon U\to E,\quad g(x,x'):=\sum_{\nu\in \Theta}
x^\nu b_\nu(x')\quad\mbox{for $x\in U_1$, $x'\in U'$}
\]
is an element of $\Pol(U,E)$.
Since $\Phi(g)=\sum_{\nu\in\Theta}x^\nu b_\nu$,
we have $\Phi(f)-\Phi(g)\in V$ and thus $f-g\in W$.\\[2.4mm]
Step~2. Now assume that $U\sub \K^n$ is compact and locally cartesian.
Given $f\in C^\alpha(U,E)$
and an open neighbourhood $W\sub C^\alpha(U,E)$
of~$f$,
we want to find $h\in W\cap \Pol(U,E)$.
By Proposition~\ref{betterdensloc},
there exists $g\in W\cap \LocPol_{\leq\alpha}(U,E)$.
After replacing $f$ with $g$,
we may assume that $f$ is locally polynomial.
Let $\cV$ be a cover of~$U$
by open subsets $V\sub U$ on which $f$ is polynomial.
By Lemma~\ref{deccompinct},
we can write
$U=W_1\cup\cdots\cup W_m$
as a disjoint union of relatively open,
compact, cartesian subsets $W_1,\ldots, W_m\sub U$,
of the form $W_j=U\cap Q_j$
with clopen, cartesian, disjoint
subsets $Q_1,\ldots,Q_m\sub \K^n$,
such that each $W_j$ is a subset
of some $V_j\in \cV$.
By the last property, for
each $j$ we find $p_j\in \Pol(\K^n,E)$
such that $f|_{W_j}=p_j|_{W_j}$.
Let $\pr_i\colon \K^n\to \K$
be the projection onto the $i$-th component,
for $i\in \{1,\ldots, n\}$.
Then $K_i:=\pr_i(U)\sub \K$
is compact and does not have isolated points
(as $K_i=\bigcup_{j=1}^m\pr_i(W_j)$).
Thus
\[
K:=K_1\times\cdots\times K_n\sub \K^n
\]
is a compact cartesian set.
We define a (new) function $g\colon K\to E$
via $g(x):=p_j(x)$ if $x\in K\cap Q_j$,
$g(x):=0$ if $x\in K\setminus (Q_1\cup\cdots\cup Q_m)$.
Lemma~\ref{islocal} implies that $g\in C^\alpha(K,E)$.
Moreover, $g|_U=f$.
The restriction map $\rho\colon C^\alpha(K,E)\to C^\alpha(U,E)$
is continuous (Remark~\ref{gentypeCalph}\,(a)),
whence $\rho^{-1}(W)$
is an open neighbourhood of~$g$.
By Step~1,
there exists $h\in \rho^{-1}(W)\cap \Pol(K,E)$.
Then $h|_U\in W\cap \Pol(U,E)$
and hence $\Pol(U,E)$ is dense.\\[2.4mm]
Step~3. Now consider a locally closed,
locally cartesian subset $U\sub \K^n$.
Let $f\in C^\alpha(U,E)$ and $P\sub C^\alpha(U,E)$
be an open $0$-neighbourhood.
As in Step~3 of the proof of Proposition~\ref{betterdensloc},
we may assume that $P=\rho^{-1}_K(Q)$
for some open $0$-neighbourhood $Q\sub C^\alpha(K,E)$
and compact, locally cartesian subset $K\sub U$.
By Step~2 and Lemma~\ref{extlocp}, there exists $p\in \Pol(U,E)$
such that $f|_K-p|_K\in Q$.
Thus $f-p\in \rho_K^{-1}(Q)=P$,
and thus $\Pol(U,E)$ is dense.
\end{proof}
\appendix
\section{Outlook on related concepts and results}
In this appendix, we give an alternative definition of $C^\alpha$-maps,
which applies to maps on subsets of arbitrary (possibly infinite-dimensional!)
topological vector spaces.
We also describe various assertions (to be proved elsewhere)
which will make the theory of $C^\alpha$-maps more complete
(but were irrelevant for our current ends).
It is natural to approach them with the methods
developed in \cite{Alz}, \cite{AaS}, \cite{BGN}, \cite{CMP}, \cite{ZOO},
and \cite{IM2}.\\[2.4mm]
We recall from \cite{IM2} (and the earlier work \cite{BGN},
in the case of open sets):
\begin{defn}
If $\K$ is a topological field,
$E$ a topological $\K$-vector space
and $U\sub E$ a subset with dense interior,
we define $U^{[0]}:=U$ and
\[
U^{[1]}:=\{(x,y,t)\in U\times E\times\K\colon x+t y\in U\}\,.
\]
Recursively, we let $U^{[k]}:=(U^{[1]})^{[k-1]}$
for $k\in \N$.
If also $F$ is a topological $\K$-vector space
and $f\colon U\to F$ a function,
we say that $f$ is $C^0_{BGN}$
it $f$ is continuous, and define $f^{[0]}:=f$ in the case.
We say that $f$ is $C^1_{BGN}$ if $f$ is continuous and
there exists a continuous map
\[
f^{[1]}\colon U^{[1]}\to F
\]
such that $f^{[1]}(x,y,t)=\frac{1}{t}(f(x+ty)-f(x))$ for all $(x,y,t)\in U^{[1]}$
with $t\not=0$.
Recursively, having defined when $f$ is $C^j_{BGN}$
and associated maps $f^{[j]}$ for $j\in \{0,1,\ldots,k-1\}$,
where $k\in \N$,
we say that $f$ is $C^k_{BGN}$ if $f$ is $C^1_{BGN}$
and $f^{[1]}$ is $C^{k-1}_{BGN}$.
Let $f^{[k]}:=(f^{[1]})^{[k-1]}\colon U^{[k]}\to F$
in this case.
\end{defn}
We now define the analogous $C^\alpha$-maps.
\begin{defn}\label{cainf}
Let $\K$ be a topological field,
$E_1,\ldots, E_\ell$ be topological $\K$-vector spaces
and $U_j\sub E_j$ be subsets with dense interior, for $j\in \{1,\ldots, \ell\}$.
Let $\alpha=(\alpha_1,\ldots,\alpha_\ell)\in (\N_0\cup\{\infty\})^\ell$
and $F$ be a topological $\K$-vector space.
We say that a map $f\colon U_1\times\cdots\times U_\ell\to F$ is $C^\alpha_{BGN}$
if there exist continuous mappings
\[
f^{[\beta]}\colon U_1^{[\beta_1]}\times\cdots\times U_\ell^{[\beta_\ell]}\to F
\]
for all $\beta\in \N_0^\ell$ with $\beta\leq\alpha$,
such that
\[
f^{[0,\beta_\ell]}(x,y):=(f(x,\sbull))^{[\beta_\ell]}(y)
\]
for all $x\in U_1\times \cdots\times U_{\ell-1}$
and $y\in U_\ell^{[\beta_\ell]}$
and, recursively,
\[
f^{[0,\beta_j,\beta')]}(x,y,z):=(f^{[0,0,\beta']}(x,\sbull,z))^{[\beta_j]}(y)
\]
for all $x\in U_1\times \cdots\times U_{j-1}$,
$y\in U_j^{[\beta_j]}$ and $z\in U_{j+1}^{[\beta_{j+1}]}\times\cdots\times
U_\ell^{[\beta_\ell]}$,
with $\beta':=(\beta_{j+1},\ldots,\beta_\ell)$.
Endow the space $C^\alpha(U,F)_{BGN}$ of all $C^\alpha_{BGN}$-maps
$f\colon U\to F$ with the initial topology with respect to
the maps
\[
C^\alpha(U,F)_{BGN}\to C(U^{[\beta]},F),\quad f\mto f^{[\beta]}
\]
(using the compact-open topology on the right-hand side),
for $\beta\in \N_0^\ell$ such that $\beta\leq \alpha$.
\end{defn}
Then the following holds:
\begin{numba}
For $E_1,\ldots, E_\ell$, $F$ and $U_1,\ldots, U_\ell$
as before and $r\in \N_0\cup\{\infty\}$,
a map $f\colon U_1\times\cdots\times U_\ell\to F$
is $C^r_{BGN}$ if and only if $f$ is $C^\alpha$
for all $\alpha\in \N_0^\ell$
such that $|\alpha|\leq r$
(cf.\ \cite[Lemma~3.9]{BGN} to get an induction started).
More generally, if also $H_1,\ldots, H_k$ are topological $\K$-vector spaces,
$V_i\sub H_i$ subsets with dense interior for $i\in \{1,\ldots, k\}$,
$\beta\in (\N_0\cup\{\infty\})^k$
and
\[
f\colon V_1\times \cdots\times V_k\times U_1\times\cdots\times U_\ell\to F
\]
a map, then  $f$ is $C^{(\beta,\alpha)}$
for all $\alpha$ as before
if and only if $f$ is $C^{(\beta,r)}$
on $V_1\times \cdots\times V_k\times U\sub H_1\times \cdots\times H_k\times E$,
with $E:=E_1\times\cdots\times E_\ell$ and $U:=U_1\times\cdots\times U_\ell$.
\end{numba}
\begin{numba}
Consider $E_j=\K^{n_j}$ for $j\in \{1,\ldots,\ell\}$,
locally cartesian subsets $U_j\sub E_j$ with dense interior,
and $\alpha\in (\N_0\cup\{\infty\})^\ell$.
Let $U:=U_1\times\cdots\times U_\ell$
and $f\colon U\to F$
be a function from~$U$ to a topological $\K$-vector space~$F$.
Then:
\begin{itemize}
\item[{\rm(a)}]
$f$ is $C^\alpha$ in the sense of Definition~{\rm\ref{defmocom}}
if and only if $f$ is $C^\alpha_{BGN}$.
Moreover, $C^\alpha(U,F)=C^\alpha(U,F)_{BGN}$
as a topological vector space.\footnote{See already \cite[Theorem~A]{CMP}
for the case of $C^r$-functions,
and Theorem~B.2 of its preprint-version
{\tt arXiv:math/0609041v1}
for the equality of the $C^r$-topologies.}
\item[{\rm(b)}]
If $\K\in \{\R,\C\}$ and $F$ is locally convex,
then $f$ is $C^\alpha$ if and only if the partial derivatives
$\partial^\beta f$ exist on the interior of
$U$
and extend to continuous functions $\partial^\beta f \colon U\to F$
$($denoted by the same symbol$)$,
for all $\beta\in N_\alpha$.
\end{itemize}
\end{numba}
See \cite{Alz}, \cite[Lemmas~3.17 and 3.18]{AaS}, \cite[Proposition~4.5]{BGN},
\cite[1.10]{IM2}
and  \cite[Proposition~2.18]{Nag} for special cases of the following Chain Rule:
\begin{numba}
If $f\colon U_1\times \cdots\times U_\ell\to F$ is $C^\alpha_{BGN}$
and $g_j\colon V_{j,1}\times\cdots\times V_{j,k_j}\to U_j\sub E_j$
is $C^{\beta_j}$ (resp., $C^{\beta_j}_{BGN}$) and $|\beta_j|\leq \alpha_j$
for $j\in \{1,\ldots,\ell\}$, then the map
\[
f\circ (g_1\times \cdots\times g_\ell)\colon
(V_{1,1}\times\cdots\times V_{1,k_1})\times\cdots\times
(V_{\ell,1}\times\cdots\times V_{\ell,k_\ell})\to F
\]
is $C^{(\beta_1,\ldots,\beta_\ell)}$
(resp., $C^{(\beta_1,\ldots,\beta_\ell)}_{BGN}$).
The composition is also $C^{(\beta_1,\ldots,\beta_\ell)}$
if $f$ is $C^\alpha$
and each $g_j$ is $C^{\beta_j}$.
\end{numba}
\begin{numba}
If $\K$ is $\R$ or $\C$,
all of $E_1,\ldots, E_\ell,F$ are locally convex and $U_1,\ldots, U_\ell$
open, then the $C^\alpha_{BGN}$-maps of Definition~\ref{cainf}
should also coincide with the $C^\alpha$-maps
discussed in \cite{AaS} (in the case $\ell=2$)
and the work in progress \cite{Alz}
(see \cite[Proposition~7.4]{BGN} for the case of $C^r$-maps).
Moreover, both approaches should give the same
topology on $C^\alpha(U,E)$
(as in the case of $C^r$-maps settled in~\cite[Proposition~4.19\,(d)]{ZOO}).
\end{numba}
\begin{numba}\label{ifmultilin}
If $E_1,\ldots, E_\ell$, $U_1\ldots, U_\ell$ and $\alpha$
are as in Definition~\ref{cainf}, $H:=H_1\times\cdots\times H_k$
a direct product of topological $\K$-vector spaces and
$f\colon U_1\times \cdots\times U_\ell\times H\to F$
a $C^{(\alpha,0)}$-map such that $f(x,\sbull)\colon H_1\times\cdots\times H_k\to F$ is
$k$-linear for all $x\in U_1\times\cdots\times U_\ell$,
then $f$ is $C^{(\alpha,\infty)}$ (cf.\
\cite[Lemma~3.14]{AaS}
for a special case).
\end{numba}
\begin{numba}
Assume that $E_1,\ldots, E_\ell$, $F_1\ldots, F_k$ and $H$ are topological vector spaces
over the topological field $\K$, $U\sub E_1\times \cdots\times E_\ell=:E$
and $V\sub F_1\times \cdots\times F_\ell=:F$
subsets with dense interior. Let $\alpha\in (\N_0\cup\{\infty\})^\ell$
and $\beta\in (\N_0\cup\{\infty\})^k$.
Then $f^\vee(x):=f(x,\sbull)\in C^\beta(V,H)$
for each $f\in C^{(\alpha,\beta)}(U\times V,H)$
and $x\in U$, the map $f^\vee\colon U\to C^\beta(V,H)$ is $C^\alpha$,
and
\[
\Phi\colon C^{(\alpha,\beta)}(U\times V,H)\to C^\alpha(U,C^\beta(V,H)),\quad
f\mto f^\vee
\]
is a linear topological embedding.
If $\K$, $E$ and $F$ are metrizable
or $\K$ and $V$ are locally compact,
then $\Phi$ is a topological embedding.\footnote{Less is needed:
Assume that $V^{[\eta]}$ is locally compact for each $\eta\in \N_0^k$ such that $\eta\leq\beta$,
or that $(U\times V)^{[\gamma]}$ is a k-space
for all $\gamma\in \N_0^{\ell+ k}$ such that $\gamma\leq (\alpha,\beta)$.} 
\end{numba}
\begin{numba}
If $\K$ and $U=U_1\times \cdots\times U_\ell$ is locally compact
in the situation of Definition~\ref{cainf},
then the evaluation map
\[
C^\alpha(U,F)\times U_1\times \cdots\times U_\ell \to F,\quad (f,x)\mto f(x)
\]
is $C^{(\infty,\alpha)}$.
\end{numba}
(See \cite{Alz}, \cite[Proposition~3.20]{AaS} and \cite[Proposition~11.1]{ZOO}
for special cases, and combine them with \ref{ifmultilin}
to increase the order of differentiability in the linear first argument).
Also the differentiability properties of the composition
$f\circ g$ as a function of $(f,g)$ can be analyzed
(generalizing the discussions in
\cite{Alz} and \cite[Proposition~11.2]{ZOO}).
\begin{numba}
If $U_j\sub \K^{n_j}$ is locally cartesian for $j\in \{1,\ldots, \ell\}$,
$U:=U_1\times\cdots\times U_\ell\sub \K^n$ with $n:=n_1+\cdots+n_\ell$,
$\alpha\in (\N_0\cup\{\infty\})^\ell$ and
$f\colon U\to E$ is a $C^\alpha$-function as in
Definition~\ref{defmocom},
then $f^{<\beta>}\colon U_1^{<\beta_1>}\times\cdots\times U_\ell^{<\beta_\ell>}\to F$
is a $C^{<\alpha-(|\beta_1|,\ldots,|\beta_\ell|)>}$-function,
for each $\beta=(\beta_1,\ldots,\beta_\ell)\in N_\alpha$. 
\end{numba}
%
%
%
%
%
%
%
%
%
%
%
%
%
%
%
%
%
%
%
%
%
%
%
%
%
%
%
%
%
\section{The compact-open topology}\label{appcotop}
In this appendix, we give a self-contained introduction to the compact-open topology
on function spaces.
Most of the results are classical, and no originality is claimed
(cf.\ \cite[Chapter~X, \S1--\S3]{BTG}
and \cite{Eng}, for example).
However, the sources treating
the topic do so with a different thrust.
Here, we compile precisely those results which are
the foundation
for the study of non-linear mappings between function spaces,
and their differentiability properties.\\[2.4mm]
We recall: If $X$ and $Y$ are Hausdorff topological spaces,
then the compact-open topology on $C(X,Y)$
is the topology given by the subbasis of open sets
\begin{equation}\label{thesubba}
\lfloor K,U\rfloor :=\{\gamma\in C(X,Y)\colon \gamma(K)\sub U\},
\end{equation}
for $K$ ranging through the set $\cK(X)$ of all compact subsets
of $X$ and $U$ though the set of open subsets of~$Y$.
In other words, finite intersections
of sets as in (\ref{thesubba}) form a basis for the compact-open topology on $C(X,Y)$.
We always endow $C(X,Y)$ with the compact-open topology (unless the contrary is stated).
\begin{rem}
The compact-open topology on $C(X,Y)$ makes the point evaluation
\[
\ev_x\colon C(X,Y)\to Y,\quad \gamma\mto \gamma(x)
\]
continuous, for each $x\in X$.
(If $U\sub Y$ is open, then $\ev_x^{-1}(U)=\lfloor \{x\}, U\rfloor$
is open in $C(X,Y)$).
As the maps $\ev_x$ separate points on $C(X,Y)$ for $x\in X$,
it follows that
the compact-open topology on $C(X,Y)$
is Hausdorff.
\end{rem}
\begin{la}\label{subbaca}
Let $X$ and $Y$ be Hausdorff topological spaces and
$\cS$ be a subbasis of open subsets of~$Y$.
Then
\[
\cV:=\{\lfloor K,U\rfloor \colon K\in \cK(X), U\in \cS\}
\]
is a subbasis for the compact-open topology on $C(X,Y)$.
\end{la}
\begin{proof}
Each of the sets $V\in \cV$ is open in the compact-open topology.
To see that $\cV$ is a subbasis for the latter,
it suffices to show that
$\lfloor K,U\rfloor$
is open in the topology generated by $\cV$,
for each $K\in \cK(X)$ and open subset $U\sub Y$.
To check this, let $\gamma\in \lfloor K,U\rfloor$.
For each $x\in K$, there is a finite subset $F_x\sub \cS$ such that
\[
\gamma(x)\in \bigcap_{S\in F_x}S
\sub U.
\]
Set $S_x:=\bigcap_{S\in F_x}S$.
Since~$K$ is compact and hence locally compact,
each $x\in K$ has a compact neighbourhood $K_x\sub \gamma^{-1}(S_x)$ in~$K$.
By compactness, there is a finite subset $\Phi\sub K$
such that $K=\bigcup_{x\in \Phi}K_x^0$ (where $K_x^0$
denotes the interior of $K_x$ relative $K$).
Then $V:=\bigcap_{x\in \Phi}\lfloor K_x,S_x\rfloor=
\bigcap_{x\in \Phi}\bigcap_{S\in F_x}\lfloor K_x,S\rfloor$
is open in the topology generated by~$\cV$,
and $V\sub \lfloor K,U\rfloor$.
In fact, let $\gamma\in V$.
For each $y\in K$, there is $x\in \Phi$ such that $y\in K_x$.
Hence $\gamma(y)\in \gamma(K_x)\sub S_x\sub U$
and thus $\gamma\in \lfloor K,U\rfloor$.
\end{proof}
Mappings between function spaces
of the form $C(X,f)$ (so-called superposition operators)
are used frequently.
\begin{la}\label{covsuppo}
Let $X$, $Y_1$ and $Y_2$ be Hausdorff topological spaces.
If a map $f\colon Y_1\to Y_2$ is continuous, then also
the following map is continuous:
\[
C(X,f)\colon C(X,Y_1)\to C(X,Y_2), \quad \gamma\mto f\circ \gamma\, .
\]
\end{la}
\begin{proof}
The map $C(X,f)$ will be continuous if pre-images of subbasic open sets are open.
To this end, let $K\in \cK(X)$ and $U\sub Y_2$ be open.
Then $C(X,f)^{-1}(\lfloor K,U\rfloor)=\lfloor K,f^{-1}(U)\rfloor$ is open in $C(X,Y_1)$
indeed.
\end{proof}
\begin{la}\label{inipush}
Let $X$ and $Y$ be Hausdorff topological spaces.
Assume that the topology on $Y$ is initial with respect to
a family $(f_j)_{j\in J}$ of maps $f_j\colon Y\to Y_j$\linebreak
to Hausdorff topological spaces~$Y_j\!$.\hspace*{-.5mm}
Then the compact-open topology on $C(X,Y)$ is initial with respect to the
family $(C(X,f_j))_{j\in J}$ of the mappings\linebreak
$C(X,f_j)\colon C(X,Y)\to C(X,Y_j)$.
\end{la}
\begin{proof}
Let $\cS$ be the set of all subsets of $Y$ of the form
$f_j^{-1}(W)$, with $j\in J$ and $W$ an open subset of $Y_j$.
By hypothesis, $\cS$ is a subbasis for the topology of~$Y$.
Hence, by Lemma~\ref{subbaca},
the sets $\lfloor K, f_j^{-1}(W)\rfloor$
form a subbasis for
the compact-open topology on $C(X,Y)$,
for $j,W$ as before and $K\in \cK(X)$.
But $\lfloor K, f_j^{-1}(W)\rfloor=C(X,f_j)^{-1}(\lfloor K,W\rfloor)$
(since
$\gamma(K)\sub f_j^{-1}(W)$
$\aeq$ $f_j(\gamma(K))\sub W$ $\aeq$ $(f_j\circ \gamma)(K)\sub W$),
and these sets form a subbasis for the initial topology on $C(X,Y)$ with respect to the family $(C(X,f_j))_{j\in J}$.
\end{proof}
We mention three consequences of Lemma~\ref{inipush}.
\begin{la}\label{ctsemb}
If $X$ is a Hausdorff topological space and $f\colon Y_1\to Y_2$ is a\linebreak
topological embedding of Hausdorff topological spaces,
then also the mapping\linebreak
$C(X,f)\colon C(X,Y_1)\to C(X,Y_2)$
is a topological embedding.
\end{la}
\begin{proof}
Since $f$ is injective, also $C(X,f)$ is injective.
By hypothesis, the topology on $Y_1$ is initial with respect to $f$.
Hence the topology on $C(X,Y_1)$ is initial with respect to $C(X,f)$
(by Lemma~\ref{inipush}).
As a consequence, the injective map $C(X,f)$ is a topological embedding.
%
%
\end{proof}
\begin{rem}\label{reminduco}
In particular, if $X$ and $Y_2$ are Hausdorff topological spaces and
$Y_1\sub Y_2$ is a subset, endowed with the induced topology,
then topology induced by $C(X,Y_2)$ on $C(X,Y_1)$ coincides with the
compact-open topology on $C(X,Y_1)$.
\end{rem}
Next, we deduce that $C(X,\prod_{j\in J} Y_j)=\prod_{j\in J}C(X,Y_j)$.
\begin{la}\label{cotopprod}
Let $X$ be a Hausdorff topological space and $(Y_j)_{j\in J}$
be a family of Hausdorff topological spaces, with cartesian product
$Y:=\prod_{j\in J}Y_j$ $($endowed with the product topology$)$
and the projections
$\pr_j\colon Y\to Y_j$. Then the natural map
\[
\Phi:=(C(X,\pr_j))_{j\in J}\colon C(X,Y)\to \prod_{j\in J}C(X,Y_j)
\]
is a homeomorphism.
\end{la}
\begin{proof}
The map $\Phi$ is a bijection (because
a map $f\colon X\to Y$ is continuous if and only if all of its components
$\pr_j\circ f$ are continuous). The topology on~$Y$
being initial with respect to the family $(\pr_j)_{j\in J}$,
the topology on $C(X,Y)$ is initial with respect to the family
$(C(X,\pr_j))_{j\in J}$ (Lemma~\ref{inipush}).
Thus $\Phi$ is a topological embedding
and hence a homeomorphism (being bijective).
\end{proof}
It is often useful that $C(X,\pl \, Y_j)=\pl \, C(X,Y_j)$.
\begin{la}
Let $(J,\leq)$ be a directed set, $((Y_j)_{j\in J}, (\phi_{j,k})_{j\leq k})$
be a projective system of Hausdorff topological
spaces\footnote{Thus $\phi_{j,k}\colon Y_k\to Y_j$ is a continuous
map for $j,k\in J$ such that $j\leq k$,
with $\phi_{j,j}=\id_{Y_j}$
and $\phi_{j,k}\circ \phi_{k,\ell}=\phi_{j,\ell}$ if $j\leq k\leq \ell$.}
and $Y$ be its projective limit,
with the limit maps $\phi_j\colon Y\to Y_j$.
Then the topological space $C(X,Y)$ is the projective limit of
$((C(X,Y_j))_{j\in J}, (C(X,\phi_{j,k}))_{j\leq k})$,
together with the limit maps $C(X,\phi_j)\colon C(X,Y)\to C(X,Y_j)$.
\end{la}
\begin{proof}
Let $P:=\prod_{j\in J}Y_j$ and $\pr_j\colon P\to Y_j$
be the projection onto the $j$-th component.
It is known that the map
\[
\phi:=(\phi_j)_{j\in J}\colon Y\to P
\]
is a topological embedding with image $\phi(Y)=\{(x_j)_{j\in J}\in P\colon
(\forall j,k\in J)$
$j\leq k \impl x_j=\phi_{j,k}(x_k)\}$.
Thus $C(X,\phi)$ is a topological embedding (see Lemma~\ref{ctsemb})
and hence so is
\[
\Phi\circ C(X,\phi) \colon C(X,Y)\to\prod_{j\in J}C(X,Y_j),
\]
using the homeomorphism $\Phi:=(C(X,\pr_j))_{j\in J}\colon C(X,P)\to
\prod_{j\in J}C(X,Y_j)$ from Lemma~\ref{cotopprod}.
The image of $\Phi\circ C(X,\phi)$ is contained in the projective limit
\[
L:=\{(f_j)_{j\in J}\in\prod_{j\in J}C(X,Y_j)\colon
(\forall j,k\in J)\, j\leq k \impl f_j=\phi_{j,k}\circ f_k\}
\]
of the spaces $C(X,Y_j)$.
If $(f_j)_{j\in J} \in L$ and $x\in X$, then
$\phi_{j,k}(f_k(x))=f_j(x)$ for all $j\leq k$, whence
there exists $f(x)\in Y$ with $\phi_j(f(x))=f_j(x)$.
The topology on~$Y$ being initial with respect to the maps $\phi_j$,
we deduce from the continuity of the maps $\phi_j\circ f= f_j$
that $f\colon X\to Y$ is continuous.
Now $(\Phi\circ C(X,\phi))(f)=(f_j)_{j\in J}$.
Hence $\Phi\circ C(X,\phi)$ is a homeomorphism from
$C(X,Y)$ onto~$L$, and the assertions follow.
%
\end{proof}
Also composition operators (or pullbacks) $C(f,Y)$
are essential tools.
\begin{la}\label{pubas}
Let $X_1$, $X_2$ and $Y$ be Hausdorff topological spaces.
If a map $f\colon X_1\to X_2$ is continuous, then also
the map
\[
C(f,Y)\colon C(X_2,Y)\to C(X_1,Y), \quad \gamma\mto \gamma\circ f
\]
is continuous.
\end{la}
\begin{proof}
If $K\sub X_1$ is compact and $U\sub Y$ an open set,
then $f(K)\sub X_2$ is compact and $C(f,Y)^{-1}(\lfloor K,U\rfloor )
= \lfloor f(K),U\rfloor$
(since $(\gamma\circ f)(K)\sub U$ $\aeq$
$\gamma (f(K))\sub U$
for $\gamma\in C(X_2,Y)$).
\end{proof}
\begin{rem}\label{resop}
If $X_2$ and $Y$ are Hausdorff topological spaces and $X_1\sub X_2$ a subset,
then the restriction map
\[
\rho\colon C(X_2,Y)\to C(X_1,Y),\quad \gamma\mto \gamma|_{X_1}
\]
is continuous. In fact, since $\gamma|_{X_1}=\gamma\circ f$
with the continuous inclusion map $f\colon X_1\to X_2$, $x\mto x$,
we have $\rho=C(f,Y)$ and Lemma~\ref{pubas} applies.
\end{rem}
\begin{la}\label{coveremb}
Let $X$ and $Y$ be Hausdorff topological spaces
and $(X_j)_{j\in J}$ be a family of subsets of $X$ whose interiors $X_j^0$
cover~$X$.
Then the map
\[
\rho\colon C(X,Y)\to \prod_{j\in J}C(X_j,Y),\quad \gamma\mto (\gamma|_{X_j})_{j\in J}
\]
is a topological embedding with closed image.
\end{la}
\begin{proof}
It is clear that the map $\rho$ is injective,
and it is continuous since each of its components is continuous,
by the preceding remark.
The image of $\rho$ consists of all $(\gamma_j)_{j\in J}\in
\prod_{j\in J}C(X_j,Y)$ such that
\[
(\forall j,k\in J)\,(\forall x\in X_j\cap X_k)\quad
\gamma_j(x)=\gamma_k(x)\,.
\]
The point evaluation $\ev_x\colon C(X_j,Y)\to Y$
and the corresponding one on $C(X_k,Y)$ are continuous if $x\in X_j\cap X_k$.
Since $Y$ is Hausdorff (and thus the diagonal is closed in $Y\times Y$),
it follows that $\im(\rho)$ is closed.
Since $\rho$ is injective, it will be an open map onto its image
if it takes the elements of a subbasis to relatively open sets.
To verify this property,
let $K\sub X$ be compact and $U\sub Y$ be open.
Each $x\in K$ is contained in the interior $X_{j_x}^0$ (relative $X$) for
some $j_x\in J_0$.
Because $K$ is locally compact, $x$ has
a compact neighbourhood $K_x$ in $K$
such that $K_x\sub X_{j_x}^0$.
By compactness, there is a finite subset $\Phi\sub K$
such that $K=\bigcup_{x\in \Phi}K_x$.
Then $W_x:=\lfloor K_x,U\rfloor\sub C(X_{j_x},U)$ is an open subset of $C(X_{j_k},Y)$
for all $x\in \Phi$. Hence
$W:=\{(\gamma_j)_{j\in J}\in \prod_{j\in J}C(X_j,Y) \colon (\forall x\in \Phi)\;\gamma_{j_x}\in W_x\}$
is open in
$\prod_{j\in J}C(X_j,Y)$.
Since $\rho(\lfloor K,U\rfloor)
=W\cap\im(\rho)$,
we see that $\rho(\lfloor K,U\rfloor)$ is relatively open in $\im(\rho)$,
which completes the proof.
\end{proof}
\begin{la}\label{laeval}
If $X$ is a locally compact space and $Y$ a Hausdorff topological space,
then the evaluation map
\[
\ve\colon C(X,Y)\times X\to Y,\quad (\gamma,x)\mto \gamma(x)
\]
is continuous.
\end{la}
\begin{proof}
Let $U\sub Y$ be open and $(\gamma,x)\in \ve^{-1}(U)$.
By local compactness, there exists a compact neighbourhood $K\sub X$ of $x$
such that $K\sub \gamma^{-1}(U)$ .
Then $\lfloor K,U\rfloor\times K$
is a neighbourhood of $(\gamma,x)$
in $C(X,Y)\times X$ and $\lfloor K,U\rfloor\times K\sub \ve^{-1}(U)$.
As a consequence, $\ve^{-1}(U)$
is open and thus $\ve$ is continuous.
\end{proof}
\begin{la}
Let $X$, $Y$ and $Z$ be Hausdorff topological spaces.
If $Y$ is locally compact, then the composition map
\[
\Gamma\colon C(Y,Z)\times C(X,Y)\to C(X,Z),\quad (\gamma,\eta)\mto \gamma\circ \eta
\]
is continuous.
\end{la}
\begin{proof}
Let $K\sub X$ be compact and $U\sub Z$ be an open set.
Let $(\gamma,\eta)\in \Gamma^{-1}(\lfloor K,U\rfloor)$.
For each $x\in K$, there is a compact neighbourhood~$L_x$
of $\eta(x)$ in~$Y$ such that $L_x\sub \gamma^{-1}(U)$.
We choose a compact neighbourhood $M_x$
of $\eta(x)$ in~$Y$ such that $M_x\sub L_x^0$.
Then $K_x:=K\cap\eta^{-1}(M_x)$ is a compact
neighbourhood of~$x$ in~$K$. Hence,
there is a finite subset $\Phi\sub K$ such that
$K=\bigcup_{x\in \Phi}K_x$.
Then $W:=\bigcap_{x\in \Phi}\lfloor K_x,L_x^0\rfloor$
is an open neighbourhood of~$\eta$ in $C(X,Y)$
and $V:=\bigcap_{x\in \Phi}\lfloor L_x,U\rfloor$
is an open neighbourhood of~$\gamma$ in $C(Y,Z)$.
We claim that $V\times W\sub \Gamma^{-1}(\lfloor K,U\rfloor)$.
If this is true, then $\Gamma^{-1}(U)$ is open and hence $\Gamma$
is continuous. To prove the claim, let $\sigma\in V$ and $\tau\in W$.
If $y\in K$, then $y\in K_x$ for some $x\in \Phi$.
Hence $\tau(y)\in \tau(K_x)\sub L_x^0$
and thus $\sigma(\tau(y))\in \sigma(L_x)\sub U$,
showing that $\sigma\circ \tau\in \lfloor K,U\rfloor$ indeed.
\end{proof}
\begin{la}\label{fewercp}
Let $X$ and $Y$ be Hausdorff topological spaces
and $\cL$ be a set of compact subsets of~$X$.
Assume that,
for each compact subset $K\sub X$
and open subset $V\sub X$ with $K\sub V$,
there exist $n\in \N$ and $K_1,\ldots, K_n\in \cL$
such that $K\sub \bigcup_{i=1}^nK_i\sub V$.
Then the sets
\[
\lfloor K,U\rfloor,\quad\mbox{for $K\in \cL$ and open sets $\,U\sub Y$,}
\]
form a basis for the compact-open topology on $C(X,Y)$.
\end{la}
\begin{proof}
Given a compact subset $K\sub X$ and open subset
$U\sub Y$, let\linebreak
$\gamma\in \lfloor K,U\rfloor$.
Then $V:=\gamma^{-1}(U)$
is an open subset of $X$ that contains~$K$,
whence we can find $K_1,\ldots,K_n\in \cL$ as described in the hypotheses.
Then $\gamma\in \bigcap_{i=1}^n \lfloor K_i,U\rfloor\sub \lfloor K,U\rfloor$.
The assertion follows.
\end{proof}
\begin{prop}\label{ctsexp}
Let $X$, $Y$ and $Z$ be Hausdorff topological spaces.
If $f\colon X\times Y\to Z$ is continuous, then also
\begin{equation}\label{deffv}
f^\vee\colon X\to C(Y,Z), \quad f^\vee(x):=f(x,\sbull)
\end{equation}
is continuous. Moreover, the map
\begin{equation}\label{expiso}
\Phi \colon C(X\times Y,Z)\to C(X,C(Y,Z)),\quad f\mto f^\vee
\end{equation}
is a topological embedding.
If $Y$ is locally compact or $X\times Y$ is a k-space,
then $\Phi$ is a homeomorphism.
\end{prop}
\begin{proof}
Let $\gamma\in C(X\times Y,Z)$.
To see that $\gamma^\vee$ is continuous, let $K\sub Y$ be compact and $U\sub Z$ be an open set.
If $x\in X$ such that $\gamma^\vee(x)\in \lfloor K,U\rfloor$,
then $\gamma(\{x\}\times K)\sub U$
and thus the product $\{x\}\times K$ of compact sets is contained in the open set
$\gamma^{-1}(U)$.
By the Wallace Lemma \cite[5.12]{Kel},
there is an open subset $V\sub X$ with $\{x\}\sub V$
and $V \times K\sub \gamma^{-1}(U)$.
Then $V\sub (\gamma^\vee)^{-1}(\lfloor K,U\rfloor)$,
showing that  $(\gamma^\vee)^{-1}(\lfloor K,U\rfloor)$
is a neighbourhood of~$x$ and hence open (as $x$ was arbitrary).
Thus $\gamma^\vee$ is continuous.\\[2.4mm]
To see that $\Phi$ is continuous, recall that
the sets $\lfloor L,U\rfloor$, with $L\in \cK(Y)$ and open sets $U\sub Z$,
form a subbasis of the compact-open topology on $C(Y,Z)$.
Hence,
by Lemma~\ref{subbaca},
the sets $\lfloor K,\lfloor L,U\rfloor\rfloor$
form a subbasis of the compact-open topology on $C(X,C(Y,Z))$,
for $L$ and $U$ as before and $K\in \cK(X)$.
We claim that
\begin{equation}\label{hlp}
\Phi^{-1}(\lfloor K,\lfloor L,U\rfloor\rfloor)=
\lfloor K\times L, U\rfloor\,.
\end{equation}
If this is true, then $\Phi$ is continuous.
To prove the claim,
let $\gamma\in C(X\times Y,Z)$.
Then $\Phi(\gamma)\in
\lfloor K,\lfloor L,U\rfloor\rfloor$ $\aeq$
$\Phi(\gamma)(K)\sub \lfloor L,U\rfloor$
$\aeq$ $(\forall x\in K)$
$\gamma(x,\sbull)\in \lfloor L,U\rfloor$
$\aeq$ $(\forall x\in K)$ $(\forall y\in L)$
$\gamma(x,y)\in U$
$\aeq$ $\gamma\in \lfloor K\times L,U\rfloor$,
establishing~(\ref{hlp}).\\[2.4mm]
Because $\Phi$ is (obviously) injective,
it will be open onto its image if it takes
open sets in a subbasis to relatively open sets.
Let $\cL\sub \cK(X\times Y)$ be the set of all products $K\times L$,
where $K\in\cK(X)$ and $L\in \cK(Y)$.
We claim that $\cL$ satisfies
the hypotheses of Lemma~\ref{fewercp}
(for the function space $C(X\times Y,Z)$).
If this is so, then the sets $\lfloor K\times L,U\rfloor$
with $K\in \cK(X)$ and $L\in \cK(Y)$
form a subbasis for the compact-open topology on $C(X\times Y,Z)$,
and now~(\ref{hlp}) shows that $\Phi$ is open onto its image
(and hence a topological embedding).\\[2.4mm]
To verify the claim,
let $\pi_1\colon X\times Y\to X$ and
$\pi_2\colon X\times Y\to Y$ be the projection onto the first and second component, respectively.
If $M\sub X\times Y$ is compact and $V\sub X\times Y$ an open subset
such that $K\sub V$, then $M\sub M_1\times M_2$,
where $M_1:=\pi_1(M)$ and $M_2:=\pi_2(M)$ are compact.
For each $(x,y)\in M$,
there is an open neighbourhhod $U_{x,y}\sub X$ of $x$ and
and open neighbourhood $V_{x,y}\sub Y$ of~$y$ such that
$U_{x,y}\times V_{x,y}\sub V$.
Because $M_1$ and $M_2$ are locally compact, there exist
compact neighbourhoods $K_{x,y}\sub M_1$ of $x$
and $L_{x,y}\sub M_2$ of~$y$
such that $K_{x,y}\sub U_{x,y}$ and $L_{x,y}\sub V_{x,y}$.
Then $K_{x,y}\times L_{x,y}
\sub U_{x,y}\times V_{x,y}\sub V$.
By compactness, $K$ is covered by finitely many of the sets
$K_{x,y}\times L_{x,y}$. Thus all hypotheses
of Lemma~\ref{fewercp} are indeed satisfied.\\[2.4mm]
Now assume that $Y$ is locally compact.
Let $\eta\in C(X,C(Y,Z))$.
Because the evaluation map $\ve\colon C(Y,Z)\times Y\to Z$
is continuous by Lemma~\ref{laeval},
the map $\gamma:=\eta^\wedge:= \ve\circ (\eta\times \id_Y)\colon X\times Y\to Z$,
$(x,y)\mto \eta(x)(y)$ is continuous.
Since $\Phi(\gamma)=(\eta^\wedge)^\vee=\eta$,
we see that $\Phi$ is surjective and hence (being an embedding) a homeomorphism.\\[2.4mm]
Finally, assume that $X\times Y$ is a k-space.
Again, we only need to show that $\Phi$ is surjective.
Let $\eta\in C(X,C(Y,Z))$
and define $\gamma:=\eta^\wedge\colon X\times Y\to Z$, $(x,y)\mto \eta(x)(y)$.
If we can show that~$\gamma$ is continuous, then $\Phi(\gamma)=\eta$
(as required).
Because $X\times Y$ is a k-space, it suffices to show
that $\gamma|_K$ is continuous for each compact subset
$K\sub X\times Y$.
This will follow if $\gamma|_{K_1\times K_2}$
is continuous for all compact subsets $K_1\sub X$ and $K_2\sub Y$
(given $K$ as before, $K$ is contained in the compact set $K_1\times K_2$
with $K_1 :=\pi_1(K)$, $K_2:= \pi_2(K)$).
We now use that the restriction map
$\rho\colon C(Y,Z)\to C(K_2,Z)$
is continuous (see Remark~\ref{resop})
and hence also the map
\[
\zeta := \rho\circ \eta|_{K_1} \colon K_1\to C(K_2,Z).
\]
Since $K_2$ is compact,
$\zeta^\wedge\colon K_1\times K_2\to Z$ is continuous
(by the preceding part of the proof).
Because $\gamma|_{K_1\times K_2}=\zeta^\wedge$,
the continuity of $\gamma$
follows.
\end{proof}
\begin{rem}
Given Hausdorff topological spaces $X$, $Y$, $Z$ and a map
$\eta\colon X\to C(Y,Z)$, define
\[
\eta^\wedge\colon X\times Y\to Z,\quad (x,y)\mto \eta(x)(y).
\]
The preceding proposition entails:
If $\eta^\wedge$ is continuous, then $\eta=(\eta^\wedge)^\vee$
is continuous.
If $Y$ is locally compact or $X\times Y$ is a k-space,
then $\eta=(\eta^\wedge)^\vee$ is continuous if and only if
$\eta^\wedge$ is continuous.
\end{rem}
\begin{la}\label{tocon}
Let $X$ and $Y$ be Hausdorff topological spaces.
For $y\in Y$, let $c_y\colon X\to Y$ be the constant map $x\mto y$.
Then the map
\[
c\colon Y\to C(X,Y),\quad y\mto c_y
\]
is continuous
and $($if $X\not=\emptyset)$
in fact a topological embedding.
\end{la}
\begin{proof}
If $K\sub X$ is compact and $U\sub Y$ an open set,
then $c^{-1}(\lfloor K, U\rfloor)=U$ (if $K\not=\emptyset$)
and $c^{-1}(\lfloor K, U\rfloor)=Y$ (if $K=\emptyset$).
In either case, the preimage is open, entailing that $c$ is continuous.
If $X$ is not empty, we pick $x\in X$.
Then the point evaluation $\ev_x\colon C(X,Y)\to Y$, $\gamma\mto \gamma(x)$
is continuous and $\ev_x\circ c=\id_Y$,
entailing that $c$ is a topological embedding.
\end{proof}
By the next lemma, so-called pushforwards
are continuous.
\begin{la}
Let $X$, $Y$ and $Z$ be Hausdorff topological spaces
and\linebreak
$f\colon X\times Y\to Z$ be continuous.
Then also the following map is continuous:
\[
f_*\colon C(X,Y)\to C(X,Z),\quad \gamma\mto f\circ (\id_X,\gamma).
\]
\end{la}
Thus $f_*(\gamma)(x)=f(x,\gamma(x))$.\\[2.4mm]
\begin{proof}
Identifying $C(X,X)\times C(X,Y)$ with $C(X,X\times Y)$ as in
Lemm~\ref{cotopprod},
we can write
$f_*(\gamma)=C(X,f)(\id_X,\gamma)$.
Since $C(X,f)$ is continuous by Lemma~\ref{covsuppo},
the continuity of $f_*$ follows.
\end{proof}
\noindent
We also have two versions with parameters.
\begin{la}\label{pushpar}
Let $X$, $Y$, $Z$ and $P$ be Hausdorff topological spaces
and\linebreak
$f\colon P\times X\times Y\to Z$ be a continuous map.
For $p\in P$, abbreviate $f_p  :=f(p,\sbull)\colon X\times Y\to Z$.
Then also the following map is continuous:
\[
P\times C(X,Y)\to C(X,Z),\quad (p,\gamma)\mto f_p\circ (\id_X,\gamma).
\]
\end{la}
\begin{proof}
Let $c\colon P\to C(X,P)$, $p\mto c_p$ be the continuous map discussed in
Lemma~\ref{tocon}.
The map $C(X,f)\colon C(X,P\times X\times Y)\to C(X,Z)$ is continuous
(by Lemma~\ref{covsuppo}).
Identifying $C(X,P)\times C(X,X)\times C(X,Y)$ with\linebreak
$C(X,P\times X\times Y)$,
we have $f_p\circ (\id_X,\gamma)=C(X,f)(c_p,\id_X,\gamma)$,
which is continuous in $(p,\gamma)$.
\end{proof}
\begin{la}\label{parcov}
Let $X$, $Y$, $Z$ and $P$ be Hausdorff topological spaces
and\linebreak
$f\colon P\times Y\to Z$ be a continuous map.
For $p\in P$, set $f_p:=f(p,\sbull)\colon Y\to Z$.
Then also the following map is continuous:
\begin{equation}\label{phi}
\psi\colon P\times C(X,Y)\to C(X,Z),\quad (p,\gamma)\mto f_p\circ \gamma.
\end{equation}
\end{la}
\begin{proof}
Since $g\colon P\times X\times Y\to Z$, $g(p,x,y):=f(p,y)$
is continuous, so is
$\psi\colon P\times C(X,Y)\to C(X,Z)$, $(p,\gamma)\mto g_p\circ (\id_X,\gamma)$,
by Lemma~\ref{pushpar}.
\end{proof}
\begin{la}\label{sammelsu}
If $X$ is a Hausdorff topological space and $G$ a topological group,
then the following holds:
\begin{itemize}
\item[{\rm(a)}]
$C(X,G)$ is a topological group with respect to the pointwise
group\linebreak
operations.
\item[{\rm(b)}]
Let $\cL$ be a set of compact subsets of $X$ such that each $K\in \cK(X)$
is contained in some $L\in \cL$.
Also, let $\cU$ be a basis of open identity neighbourhoods in~$G$.
For $\gamma\in C(X,G)$,
the sets $\gamma\cdot \lfloor K, U\rfloor$
then form a basis of open neighbourhood of $\gamma$ in $C(X,G)$
$($for $K\in \cL$ and $U\in \cU$ in a basis $\cU)$,
and so do the sets
$\lfloor K, U\rfloor\cdot \gamma$.
The compact-open topology therefore coincides with the topology of
uniform convergence on compact sets,
both with respect to the left and also the right uniformity on~$G$.
\item[{\rm(c)}]
If $X$ is hemicompact and $G$ is metrizable, then
$C(X,G)$ is metrizable.
\item[{\rm(d)}]
If $X$ is a k-space and $G$ is complete,
then $C(X,G)$ is complete.
\item[{\rm(e)}]
If $X$ is a k-space and $G$ is sequentially complete,
then $C(X,G)$ is\linebreak
sequentially complete.
\item[{\rm(f)}]
If $E$ is a topological vector space over
a topological field $\K$, then the pointwise operations make
$C(X,E)$ a topological $\K$-vector space.
Moreover, $C(X,\K)$ is a topological $\K$-algebra and
$C(X,E)$ is a topological $C(X,\K)$-module.
\item[{\rm(g)}]
If $(\K,|.|)$ is an ultrametric field and
$E$ a locally convex topological $\K$-vector space,
then also
$C(X,E)$ is locally convex.
\item[{\rm(h)}]
If $(\K,|.|)$ is an ultrametric field,
$E$ an ultrametric normed space
and $X$ is compact, then also
$C(X,E)$ admits an ultrametric norm defining its topology.
\end{itemize}
\end{la}
\begin{proof}
(a) Since $G$ is a topological group, the group multiplication\linebreak
$\mu\colon G\times G\to G$, $(x,y)\mto xy$ and the inversion map $\iota\colon G\to G$, $x\mto x^{-1}$
are continuous. Identifying $C(X,G)\times C(X,G)$ with $C(X,G\times G)$ (as in Lemma~\ref{cotopprod}),
the group multiplication of $C(X,G)$ is the mapping
$C(X,\mu)\colon$ $C(X,G\times G)\to C(X,G)$,
which is continuous by Lemma~\ref{covsuppo}.
The group inversion is the map $C(X,\iota)\colon C(X,G)\to C(X,G)$
and hence continuous as well. Thus $C(X,G)$ is a topological group.

(b) If $K_1,\ldots, K_n\in \cK(X)$ and $U_1,\ldots, U_n\sub G$ are open identity neighbourhoods,
we find $U\in \cU$ such that $U\sub U_1\cap\cdots\cap U_n$,
and $K\in \cL$ such that $K_1\cup\cdots\cup K_n\sub K$.
Then $\bigcap_{j=1}^n\lfloor K_j,U_j\rfloor\supseteq \lfloor K,U\rfloor$.
Therefore the sets $\lfloor K,U\rfloor$ with $K\in \cL$ and $U\in \cU$ form a basis
of open identity neighbourhoods for $C(X,G)$.
Since left and right translations in the topological group $C(X,G)$ are homeomorphisms,
the remainder of (b) follows.

(c) Let $K_1\sub K_2\sub \cdots$ be an ascending sequence of compact subsets of~$X$, with union $X$,
such that each compact subset of $X$ is contained in some $K_n$.
Also, let $U_1\supseteq U_2\supseteq\cdots$ be a descending sequence of open identity neighbourhoods in $G$
which gives a basis for the filter of identity neighbourhoods.
By~(b), the sets $\lfloor K_n, U_n\rfloor$, for $n\in \N$, provide a countable basis
of identity neighbourhoods in $C(X,G)$.
As a consequence, the topological group
$C(X,G)$ is metrizable \cite[Theorem 8.3]{HaR}.

(d) Let $(\gamma_a)_{a\in A}$ be a Cauchy net in $C(X,G)$.
For each $x\in X$, the point evaluation $\ev_x\colon C(X,G)\to G$ is a continuous homomorphism.
Hence $(\gamma_a(x))_{a\in A}$ is a Cauchy net in~$G$
and hence convergent to some element $\gamma(x)\in G$
(as  $G$ is a assumed complete).
For each compact set $K\sub X$, the restriction map $\rho_K\colon C(X,G)\to C(K,G)$ is a continuous
homomorphism (cf.\ Remark~\ref{resop}), whence $(\gamma_a|_K)_{a\in A}$
is a Cauchy net in $C(K,G)$.
We claim that $C(K,G)$ is complete. If this is true,
then $\gamma_a|_K\to \gamma_K$
for some continuous map $\gamma_K\in C(K,G)$.
Since $\gamma(x)=\gamma_K(x)$ for each $x\in K$, we see that $\gamma|_K=\gamma_K$
is continuous. Because $X$ is a k-space, this implies that~$\gamma$
is continuous. If $K\in \cK(X)$ and $U\sub G$ is an identity neighbourhood,
then $(\gamma_K)^{-1}\gamma_a|_K\in \lfloor K,U\rfloor$ inside $C(K,G)$
for sufficiently large~$a$ (see (b))
and hence $\gamma^{-1}\gamma_a\in \lfloor K,U\rfloor$ inside $C(X,G)$,
showing that $\gamma_a\to \gamma$.

To prove the claim, we may assume that $X=K$ is compact.
Let $U\sub G$ be an open identity neighbourhood
and $V\sub U$ be an open identity neighbourhood
with closure $\overline{V}\sub U$.
There is $a\in A$ such that $\gamma_b^{-1}\gamma_c\in \lfloor K, V\rfloor$
for all $b,c\geq a$ in $A$. Thus, for all $x\in K$, $\gamma_b(x)^{-1}\gamma_c(x)\in V$.
Passing to the limit in $c$, we obtain $\gamma_b(x)^{-1}\gamma(x)\in \overline{V}\sub U$
and thus
\begin{equation}\label{sireu}
\gamma(x)\in \gamma_b(x)U, \quad \mbox{for all $\,x\in K$ and $b\geq a$.}
\end{equation}
If $W$ is any identity neighbourhood in $G$,
we can choose $U$ from before so small that $U^{-1}UU\sub W$.
By continuity of $\gamma_a$, each $x\in K$ has a neighbourhood $L\sub K$
such that $\gamma_a(y)^{-1}\gamma_a(y)\in U$ for all $y\in L$.
Combining this with~(\ref{sireu}), we see that
\[
\gamma^{-1}(y)\gamma(x)\in U^{-1}\gamma_a(y)^{-1}\gamma_a(x)U\sub U^{-1}UU\sub W
\]
for all $y\in L$. Thus $\gamma$ is continuous at~$x$ and hence continuous.
Finally, we have $\gamma_b^{-1}\gamma\in \lfloor K,U\rfloor$ for all $b\geq a$,
by (\ref{sireu}). Hence $\gamma_b\to \gamma$ in $C(K,G)$.

(e) Proceed as in (d), replacing the Cauchy net with a Cauchy sequence.

(f) By (a), $C(X,E)$ is a topological group.
Let $\sigma\colon \K\times E\to E$, $(z,v)\mto zv$ be multiplication with scalars.
Then the map $\K\times C(X,E)\to C(X,E)$, $(z,\gamma)\mto \sigma(z,\sbull) \circ \gamma$
is continuous (by Lemma~\ref{parcov}).
As this is the multiplication by scalars in $C(X,E)$,
the latter is a topological vector space.
Let $\mu\colon \K\times \K\to\K$ be the multiplication
in the field $\K$. Identifying $C(X,\K)\times C(X,\K)$
with $C(X,\K\times\K)$,
the algebra multiplication of $C(X,\K)$ is the map
$C(X,\mu)\colon C(X,\K\times\K)\to C(X,\K)$,
which is continuous by Lemma~\ref{covsuppo}.
Hence $C(X,\K)$ is a topological algebra.
Identifying $C(X,\K)\times C(X,E)$
with $C(X,\K\times E)$,
the $C(X,\K)$-module multiplication on $C(X,E)$ is the map
$C(X,\sigma)\colon C(X,\K\times E)\to C(X, E)$,
which is continuous by Lemma~\ref{covsuppo}.
Hence $C(X,E)$ is a topological $C(X,\K)$-module.

(g) Let $\bO:=\{z\in \K\colon |z|\leq 1\}$ be the valuation ring of $\K$.
By hypothesis, the open $\bO$-submodules $U\sub E$ form a basis
of $0$-neighbourhoods in $E$.
For each $K\in \cK(X)$,
the set $\lfloor K,U\rfloor$ then is an open $0$-neighbourhood
in $C(X,E)$ and an $\bO$-module.
Part\,(b) implies that these sets form a basis of $0$-neighbourhoods
in $C(X,E)$. Thus $C(X,E)$ is locally convex.

(h) If the ultrametric norm $\|.\|$ defines the topology of~$E$, then the ultrametric norm
$\|.\|_\infty\colon C(X,E)\to [0,\infty[$, $\|\gamma\|_\infty:=\sup\{\|\gamma(x)\|\colon x\in X\}$
defines the topology of $C(X,E)$,
as is clear from (b) and the observation that $\lfloor X, B^E_r(0)\rfloor =
B^{C(X,E)}_r(0)$ for each $r>0$.\vspace{-6mm}
\end{proof}
%
%
%
%
%
%
%
%
%
%
%
%
%
%
{\footnotesize

{\bf Helge Gl\"{o}ckner}, Universit\"{a}t Paderborn,
Institut f\"{u}r Mathematik,\\
Warburger Str.\ 100,
33098 Paderborn, Germany. \,Email: {\tt glockner@math.upb.de}}\vfill

\begin{thebibliography}{MM}\itemsep+.5pt\vspace{-.5mm}
%
%
\bibitem{Ama}
Amann, H., 
``Gew\"{o}hnliche Differentialgleichungen,''
de Gruyter, Berlin, 1983.
%
%
\bibitem{Alz}
Alzaareer, H.,
\emph{Lie groups of mappings on non-compact spaces and manifolds},
Ph.D.-thesis in preparation, University of Paderborn
%
%
\bibitem{AaS}
Alzaareer, H. and A. Schmeding,
\emph{Differentiable mappings on products with different degrees of differentiability in the two factors},
preprint, {\tt arXiv:1208.6510v1}
%
%
\bibitem{ArS}
Araujo, J.. and Schikhof, W.\,H.,
\emph{The Weierstrass-Stone approximation theorem for $p$-adic $C^n$-functions},
Ann.\ Math\. Blaise Pascal {\bf 1} (1994), 61--74.
%
%
\bibitem{BGN}
Bertram, W., H. Gl\"{o}ckner and K.-H. Neeb,
\emph{Differential calculus over general base fields and rings},
Expo.\ Math.\ \textbf{22} (2004),
213--282.
%
%
%
%
\bibitem{BTG}
Bourbaki, N., ``General Topology, Part 2,''
Addison-Wesley, Reading
and Hermann, Paris, 1966.
%
%
\bibitem{DSm}
De Smedt, S.
\emph{$p$-adic continuously differentiable functions of
several variables}, Collect.\ Math.\ \textbf{45}
(1994), 137--152.
%
%
%
%
\bibitem{Eng} Engelking, R.,
``General Topology,'' Heldermann Verlag,
1989.
%
%
%
%
\bibitem{FaK}
Fr\"{o}licher, A. and A. Kriegl, ``Linear Spaces and Differentiation Theory,''
John Wiley, Chichester, 1988.
%
%
\bibitem{GCX}
Gl\"{o}ckner, H., \emph{Lie group structures on quotient groups
and universal complexifications for infinite-dimensional Lie groups},
J. Funct.\ Anal.\ {\bf 194} (2002), 43--59.
%
%
\bibitem{CMP}
Gl\"{o}ckner, H.,
\emph{Comparison of some notions of $C^k$-maps in multi-variable non-archimedian analysis},
Bull.\ Belg.\ Math.\ Soc.\ Simon Stevin {\bf 14} (2007), 877--904. 
%
%
\bibitem{DIF}
Gl\"{o}ckner, H.,
\emph{Patched locally convex spaces, almost local mappings,
and the diffeomorphism groups of non-compact manifolds},
manuscript, 2002.
%
%
\bibitem{ZOO}
Gl\"{o}ckner, H.,
\emph{Lie groups over non-discrete topological fields},
preprint, arXiv:math.GR/0408008\,.
%
%
\bibitem{IM2}
Gl\"{o}ckner, H.,
\emph{Finite order differentiability properties, fixed points
and implicit functions over valued fields},
preprint, arXiv:math.FA/0511218\,.
%
%
%
%
\bibitem{HaR}
Hewitt, E. and K.\,A. Ross, 
``Abstract Harmonic Analysis I,''
Springer, ${}^2$1979.
%
%
\bibitem{Kel}
Kelley, L., ``General Topology,'' Springer, New York, 1975.
%
%
\bibitem{KaM} Kriegl, A. and P.\,W. Michor,
``The Convenient Setting of Global Ana\-lysis,''
AMS, Providence, 1997.
%
%
\bibitem{Mon}
Monna, A.\,F., ``Analyse non-archim\'{e}dienne,''
Springer-Verlag, 1970.
%
%
\bibitem{NDI}
Nagel, E.,
 ``Fractional Non-Archimedean Differentiability,''
 Doctoral Dissertation, Universit\"{a}t M\"unster, 2011
(see
{\tt http://miami.uni-muenster.de/servlets/ DerivateServlet/Derivate-5905/diss\_{}nagel\_{}enno.pdf})
%
%
\bibitem{Nag}
Nagel, E., \emph{Fractional non-Archimedean calculus in many variables},
to appear in
$p$-Adic Numbers, Ultrametric Analysis and Applications
(see {\tt
http://www.math.jussieu.fr/$\wt{\;}$nagel/publications/FracMany.pdf})
%
%
\bibitem{Sch}
Schikhof, W.\,H.,
``Ultrametric Calculus,''
Cambridge University Press, 1984.
%
%
\bibitem{PSn}
Schneider, P.,
``Nonarchimedean Functional Analysis,''
Springer, Berlin, 2002.
%
%
%
%
%
%
%
\bibitem{vRo}
van Rooij, A.\,C.\,M.,
``Non-Archimedean Functional Analysis,''
Marcel Dekker, New York, 1978.
%
%
\end{thebibliography}
\end{document}